\documentclass[11pt,twoside]{amsart}

\usepackage{color}
\usepackage{verbatim}
\usepackage{amsmath}
\usepackage{amsfonts,amsthm,amssymb}
  \usepackage{mathtools}

\newtheorem{thm}{Theorem}[section]  
\newtheorem{prop}{Proposition}[section]  
  
\newtheorem{lem}{Lemma}[section]  
\newtheorem{remark}{Remark}[section]

\def\Id{{\rm Id}\,}

\def\dN{\delta\!\wt N}

\def\d{\partial}
\def\ddj{\dot \Delta_j}

\newcommand{\with}{\quad\hbox{with}\quad}
\newcommand{\andf}{\quad\hbox{and}\quad}

\def\wh{\widehat}
\def\wt{\widetilde}

\newcommand\C{\mathbb{C}}
\newcommand\R{\mathbb{R}}
\renewcommand\S{\mathbb{S}}

\newcommand\Z{\mathbb{Z}}

\newcommand{\N}{\mathbb{N}}
\newcommand{\ep}{\varepsilon}

\DeclareMathOperator{\Supp}{Supp}

\renewcommand{\Re}{\,{\rm{Re}}\,}
\renewcommand{\Im}{\,{\rm{Im}}\,}
\def\sb{\smallbreak}
\def\mb{\medbreak}

\def\divx{\, \hbox{\rm div}_x\,  }

\renewcommand{\div}{\mbox{\rm div}\;\!}

\def\cA{{\mathcal A}}

\def\cC{{\mathcal C}}
\def\cD{{\mathcal D}}

\def\cF{{\mathcal F}}

\def\cH{{\mathcal H}}

\def\cL{{\mathcal L}}
\def\cM{{\mathcal M}}
\def\cO{{\mathcal O}}
\def\cP{{\mathcal P}}

\def\cS{{\mathcal S}}
\def\cZ{{\mathcal Z}}

\newcommand{\Int}{\displaystyle \int}

\def\d{\partial}

\def\ep{\varepsilon}
\def\eps{\varepsilon}

\begin{document}
\title[Partially dissipative systems]{Partially dissipative  systems in the critical regularity setting, 
and strong relaxation limit}
\date\today

\subjclass{35Q35; 76N10}
\keywords{Hyperbolic systems, critical regularity, relaxation limit, partially dissipative}

\author[R. Danchin]{Rapha\"el Danchin}
\address[R. Danchin]{Univ Paris Est Creteil, CNRS, LAMA, F-94010 Creteil, France Univ Gustave Eiffel, LAMA, F-77447 Marne-la-Vallée, France}
\email{danchin@univ-paris12.fr}

 {\begin{center}
\parbox{14.5cm}
{\begin{abstract} 
 Many physical phenomena may be modelled by  first order hyperbolic equations with degenerate dissipative or diffusive terms. 
 This is the case for example in  gas dynamics, where the mass is conserved during the evolution, but the momentum balance includes a diffusion (viscosity) or damping (relaxation) term,  or, in numerical simulations, of  conservation laws by relaxation schemes.
  
	Such so-called \emph{partially dissipative systems} have been first pointed out 
	by S.K.~Godunov in a short note in Russian  in 1961. 
	Much later, in 1984,  S.~Kawashima highlighted in his PhD thesis  a simple criterion ensuring
	the existence of  global  strong solutions in the vicinity of a linearly stable constant state. 
	This criterion has  been revisited in a number of  research works. In particular, K. Beauchard and E. Zuazua  proposed in 2010 an explicit method for constructing  a Lyapunov functional allowing to refine Kawashima's results and  to establish global existence results in some 
	situations that were not covered before.  

	These notes originate essentially  from  the PhD thesis of 
T. Crin-Barat  that was initially motivated by an earlier observation of the author in a Chapter 
of the handbook coedited by Y. Giga  and A. Novotn\'y.  
 Our main aim is to adapt  the method of Beauchard and Zuazua to a class of symmetrizable quasilinear hyperbolic systems (containing the compressible Euler equations), in a critical regularity setting that allows to keep track of the dependence with respect to e.g. the relaxation parameter.
 Compared to Beauchard and Zuazua's work, we  exhibit a `damped mode'  that will have a key role
 in the construction of global  solutions with critical regularity,   in the proof of optimal time-decay 
estimates and, last but not least, in the study of  the strong relaxation  limit.
For simplicity, we here focus on a simple class of partially dissipative systems, 
but the  overall strategy is rather flexible,  and adaptable to much more involved situations. 
 \end{abstract}}\end{center}}
\maketitle


\section*{Introduction}

An important recent mathematical literature has been devoted to the study of first order systems of conservation laws. 
These systems that come into play in the description  of  a number 
phenomena  in mechanics, physics or engineering  typically read 
\begin{equation}\label{eq:CL}\d_t f^0(V) + \sum_{k=1}^d \d_{x_k} (f^k(V))=0\end{equation}
where   the vector-fields $f^k,$ $k=0,\cdots,d$ are defined on some open subset  $\cO$ of $\R^n,$
and the unknown $V$ depends on the time variable $t\in\R_+\triangleq [0,\infty)$ and on the space variable $x\in\R^d.$ 

Under rather general conditions,  for example  whenever \eqref{eq:CL} is Friedrichs-symmetrizable, 
it is well known that   for any $\bar V$ in $\cO$ and initial data $V_0:\R^d\to \cO$ such that
$V_0-\bar V$ belongs to some Sobolev space $H^s(\R^d)$ with $s>1+d/2,$
then \eqref{eq:CL} supplemented with initial data $V_0$ admits
a unique classical solution $V$ on some time interval $[-T,T],$
 satisfying $(V-\bar V)\in\cC_b([-T,T]; H^s(\R^d))$
 (the reader may find  the detailed statement and the proof  
in e.g. \cite[Chap. 10]{BS}). At the same time, for most systems of the above type, 
 smooth solutions (even small ones)  blow-up after finite time. 
\smallbreak
In many  physical systems however, friction or  diffusion  phenomena (through e.g. thermal conduction or viscosity) 
cannot be neglected. Typically, they act   on  some components of the unknown, while
other components remain unaffected.  An informative example is  gas dynamics
where  the  mass is conserved  (as well as the entropy  in the isentropic case).  
In order to have an  accurate description  corresponding to these situations, 
it is thus suitable to add in  \eqref{eq:CL}   zero (friction) or second (diffusion) order terms that
act on a part of the unknown  but, possibly, not on all components. 
The resulting  class of systems is named, depending on the authors and on the context, 
hyperbolic-parabolic, partially diffusive or partially dissipative. It has been extensively 
studied  since the pioneering work by S. Kawashima in his PhD thesis \cite{Ka84}. 
 One of the main issues is to find as weak as possible conditions ensuring the existence of global solutions 
 close to  constant states,  to describe their long time asymptotics and, where applicable, 
to study  the convergence to some limit system.
 \smallbreak
Rather than  writing out now the class of systems that enter in our study,  let us give a simple example from 
 multi-dimensional gas dynamics. 
 In he barotropic and isothermal case, the governing equations then read: 
\begin{equation}\label{eq:euler}
\left\{\begin{aligned}&\d_t\varrho +\divx(\varrho v)=0\quad &\hbox{in}\quad \R_+\times\R^d,\\
&\d_t(\varrho v)+\divx(\varrho v\otimes v) +\nabla_x P = A(\varrho,v)\quad &\hbox{in}\quad \R_+\times\R^d.
\end{aligned}\right.\end{equation}
Above, $\varrho=\varrho(t,x)\in\R_+$ stands for the density of the gas, and $v=v(t,x)\in\R^d,$ for the velocity. 
The pressure $P=P(\varrho)$ is a given function of the density. A  typical example is the \emph{isentropic
pressure law}   $P(\varrho)=a\varrho^\gamma$ with $a>0$ and $\gamma>1$.  The first equation corresponds to the mass conservation
and the second one, to the momentum balance. 
We assume that the fluid domain is the whole space which, somehow, means 
that boundary effects are neglected. This is a fundamental assumption for our analysis, that strongly relies on Fourier methods.  
\medbreak
Regarding $A,$ the usual assumptions   are: \begin{itemize}
\item[---] either $A$ is identically zero: then \eqref{eq:euler} is the \emph{barotropic compressible Euler equations} that 
 is known to be Friedrichs-symmetrizable (again, refer e.g. to \cite[Chap. 10]{BS}) and thus enters
in the class considered in \eqref{eq:CL};
\item[---]  or $A(\varrho,v)= \mathfrak f\, \varrho v$ for some $\mathfrak f>0$ (this is the so-called  \emph{damped barotropic compressible Euler equations}, also 
named \emph{Euler equations with relaxation parameter $\ep$} if $\mathfrak f=\ep^{-1}$);
\item[---]  or $A(\varrho,v)=  \divx(\mu(\varrho)(\nabla_x u+\! {}^t\nabla_x u)) + \nabla_x(\lambda(\varrho)\divx u)$ for 
some smooth functions $\lambda$ and $\mu$ satisfying $\mu>0$ and $\lambda+2\mu>0$ 
(then, \eqref{eq:euler} is the \emph{barotropic compressible Navier-Stokes equations}). 
\end{itemize}
 \smallbreak
It is by now well understood that in the first situation (neither viscosity nor damping),  smooth initial data
 generate a local-in-time solution that  is likely to blow up after  finite time 
 (see e.g. \cite{Alinhac,S86})  
whereas in the second and third situations, 
small and sufficiently smooth perturbations of a constant density state
\begin{equation}\label{eq:constantstate} (\bar\varrho,0)\with\bar\varrho>0\andf  P'(\bar\varrho)>0
\end{equation} 
 produce global strong solutions that are defined
for all positive times. 
\smallbreak
\medbreak
The good diffusive properties of the barotropic compressible Navier-Stokes equations in the whole space $\R^3$
(and, more generally, of  the full non-isothermal polytropic system)  have been first observed by A. Matsumura and T. Nishida 
at  the end of the 70ies. In  \cite{MN}, they   established the
global existence  of strong solutions for $H^3(\R^3)$  perturbations of   any  constant  state of type \eqref{eq:constantstate}
(see \cite{D1} for a version of this result in the broader setting of `critical Besov spaces').  
An important achievement in the study of general first  order partially dissipative symmetric hyperbolic
 systems having  both terms of order $0$ and $2$  has been made 
by S. Kawashima in 1984, in his PhD thesis \cite{Ka84}. There, he exhibited a rather simple sufficient condition 
that is nowadays called \emph{the (SK) (meaning Shizuta-Kawashima) condition} 
for  global existence of strong solutions  in the neighborhood of linearly stable
constant solutions.  In the case where there is only  a $0$-order  partially dissipative term, 
Condition (SK) exactly  says that for the linearized system,  \emph{the intersection between the kernel of the 
$0$-order term and the set of all eigenvectors of the symmetric first order term 
is reduced to $\{0\}.$} 

A  bit later, S.~Shizuta and S. Kawashima in \cite{SK}  observed
that Condition (SK) is  equivalent to the fact that,  in the Fourier space,  the real parts of all eigenvalues of 
the matrix of the linearized system about the reference  solution are strictly 
negative and also to the existence of a  \emph{compensating function}. That compensating function comes into play 
for working out a functional  that is equivalent to a Sobolev norm  of high order  
and allows to recover the optimal dissipative properties of the system. 
In the same paper, the authors pointed out that, if in addition of being in a Sobolev
space $H^s(\R^d)$ with large enough $s,$ the discrepancy of the initial data to the 
reference constant solution $\bar V$ belongs to some Lebesgue space $L^p(\R^d)$ with $p\in[1,2],$ 
then the global solution $V$ converges to $\bar V$ in $L^2(\R^d)$ with the same decay 
rate as for the heat equation, namely  $(1+t)^{-\frac d2(\frac1p-\frac12)},$ when 
$t$ goes to infinity.  Since then, more decay estimates have 
been proved under various assumptions in e.g. \cite{BHN,XX,XK3}. 
\smallbreak
A number of more accurate results have been obtained since then for specific systems. 
For instance, T. Sideris \emph{et al} \cite{STW} considered the 
three-dimensional compressible Euler equations with damping
and Y. Zeng \cite{Zeng} studied  a particular class of $4\times4$ nonlinear hyperbolic system with relaxation. 
General  partially ($0$-order) dissipative systems  have  been investigated by 
S. Kawashima and W.-A. Yong in \cite{KY1,KY2} and by  W.-A. Yong in \cite{Yong}, 
and adapted to second order partially diffusive operators by V. Giovangigli et al in \cite{GM,GY}. 
Recent works on general partially dissipative systems
in the  so-called critical functional framework (that will be recalled later in this text)
have been performed by J. Xu and S. Kawashima \cite{XK1,XK2,XK3}. 

It has also been observed by several authors that  Condition (SK) is not necessary for the existence of global 
strong solutions. For instance, in \cite{QW}, P. Qu and Y. Wang 
established a global existence result in the case where exactly one eigenvector violates Condition (SK).
In this respect, one  can also mention  the paper by R. Bianchini and R. Natalini \cite{BN}
that uses nonresonant bilinear forms, and 
 the recent work \cite{BCBT}  dedicated 
to the mathematical study of a model of mixture of compressible fluids. 
 \smallbreak
The strength of Shizuta and Kawashima's approach is that it does not require 
to compute explicitly the Green function of the linearized system under consideration. 
Although doing this calculation   for the damped barotropic Euler equations    presented above is not an issue,
computing the Green kernel associated to the corresponding linearized system in the   
nonisothermal case  is already more involved, and it soon becomes impossible for more cumbersome systems
 (like e.g. systems related to  the description of plasma or radiative phenomena, see e.g  \cite{DD}).
As said before, having  a `compensating function'  at hand allows to construct an energy functional  that encodes 
the dissipative properties of the system.  In  Shizuta and Kawashima's work however,  this functional 
is not so explicit, that makes difficult, if not impossible, to track 
the dependency of the solution with respect to the parameters of the system, when applicable. 
Another   limitation  is that it only provides estimates on the whole solution, without supplying more 
accurate informations on the part of the solution which is expected to 
experience a better dissipation. 

In \cite{BZ},  K. Beauchard and E. Zuazua  took advantage of techniques 
that originate from  Kalman  control theory for linear ODEs so as to construct explicit 
Lyapunov functionals   for general partially dissipative systems of order $1.$  
They also pointed out the  connection between Condition (SK) and the Kalman criterion for observability 
in the theory of linear ODEs  (this was also noticed  by D. Serre in  his unpublished lecture notes
\cite{Serre}).
To some extent, Beauchard and Zuazua's approach 
 may be interpreted in the broader framework of hypo-ellipticity as presented by  L. H\"ormander in \cite{Ho}  or, much more recently, by C. Villani   in \cite{Villani}. To keep these notes as elementary and short  as possible, we refrain from 
 looking deeper into this direction, though.

 Although  it is not mentioned in the  construction of a  Lyapunov functional,
 Beauchard and Zuazua's  approach  provides for free compensating functions. Furthermore, 
 the  construction is elementary  (it suffices to compute at most $n$ powers of matrices) and  
easily  localizable in the Fourier space. Hence,   at the linear level, keeping  track   
of the different behavior of the low and of the high frequencies of the solution is obvious. 
Their method further allows  to handle some systems that do not satisfy Condition (SK)
(but we shall not investigate this interesting point is these notes).  
\bigbreak
The present lecture notes aim at familiarizing the reader with the Beauchard-Zuazua approach  and recent 
updates that originate  from the thesis of T. Crin-Barat
and were published in \cite{CBD1,CBD2,CBD3}.    
As our aim  is not to provide the reader with an exhaustive theory of partially dissipative 
  hyperbolic systems but rather to present 
   a clear   road map allowing him to tackle efficiently the study of systems of this type, 
    we shall focus on the following  `academic' class of partially dissipative hyperbolic systems: 
\begin{equation}
\partial_t V + \sum_{k=1}^dA^k(V){\partial_k V}=\eps^{-1} H(V). \label{GEQSYMeps}
\end{equation}
Above,  the (smooth) functions 
$A^k$  ($k=1,\cdots,d$)  and $H$ are defined  on some   open subset   $\mathcal{O}$ of $ \mathbb{R}^n,$
and  have range in the set of  $n\times n$ real symmetric matrices,  and in $\R^n,$ respectively. 
The  unknown $V=V(t,x)$ depends on the time variable 
$t\in \mathbb{R}_+$ and  on the space variable $x\in\mathbb{R}^d$  ($d\geq1$).
We fix a constant solution $\bar V\in\mathcal{O}$  of \eqref{GEQSYMeps} (hence $H(\bar V)=0$). 
The system is supplemented with  initial data $V_0\in\mathcal O$ at time $t=0,$
that are sufficiently close to   $\bar V.$  Finally, the relaxation parameter $\eps$ is 
a given positive parameter that, except in Section \ref{s:relax}, is  taken equal to $1.$  

A basic example of a physical system in the above class is 
 the compressible Euler equations with isentropic
pressure law $P(\varrho)= a\varrho^\gamma,$  if rewritten in terms of the (renormalized) sound speed 
$$c\triangleq \frac{(\gamma A)^\frac{1}{2}}{\wt\gamma}(\varrho)^{\wt\gamma}\with\wt\gamma\triangleq \frac{\gamma-1}2\cdotp$$ 
Indeed, the pair $(c,v)$ then satisfies:
\begin{equation}\label{eq:eulerc}\left\{\begin{aligned}&\d_tc +v\cdot\nabla c+\wt\gamma c\div v=0\quad &\hbox{in}\quad \R_+\times\R^d,\\
&\d_tv+v\cdot\nabla v +\wt\gamma c\nabla c +\ep^{-1}v=0 &\hbox{in}\quad \R_+\times\R^d.
\end{aligned}\right.\end{equation}
 Under the so-called Condition (SK) (presented in  the next section) that is satisfied in particular by \eqref{eq:eulerc}, we shall  
 prove the existence of global strong solutions with `critical regularity'   for \eqref {GEQSYMeps} in the neighborhood
of any constant solution $\bar V$ (see Theorems \ref{ThmGlobal} and \ref{Thmd2}).
Then, we shall   obtain the strong convergence to $\bar V$  in the long time asymptotics  with  
explicit  decay rates (Theorem \ref{Thm:decay}). 
In Section \ref{s:relax}, we shall  investigate the strong relaxation limit, that is 
the  convergence of the solutions  of  \eqref {GEQSYMeps} to some limit system.  
Let us shortly explain what we mean in  the simple case of the compressible Euler equations. 
Making the following `diffusive' rescaling: 
$$(\varrho, v)(t,x)= (\wt\varrho,\ep\wt v)(\ep  t,x),$$
we see that the pair $(\wt\varrho,\wt v)$ satisfies:
$$\left\{\begin{aligned}&\d_\tau\wt\varrho +\div(\wt\varrho \wt v)=0\quad &\hbox{in}\quad \R_+\times\R^d,\\
&\ep^2\d_\tau(\wt \varrho  \wt v)+\ep\div(\wt \varrho \wt v\otimes \wt v) +\nabla(P(\wt\varrho)) +\wt\varrho\wt v =0\quad &\hbox{in}\quad \R_+\times\R^d.
\end{aligned}\right.$$
Hence, formally, if $\wt\varrho$ and $\wt v$ tend to some functions $N$ and $w,$ then the  second equation above yields
$$ \nabla(P(N))+ N w=0$$
which, plugged in the  mass conservation equation leads to  the so-called porous media equation:
\begin{equation}\label{eq:PM0}
\d_\tau N-\Delta(P(N))=0.\end{equation}
The rigorous justification of the convergence of the density to a solution of \eqref{eq:PM0}
  has been first carried out  by  S. Junca and M. Rascle \cite{JR} 
in the  one-dimensional case where specific techniques may be used.  In the multi-dimensional case, 
the weak convergence and  the strong convergence on bounded subsets of $\R^d$   have been proved by 
J.-F. Coulombel and C. Lin in \cite{LC}, and by Z. Wang and J. Xu in \cite{XW}.  
Results in the same spirit for a class of partially dissipative hyperbolic systems have been 
obtained by Y.-J. Peng and V. Wasiolek in \cite{PW}. 
The approach that is proposed in the present lecture notes allows to get the \emph{strong 
convergence in the whole space  with explicit convergence rates for suitable norms} when the relaxation parameter tends to zero
not only for  the Euler equations, but also for a  class
of partially hyperbolic systems (see Theorem \ref{Thm3}).  
\medbreak
It should be noted that,  at the linear level,  the method that has been 
originally proposed by K. Beauchard and E. Zuazua in \cite{BZ} 
works exactly the same for partial differential operators of any order (and, more generally, for homogeneous Fourier multipliers) provided one of them 
is skew-symmetric and the other one, nonnegative.  
We will enrich this method by exhibiting   a `damped mode' for low 
frequencies, first introduced in \cite{CBD1} and \cite{CBD2} to the best of our knowledge.  
 This the key  to an optimal  treatment of the low frequencies of the solution  in a critical framework.
With almost no additional effort, assuming a bit more integrability on the initial data (expressed in terms of negative Besov spaces like
in  the work \cite{XK3} by J. Xu and S. Kawashima), and 
arguing essentially as in the paper by Y. Guo and Y. Wang \cite{GW}, 
we  will derive optimal time decay estimates, pointing out better decay
for the high frequencies of the solution and for the damped mode. 
It turns out that adopting a critical  approach with different levels of regularity 
for low and high frequencies also  allows to keep track of  the relaxation parameter $\eps$ 
just by suitable space/time rescaling. This  substantially simplifies  the study of the strong relaxation limit. 
Here again, having a damped mode at hand plays an essential role.
\smallbreak
Except for our linear analysis, we here concentrate on  first order hyperbolic symmetric systems with a partial dissipation
term of order $0.$  The class that is considered contains the isentropic 
Euler equations with relaxation.   We  expect the whole strategy  modified accordingly 
to be adaptable  to hyperbolic-parabolic systems,  to  operators of any  order 
and to more complex situations where the partially dissipative terms
have mixed orders (see recent examples in \cite{DD} and \cite{BCBT}).
It would also be of interest to study to what extent it may be adapted 
  to   situations where   pseudo-differential operators depending on the space variable come into play. 
  Since we used mostly Fourier analysis in our investigations, most of our results can be adapted 
   to periodic boundary conditions in one or several directions, leading to the same  statements 
  in the first three sections (the strong relaxation limit studied in Section \ref{s:relax} may be different 
  since the rescaling we used there changes the size of the periodic box). 
    Handling   `physical' boundaries requires completely different tools, and we have no opinion on whether
  similar results are true or not.  
 \medbreak
The rest of  these notes unfolds as follows. 
In the next section, we  present Beauchard and Zuazua's approach for linear partially dissipative hyperbolic
systems with operators of any orders. 
This enables us to deduce quite easily global-in-time  a priori estimates in
`hybrid' Besov spaces with different  regularity exponents   for low and high frequencies. We also exhibit a damped mode, the low frequencies of which satisfy better decay estimates
and point out that, under additional structure conditions on the system, 
it is possible to use without much effort an  $L^p$ functional framework for the low frequencies. 
The following sections  focus on   the nonlinear system \eqref{GEQSYMeps}. 
In Section \ref{s:2}, we prove global-in-time results while time decay estimates are established in Section
\ref{s:decay}. In Section \ref{s:relax}, we prove strong convergence results when the relaxation parameter $\eps$ 
tends to $0$ for partially dissipative systems having the same structure as the isentropic compressible
Euler equations  with damping. A few technical results are recalled or proved in Appendix.

  \bigbreak\noindent{\bf Acknowledgements.}
  The author  is indebted to the anonymous referees for their 
 relevant remarks and  suggestions that contributed to improve 
substantially the organization of these notes. The author has been partially supported by the ANR project INFAMIE (ANR-15-CE40-0011).

 
 \section{The linear analysis}

To better understand the difference between the  three model situations corresponding to System \eqref{eq:euler}, 
 having first  a look  at the linearized equations about  $(\bar\varrho,0)$ is very informative.  
After suitable renormalization, the system to be considered reads:
\begin{equation}\label{eq:lineareuler}
\left\{\begin{aligned}&\d_ta +\div u=0\quad &\hbox{in}\quad \R_+\times\R^d,\\
&\d_tu+\nabla a +\kappa(-\Delta)^{\frac\beta2}u =0\quad &\hbox{in}\quad \R_+\times\R^d.
\end{aligned}\right.\end{equation}
The above  cases correspond to  $(\kappa,\beta)=(0,0),$ 
$(\kappa,\beta)=(\mathfrak f,0)$ with $\mathfrak f>0$ or $(\kappa,\beta)=(\mu(\bar\varrho),2)$ (in the special situation $\lambda(\bar\varrho)+\mu(\bar\varrho)=0$
the general case being similar), respectively. 

If  $\kappa=0$ then  System \eqref{eq:lineareuler} is purely first order hyperbolic and 
no  diffusion or dissipative  phenomenon is expected whatsoever since 
all Sobolev norms are constant in time.  In the multi-dimensional case, 
 \emph{dispersive} phenomena of wave equation type do exist, 
 but they concern only the density and the potential part of the velocity (they will not be discussed here). 

Let us focus on the case $\kappa>0$ and $\beta\not=1$ (not necessarily equal to $0$ or $2$). 
After suitable rescaling, one can then suppose that   $\kappa=1.$ 
In the Fourier variable $\xi$  corresponding to the physical variable  $x$, the  above system \eqref{eq:lineareuler} rewrites
\begin{equation}\label{eq:au} 
\frac d{dt}\begin{pmatrix}\wh a\\\wh u\end{pmatrix} +\begin{pmatrix} 0&i\xi\\ i{}^t\!\xi&|\xi|^\beta\end{pmatrix}
\begin{pmatrix}\wh a\\\wh u\end{pmatrix}=\begin{pmatrix}0\\0\end{pmatrix},\qquad\xi\in\R^d.\end{equation}
In order to have some insight on the long-time behavior of the solution $(a,u),$ let us 
look  at the eigenvalues of the $(d+1)\times(d+1)$  matrix of System \eqref{eq:au}. 
The eigenvalue $|\xi|^\beta$  appears with multiplicity $d-1$ (this corresponds to the `incompressible' part of 
the velocity field). The remaining  two eigenvalues $\lambda^\pm(\xi)$  capture
the coupling between $a$ and the `compressible' part of $u,$   and may be computed by  considering
the following  $2\times2$ reduced system satisfied by $a$ and $\upsilon\triangleq (-\Delta)^{-1/2}\div v,$ namely,
if $\kappa=1,$
$$\left\{\begin{aligned}&\d_ta +(-\Delta)^{1/2}\upsilon=0\quad &\hbox{in}\quad \R_+\times\R^d,\\
&\d_t\upsilon-(-\Delta)^{1/2}a+(-\Delta)^{\frac\beta2}\upsilon =0\quad &\hbox{in}\quad \R_+\times\R^d.
\end{aligned}\right.$$
The corresponding matrix in the Fourier space reads
$\begin{pmatrix} 0&|\xi|\\ -|\xi|&|\xi|^\beta\end{pmatrix}\cdotp$
\medbreak
\noindent Two different situations occur depending on whether $\beta$ is smaller or greater than $1$:
\begin{itemize}
\item The `dissipative' situation $\beta<1$: 
$$\begin{array}{ll}
\lambda^\pm(\xi)=\frac{|\xi|^\beta}2\Bigl(1\pm\sqrt{1-4|\xi|^{2(1-\beta)}}\Bigr)\quad&\hbox{if}\quad  |\xi|^{1-\beta}<1/2,\\[1ex]
\lambda^\pm(\xi)=\frac{|\xi|^\beta}2\Bigl(1\pm i\sqrt{4|\xi|^{2(1-\beta)}-1}\Bigr)\quad&\hbox{if}\quad  |\xi|^{1-\beta}>1/2.
\end{array}$$
Observe that for $\xi\to0,$ we have $ \lambda^+(\xi)\sim |\xi|^\beta$ (parabolic behavior similar to that 
of $(-\Delta)^{\beta/2}$) while $\lambda^-(\xi)\sim |\xi|^{2-\beta}$ (parabolic behavior  of type $(-\Delta)^{1-\beta/2}$).
\item The `diffusive' situation $\beta>1$: 
$$\begin{array}{ll}
\lambda^\pm(\xi)=\frac{|\xi|^\beta}2\Bigl(1\pm i\sqrt{4|\xi|^{2(1-\beta)}-1}\Bigr)\quad&\hbox{if}\quad  |\xi|^{\beta-1}<1/2,\\[1ex]
\lambda^\pm(\xi)=\frac{|\xi|^\beta}2\Bigl(1\pm \sqrt{1-4|\xi|^{2(1-\beta)}}\Bigr)\quad&\hbox{if}\quad  |\xi|^{\beta-1}>1/2.
\end{array}$$
For $\xi\to\infty,$ we have $ \lambda^+(\xi)\sim |\xi|^\beta$ (parabolic behavior like for $(-\Delta)^{\beta/2}$) while $\lambda^-(\xi)\sim |\xi|^{2-\beta}$ (parabolic behavior similar to that of~$(-\Delta)^{1-\beta/2}$).
\end{itemize}

At the linear level, the damped Euler equations and the compressible Navier-Stokes equations correspond to the 
dissipative and diffusive situations, respectively. 
We observe that, in the two cases, the whole solution decays to $0$ with a decay rate that depends on $|\xi|$ 
although there is no damping term in the linearized mass equation. 
Note however that, depending on whether $\beta<1$ or $\beta>1,$
 the behavior of the low and high frequencies of the solution  is exchanged. 
\medbreak
Let us revert to our model system \eqref{GEQSYMeps}   with $\varepsilon=1$ for simplicity,  namely 
\begin{equation}
\partial_t V + \sum_{k=1}^dA^k(V){\partial_k V}=H(V). \label{GEQSYM}
\end{equation}
Let us  fix a constant solution $\bar V$ of \eqref{GEQSYM} (that is, 
$\bar V\in\mathcal{O}$ satisfies $H(\bar V)=0$) and  make the following 
structure assumptions on the system: 
\begin{enumerate}
\item[{(\bf H1)}]  For all $V\in\cO,$ the matrices $A^k(V)$ are real symmetric;\smallbreak
\item[{(\bf H2)}]  The spectrum of  $DH(\bar V)$ is included in the set $\{z\in\C\: : \: \Re z\leq 0\}.$ 
\end{enumerate}
\smallbreak
In  the case $H\equiv0$ (no dissipation at all)   smooth solutions, even small ones,  may  blow up after finite
time. At the exact opposite, if  the spectrum of $DH(\bar V)$ is included in the set $\{z\in\C\: : \: \Re z<0\}$
then it is not difficult to show that small perturbations of $\bar V$ in the Sobolev space $H^s$ with $s>1+d/2$  generate global strong solutions that tend exponentially 
fast to $\bar V$ when time goes to infinity.
 We here address the intermediate situation where some eigenvalues of $DH(\bar V)$  vanish. 
 For expository purpose, we assume that $H$ is linear and has the block structure:
 \begin{equation}\label{eq:H}
  H(V)=\begin{pmatrix} 0\\ -L_2(V_2-\bar V)\end{pmatrix} \with V=\begin{pmatrix} V_1\\ V_2\end{pmatrix}
  \end{equation}
  where $V_1\in\R^{n_1},$ $V_2\in\R^{n_2}$ (with $n_1+n_2=n$)  and $L_2:\R^{n_2}\to\R^{n_2}$ is linear invertible and 
 such that $L_2+{}^t\!L_2$ is definite positive.
Additional structure assumptions  on $L_2$ and on the matrices $A^k$
will be specified later on.

   \subsection{Reduction of the problem}
 
Denoting $Z\triangleq V-\bar V$ and $LZ\triangleq -H(\bar V +Z)$ (with $H$ as in \eqref{eq:H}), 
 the system for $Z$ reads
\begin{equation}\label{eq:Z}\d_tZ +\sum_{k=1}^d A^k(\bar V+Z)\d_k Z + LZ=0,\end{equation}
and the corresponding linearized system is thus 
   \begin{equation}\label{eq:Zlinear}\d_tZ+\sum_{k=1}^d\bar A^k \d_kZ + LZ=F \with  \bar A^k\triangleq  A^k(\bar V)\ \hbox{ for }\ k=1,\cdots,d.\end{equation}

In the Fourier space, the above system recasts in
 $$\d_t\wh Z+i\sum_{k=1}^d\bar A^k \xi_k\wh Z + L\wh Z=\wh F.$$
The symmetry of the matrices $\bar A^j$ ensures that for all $\xi\in\R^d,$
the matrix 
\begin{equation}\label{eq:Axi}
A(\xi)\triangleq i\sum_{k=1}^d\bar A^k \xi_k\end{equation}
 is skew Hermitian, while the symmetric part of $L$ is nonnegative. 
 Denoting by $A(D)$ (resp. $B(D)$) the Fourier multiplier  of symbol\footnote{Throughout the text, 
 we agree that $A(D)\triangleq \cF^{-1} A\cF, $ where $\cF$ stands for the Fourier transform
 with respect to the variable $x.$} 
  $A$ (resp. $L$), 
 System \eqref{eq:Zlinear} rewrites 
\begin{equation}\label{eq:general}\d_t Z+A(D)Z+B(D)Z=F.\end{equation}
The analysis we present below is valid in the more general situation  where:
\begin{itemize}
\item $A(D)$ is a homogeneous (matrix-valued)  Fourier multiplier of degree $\alpha$ that satisfies
\begin{equation}\label{eq:skew} \Re\bigl( A(\omega)\eta\cdot\eta\bigr) =0\quad\hbox{for all }\ \omega\in\S^{d-1}\andf
\eta\in\C^n,\end{equation}
where $\cdot$ designates the Hermitian scalar product in $\C^n,$
\item $B(D)$ is an homogeneous (matrix-valued) Fourier multiplier of degree $\beta,$ such that, for some positive real number $\kappa,$ \smallbreak
\begin{equation}\label{eq:positive} \Re\bigl( B(\omega)\eta\cdot\eta\bigr) \geq\kappa |B(\omega)\eta|^2\quad\hbox{for all }\ \omega\in\S^{d-1}\andf\eta\in\C^n.\end{equation}
\end{itemize}
As a first  example, if one considers the linearized damped compressible Euler equations
about $(\varrho,v)=(1,0)$  in the case $P'(1)=1,$   namely
\begin{equation}\label{eq:leuler}
\left\{\begin{aligned}&\d_ta +\div u=f\quad &\hbox{in}\quad \R_+\times\R^d,\\
&\d_tu +\nabla a +\mathfrak f \,u = g\quad &\hbox{in}\quad \R_+\times\R^d,
\end{aligned}\right.\end{equation}
then we have $n_1=1,$ $n_2=d,$ and  the Fourier multipliers $A$ and $B$ read: 
$$
A(\xi)=i\begin{pmatrix}0&\xi\\ {}^t\!\xi&0\end{pmatrix}\andf B(\xi)=\mathfrak f\begin{pmatrix}0&0\\
0&I_d\end{pmatrix}\cdotp$$
They are of order $1$ and $0,$ respectively. Clearly, \eqref{eq:skew} holds true, as well as \eqref{eq:positive} with $\kappa=\mathfrak f^{-1}.$
\medbreak
As a second   example,  consider the linearized  compressible Navier-Stokes  equations
about $(\varrho,v)=(1,0).$  Denoting   $\bar\mu\triangleq \mu(\bar\varrho)$ and $\bar\lambda\triangleq \lambda(\bar\varrho),$
they  read
\begin{equation}\label{eq:lns}
\left\{\begin{aligned}&\d_ta +\div u=f\quad &\hbox{in}\quad \R_+\times\R^d,\\
&\d_tu +\nabla a -\mu\Delta u -(\lambda+\mu)\div u = g\quad &\hbox{in}\quad \R_+\times\R^d.
\end{aligned}\right.\end{equation}
We still have $n_1=1,$ $n_2=d,$  but  the Fourier multipliers $A$ and $B$ now read: 
$$
A(\xi)=i\begin{pmatrix}0&\xi\\ {}^t\!\xi&0\end{pmatrix}\andf B(\xi)=\begin{pmatrix}0&0\\
0& \bar\mu |\xi|^2\,I_d +(\bar\lambda+\bar\mu)\xi\otimes \xi\end{pmatrix}\cdotp$$
They are of order $1$ and $2,$ respectively.
Both properties  \eqref{eq:skew} and  \eqref{eq:positive} hold true (with $\kappa$ depending on $\bar\lambda$ and $\bar\mu$)
provided $\bar\mu>0$ and $\bar\lambda+2\bar \mu>0.$ 
\medbreak
System \eqref{eq:general}  may be solved by means of Duhamel's formula:
\begin{equation}\label{eq:duhamel}Z(t)=T(t)Z(0)+\int_0^t T(t-\tau) F(\tau)\,d\tau,\end{equation}
where $(T(t))_{t\geq0}$ stands for the semi-group associated to operator $-(A+B)(D).$ 
\medbreak
The value of $T(t)$ may be computed by going into the Fourier space. Indeed,  denote by $\wh Z$
the Fourier transform of $Z$ with respect to $x,$ and by $\xi$ the corresponding Fourier variable. 
Then, in the case $F=0,$ System \eqref{eq:general} rewrites:
$$\d_t\wh Z + E(\xi) \wh Z=0\with E(\xi)\triangleq  A(\xi)+B(\xi).$$
Hence $\wh Z(t,\xi)=\exp(-E(\xi) t) \wh Z_0(\xi).$ In other words, we have $T(t)= \exp(-E(D)t).$ 
\medbreak
In what follows we set for all $\omega\in\S^{d-1}$  and $\rho>0,$
\begin{equation}\label{eq:def}
A(\xi)=\rho^{\alpha}A_\omega\quad\hbox{and}\quad
B(\xi)=\kappa^{-1}\rho^{\beta}B_\omega\quad\hbox{with }\ \xi=\rho\omega.
\end{equation}
With this notation, we have 
\begin{equation}
\wh Z(t,\xi)= \wh Z(0,\xi)\exp\biggl(-\frac{t\rho^\beta}\kappa
\bigl(\kappa\rho^{\alpha-\beta}A_\omega+B_\omega\bigr)\biggr)\cdotp
\end{equation}
Making the change of variable    $\tau\triangleq (t\rho^\beta)/\kappa$ and $r\triangleq \kappa\rho^{\alpha-\beta},$
we discover that   $z(\tau)\triangleq  \wh Z(t,\xi)$  is the solution to 
\begin{equation}\label{eq:z}z'+ E_{r,\omega}z=0
\with E_{r,\omega}\triangleq r A_\omega+B_\omega\end{equation}
since  we have 
$$z(\tau)=z(0)\exp\Bigl(-\tau E_{r,\omega}\Bigr)\cdotp$$
Hence,  the case $\alpha=1,$  $\beta=0$ and $\kappa=1,$
  is generic at the linear level. 


\subsection{Derivation of a Lyapunov functional}

The long time behavior of $z$ is closely connected to the signs of the real  part of the eigenvalues
of  the matrix  $E_{r,\omega}$ defined in \eqref{eq:z}.
The  method proposed by K. Beauchard and E. Zuazua in \cite{BZ}
(see also \cite{D2,D3}), that is inspired by   Kalman's control theory for linear ODEs
 supplies a simple way for constructing an explicit  
 Lyapunov functional  and a dissipation term altogether
\emph{without computing the eigenvalues}.   
  \smallbreak
To explain the construction, fix some $r>0$ and $\omega\in\S^{d-1},$ and consider the ODE  \eqref{eq:z} satisfied by $z.$ 
Combining the  assumptions \eqref{eq:skew} and \eqref{eq:positive} with the renormalization 
\eqref{eq:def} ensures  that  \begin{equation}\label{eq:antiellipt}
\Re((A_\omega\eta)\cdot \eta)=0\ \hbox{ and }\ 
\Re((B_\omega\eta)\cdot\eta)\geq |B_\omega\eta|^2
\quad\hbox{for all }\ (\omega,\eta)\in\S^{d-1}\times\C^n.
\end{equation}
Hence, taking the Hermitian product in $\C^n$ of \eqref{eq:z} with $z$ and keeping the real part yields
\begin{equation}\label{eq:z2}\frac12 \frac d{dt}|z|^2 +|B_\omega z|^2 \leq0.\end{equation}
If $B_\omega$ has rank strictly smaller than $n,$ then the above inequality 
does not ensure decay of all the components of $z$ (even though this decay exists
whenever $r>0$ and $\omega\in\S^{d-1}$ are such that the real parts of all the eigenvalues of 
the matrix $E_{r,\omega}$ are positive). 
To recover the decay (if any) for the `missing components' 
of the solution, one can start with the identity
$$(B_\omega z)'+(r B_\omega  A_\omega +B_\omega^2)z=0.$$
Hence,  taking the Hermitian product with $BAz$   (we drop the index  $\omega$ for better readability), we obtain
$$Bz'\cdotp BAz + r|BA z|^2 +B^2z\cdot BA z=0.$$
Similarly, we have
$$Bz\cdotp BA z'+ rBz\cdotp BA^2z+ Bz\cdotp B^2 Az=0,$$
whence
$$\frac d{dt}(Bz\cdot BAz) +r|BAz|^2 +B^2z\cdotp BAz + Bz\cdotp B^2Az= -r Bz\cdotp BA^2z.$$
Remembering \eqref{eq:z2} and  using  several times  the obvious inequality 
$$2\Re (a\cdotp b)\leq  K|a|^2 + K^{-1}|b|^2$$
with suitable values of $K,$   we discover  that one can find 
some $\varepsilon_1$ (that can be taken arbitrarily small) such that 
\begin{multline}\label{eq:z1}\frac d{dt}\Bigl(|z|^2 +\varepsilon_1\min (r,r^{-1}) \Re (Bz\cdot BAz)\Bigr)+|Bz|^2
+\varepsilon_1\min(1,r^2) |BA z|^2 \\\leq C\varepsilon_1\min(1,r^2) |BA^2z|^2.\end{multline}
\medbreak 
In the case $BA^2\not=0,$ we need (at least) one more relation to handle the term in the right-hand side.
For that, one can start from the equation
$$(BAz)'+(r BA^2 +BAB)z=0$$ and take the Hermitian scalar product with $BA^2z,$ adding  up the resulting identity 
multiplied by a small enough $\varepsilon_2$ to \eqref{eq:z1}, then iterate the procedure. 
The fundamental observation 
of Beauchard and Zuazua in \cite{BZ} is that Cayley-Hamilton theorem ensures the existence  of 
complex numbers $c_0,\cdots,c_{n-1}$ so that 
$$A^n= \sum_{k=0}^{n-1} c_k A^k.$$
Consequently, one can end the process after at most $n$ steps.
 In the end,  we get positive parameters $\ep_0=1$ and $\ep_1,\cdots,\ep_{n-1}$ 
(that are defined inductively and  can be taken arbitrarily small) such that for all $\omega\in\S^{d-1}$ and $r>0,$
 we have 
\begin{multline}\label{eq:Lomega}
\frac d{dt} L_{r,\omega}(z)+\frac{\min(1,r^2)}2\sum_{\ell=0}^{n-1} \ep_\ell|B_\omega A_\omega^\ell z|^2\leq0
\\\with L_{r,\omega}(z)\triangleq |z|^2+\min(r,r^{-1})\sum_{\ell=1}^{n-1}  \ep_\ell\Re(B_\omega A_\omega^{\ell-1}z\cdotp B_\omega A_\omega^\ell z)\end{multline}
and, additionally, 
\begin{equation}\label{eq:equivL}
\frac12|z|^2\leq  L_{r,\omega}(z)\leq 2|z|^2.\end{equation}
Consequently, denoting $\displaystyle N_\omega\triangleq \inf\:\Bigl\{\sum_{\ell=0}^{n-1} \ep_\ell |B_\omega A_\omega^\ell x|^2\, ,\, x\in \S^{d-1}\Bigr\},$
we conclude from \eqref{eq:Lomega} and \eqref{eq:equivL} that 
\begin{equation}\label{eq:Lromega}
L_{r,\omega}(\tau)\leq e^{-\frac14\min(1,r^2)N_\omega \tau} L_{r,\omega}(0)\qquad \omega\in\S^{d-1},\quad r>0.\end{equation}
In the particular case where 
\begin{equation}\label{eq:Nomega}
N_\omega>0\quad\hbox{for all}\quad \omega\in\S^{d-1},\end{equation}
 (the only situation  
that will be considered in these notes) then $N_\omega$ is actually bounded away from zero 
owing to the compactness of the sphere.  Hence, \eqref{eq:Lromega} implies that 
there exists a positive constant $c$ such that
for all $r>0$ and $\omega\in\S^{d-1},$ we have
$$L_{r,\omega}(\tau)\leq e^{-2c\min(1,r^2)\tau}L_{r,\omega}(0),\qquad \tau\geq0.$$
Then, using once more \eqref{eq:equivL} and reverting to the original unknown $\wh Z,$ we conclude  that 
\begin{equation}\label{eq:whZ}|\wh Z(t,\xi)|\leq 2 e^{-c\min(\kappa^{-1}|\xi|^\beta,\kappa|\xi|^{2\alpha-\beta})t}
|\wh Z_0(\xi)|.\end{equation}
In other words,  if \eqref{eq:Nomega} holds then:
\begin{itemize}
\item  either  $\alpha>\beta,$  and  we are in a partially dissipative regime
similar to that of  linearized compressible Euler equations, 
\item or  $\alpha<\beta,$ and  we are in a partially diffusive regime analogous to that
of the linearized compressible Navier-Stokes equations. 
\end{itemize}
\smallbreak
It has been pointed out in \cite{BZ} that \eqref{eq:Nomega} is equivalent 
to the  Shizuta-Kawashima condition.
The following lemma  stresses   the link  between those two conditions, 
the strict dissipativity of System \eqref{eq:Z} and Kalman's condition for observability.  
\begin{lem}\label{l:SK}  Let $A$ and $B$ be two $n\times n$ complex valued matrices.  Assume that $A$ is skew-symmetric 
in the meaning of \eqref{eq:skew} and that $B$ is  nonnegative in the sense of \eqref{eq:positive}. 
The following  properties are equivalent:
\begin{enumerate} 
\item  For all positive   $\ep_0,\cdots,\ep_{n-1},$ we have 
$\displaystyle\sum_{\ell=0}^{n-1}\ep_\ell|BA^\ell\eta|^2>0$   for all $\eta\in\S^{n-1}.$ \mb
\item  We have the Kalman rank property, 
namely  the $n^2\times n$
 matrix $\left(\begin{smallmatrix} B\\BA\\\dots\\ BA^{n-1}\end{smallmatrix}\right)$ has rank  equal to~$n.$\mb
\item The (SK) condition holds true, namely the  intersection between $\ker B$ and the linear space of all eigenvectors of $A$ 
is reduced to $\{0\}.$\sb
\item All eigenvalues of $A+B$ have positive real parts. 
\end{enumerate}
\end{lem}
\begin{proof} The equivalence between the first three items is basic linear algebra (see details 
in e.g. \cite{BZ}), while  Inequality \eqref{eq:Lromega} (with $A_\omega=A,$  $B_\omega=B$ and $r=1$)
ensures equivalence with the last item. \end{proof}
\smallbreak
As an example, let us again consider the linearized compressible Euler equations \eqref{eq:leuler}. 
As said before,  \eqref{eq:skew} and \eqref{eq:positive} are satisfied with  
$\alpha=1,$ $\beta=0,$ $\kappa={\mathfrak f}^{-1}.$ Furthermore, we have 
$$A_\omega=i\begin{pmatrix} 0&\omega\\{}^t\!\omega&0\end{pmatrix}\andf 
B_\omega=\begin{pmatrix}0&0\\0&I_d\end{pmatrix},\quad\hbox{so that}\quad
B_\omega A_\omega=i\begin{pmatrix} 0&0\\{}^t\!\omega&0\end{pmatrix}
\cdotp$$
Hence   $\left(\begin{smallmatrix} B_\omega \\ B_\omega A_\omega\end{smallmatrix}\right)$ has
rank $d+1$ and Kalman rank condition is thus satisfied, which gives  eventually
$$|\wh Z(t,\xi)|\leq 2e^{-c\min({\mathfrak f},{\mathfrak f}^{-1}|\xi|^2)t}|\wh Z_0(\xi)|\quad\hbox{for all }\ \xi\in\R^d\andf
t\geq0.$$
Since we do not need higher powers of $A_\omega$ to ensure the Kalman rank condition, one can suspect  that 
one can restrict the sum in the definition of the Lyapunov function $L_{r,\omega}$ to only  one term  ($\ell=1$). 
Now, the reader may observe by direct computation that $B_\omega A^2_\omega = A^2_\omega B_\omega.$
Hence 
$$-B_\omega z\cdotp B_\omega A^2_\omega z = - B_\omega z\cdotp  A^2_\omega B_\omega z \leq C |B_\omega z|^2$$
and taking $\eps_1$ sufficiently small in \eqref{eq:z1}  allows to  just have  
$$\frac d{dt}\Bigl(|z|^2 +\varepsilon_1\min (r,r^{-1}) \Re (B_\omega z\cdot B_\omega A_\omega z)\Bigr)+|B_\omega z|^2
+\varepsilon_1\min(1,r^2) |B_\omega A_\omega z|^2 \leq0.$$
One can be more explicit :  since $z=(\wh a, \wh u)$ and $\xi =r\omega,$  
 we have    
$$
\min(r,r^{-1}) B_\omega z\cdotp B_\omega A_\omega z=   \min (1,|\xi|^{-2}) (\wh u \cdotp (i\xi)\wh a)
=   \min (1,|\xi|^{-2}) (\wh u \cdotp (\wh Da)).$$
Hence,  we conclude that the Lyapunov functional is of the form 
$$\cL(\xi) = |\wh a|^2 + |\wh u|^2 + \eps_1   \min (1,|\xi|^{-2}) \Re  (\wh u \cdotp (\wh Da)).$$
Combining with Fourier-Plancherel theorem, one can conclude  that in order to recover the full dissipative properties of the linearized  compressible Euler equations,  it suffices to  consider  the functional 
$$\|a\|_{L^2}^2 +\| u\|_{L^2}^2 + \eps_1\int_{\R^d} u\cdotp (\Id-\Delta)^{-1}\nabla a \,dx$$
with suitably small $\eps_1$ or, rather, spectrally localized versions of it.  
\medbreak
Similar computations are valid for the linearized compressible Navier-Stokes equations \eqref{eq:lns}. The  reader
may find more details in \cite{D1}.

\subsection{Derivation of a priori estimates} 

Let us assume from now on that $\kappa=1,$ $\alpha=1$ and $\beta=1$ in \eqref{eq:general}
(since  the general case $\alpha\not=\beta$ reduces to that one). 
Recall Duhamel's formula \eqref{eq:duhamel}. Combining with \eqref{eq:whZ}, we get
$$|\wh Z(t,\xi)|\leq 2\biggl(e^{-c\min(1,|\xi|^2)t}|\wh Z_0(\xi)| +\int_0^t e^{-c\min(1,|\xi|^2)(t-\tau)}|\wh F(\tau,\xi)|\,d\tau\biggr)\cdotp$$
Clearly, if one wants  to get  optimal estimates then low and high frequencies
have to be treated differently.  
To proceed, we shall actually  use a more accurate decomposition of the Fourier space, 
namely  a dyadic homogeneous Littlewood-Paley decomposition $(\ddj)_{j\in\Z}$ 
defined by $\ddj\triangleq \varphi(2^{-j}D).$  Here,  $\varphi$ is a smooth nonnegative function
on $\R^d,$ supported in (say) the annulus $\{\xi\in\R^d,\:  3/4\leq|\xi|\leq 8/3\}$ and satisfying 
$$\sum_{j\in\Z}\varphi(2^{-j}\xi)=1,\qquad \xi\not=0.$$
By construction, $\ddj$ is a localization operator in the vicinity of frequencies of magnitude $2^j.$ 
Since $\ddj$ commutes with any Fourier multiplier,  each $Z_j\triangleq\ddj Z$ satisfies \eqref{eq:general} with source term $F_j\triangleq\ddj F$ and 
initial data $Z_{0,j}\triangleq\ddj Z_0.$ Therefore,  we have
$$Z_j(t)=T(t)Z_{0,j}+\int_0^t T(t-\tau) F_j(\tau)\,d\tau,$$
whence, as $|\xi|\simeq2^j$ on $\Supp\wh Z_j,$ we have (changing slightly $c$ if needed), 
$$|\wh Z_j(t,\xi)|\leq 2\biggl(e^{-c\min(1,2^{2j})t}|\wh Z_{j,0}(\xi)| 
+\int_0^t e^{-c\min(1,2^{2j})(t-\tau)}|\wh F_j(\tau,\xi)|\,d\tau\biggr)\cdotp$$
Consequently, after taking the $L^2(\R^d)$ norm of both sides, then using Minkowski 
inequality and Fourier-Plancherel theorem, we end up with:
$$\|Z_j(t)\|_{L^2}\leq 2\biggl( e^{-c\min(1,2^{2j})t}\|Z_{0,j}\|_{L^2}+\int_0^te^{-c\min(1,2^{2j})(t-\tau)}
\|F_j(\tau)\|_{L^2}\,d\tau\biggr),$$
whence
\begin{equation}\label{eq:Zj}\|Z_j(t)\|_{L^2}+c\min(1,2^{2j})\int_0^t\|Z_j\|_{L^2}\,d\tau
\leq 2\biggl( \|Z_{0,j}\|_{L^2}+\int_0^t\|F_j(\tau)\|_{L^2}\,d\tau\biggr)\cdotp\end{equation}
At this stage, two important observations are in order. First,  note that
$$\|Z\|_{\dot H^s}\simeq \biggl(\sum_{j\in\Z} 2^{2js}\|Z_j\|_{L^2}^2\biggr)^{1/2}.$$
Hence, in order to get a Sobolev estimate of $Z,$  it suffices  to  multiply 
\eqref{eq:Zj} by $2^{js}$ then to perform an $\ell^2$-summation on $j\in\Z.$ 
However, the second term of \eqref{eq:Zj} will not exactly give an estimate 
in some  space $L^1(0,t;\dot H^{\sigma})$   since the time integration has
been performed \emph{before} the summation with respect to $j$:
 one ends up in one of the Chemin-Lerner (or `tilde') spaces that have been introduced 
in \cite{ChL}. They  turn out to be delicate to manipulate  and 
not adapted to the critical regularity setting we have in mind. 

The second observation  is that, owing to the factor $\min(1,2^{2j}),$ in order  to track as much information as possible, it 
is suitable  to work with \emph{different} 
regularity exponents for low and high frequencies. 

Putting the two observations together, this motivates us to multiply \eqref{eq:Zj} 
by $2^{js}$ with a different value of the `regularity exponent' $s$ for negative and positive $j$'s, 
then to perform  an $\ell^1$ summation with respect to $j.$ 
The advantage of  $\ell^1$ summation  -- that corresponds to Besov norms \emph{with last index $1$} --
is that one can freely exchange time integration
and summation on $j.$ 
Taking  into account  the possible difference of regularity between the low and high frequencies 
 leads us to   introduce for all pair $(s,s')\in\R^2$ the \emph{hybrid Besov space} $\wt B^{s,s'}_{2,1},$ 
 that is   the set of all tempered distributions $z$ such that 
\begin{equation}\label{eq:hybrid}
\|z\|_{\wt B^{s,s'}_{2,1}} \triangleq  \sum_{j<0} 2^{js}\|\ddj z\|_{L^2} + \sum_{j\geq0} 2^{js'}\|\ddj z\|_{L^2}<\infty\andf
\lim_{j\to-\infty} \|\chi(2^{-j}D)z\|_{L^\infty} =0.\end{equation}
Above, $\chi$ stands for a compactly supported smooth function on $\R^d$ such that $\chi(0)=1,$
and the condition on $\chi(2^{-j}D)z$ implies that $z$ has to tend to $0$ at $\infty$ in 
the sense of tempered distributions\footnote{This is a way to rule out polynomials from 
homogeneous Besov spaces, otherwise one would have to work modulo 
polynomials which is not suitable when studying PDEs.}. Classical homogenous Besov spaces correspond to $s=s'$ and will be denoted by $\dot B^s_{2,1}.$
\smallbreak
In what follows, it will be sometimes convenient to use 
the following notation  for all $\sigma\in\R$:
$$\|z\|^\ell_{\dot B^\sigma_{2,1}}\triangleq \sum_{j<0} 2^{j\sigma}\|\ddj z\|_{L^2}\andf
\|z\|^h_{\dot B^\sigma_{2,1}}\triangleq \sum_{j\geq0} 2^{j\sigma}\|\ddj z\|_{L^2}.$$
\smallbreak
Even though most of the functions we shall consider here will have range in the set of  vectors or even matrices, 
we shall keep the same notation for Besov spaces pertaining to this case. 
\medbreak
Now, multiplying \eqref{eq:Zj} by $2^{js}$ (resp. $2^{js'}$) for $j\leq1$ (resp. $j\geq0$) and summing up on $j\leq1$ (resp.  $j\geq0$) leads to
\begin{eqnarray}\label{eq:Zlf} \|Z(t)\|_{\dot B^s_{2,1}}^\ell + \int_0^t\|Z\|_{\dot B^{s+2}_{2,1}}^\ell\,d\tau
&\!\!\!\leq\!\!\!&  2\biggl(\|Z_0\|_{\dot B^s_{2,1}}^\ell + \int_0^t\|F\|_{\dot B^{s}_{2,1}}^\ell\,d\tau\biggr),\\\label{eq:Zhf}
 \|Z(t)\|_{\dot B^{s'}_{2,1}}^h + \int_0^t\|Z\|_{\dot B^{s'}_{2,1}}^h\,d\tau
&\!\!\!\lesssim\!\!\!& 2\biggl(\|Z_0\|_{\dot B^{s'}_{2,1}}^h + \int_0^t\|F\|_{\dot B^{s'}_{2,1}}^h\,d\tau\biggr)\cdotp\end{eqnarray}
Hence, putting together those two inequalities yields 
\begin{equation}\label{eq:Ztotal}
 \|Z(t)\|_{\wt B^{s,s'}_{2,1}} + \int_0^t\bigl(\|Z\|_{\dot B^{s+2}_{2,1}}^\ell
 +\|Z\|_{\dot B^{s'}_{2,1}}^h\bigr)\,d\tau \leq 2\biggl(  \|Z_0\|_{\wt B^{s,s'}_{2,1}}
 +\int_0^t \|F\|_{\wt B^{s,s'}_{2,1}}\,d\tau\biggr)\cdotp\end{equation}
Since a part of the solution experiences direct dissipation, one can 
suspect  the low frequency integrability we get in this way to be not  optimal. 
Recovering better integrability  for a part of the solution is the goal of the next subsection.


\subsection{The damped mode}

Assume that the system has an orthogonal block structure, that is independent of the frequency, namely
$${\rm rank} B_\omega \overset\perp \oplus {\rm Ker} B_\omega=\C^n\quad\hbox{for all }\ \omega\in\S^{n-1},$$
with $M\triangleq {\rm Ker} B_\omega$ independent of $\omega.$ 
\medbreak
Denote by $\cP$ the orthogonal projector onto $M^\perp$ and set 
\begin{equation}\label{eq:W0} W\triangleq \cP(A+B)(D)Z.\end{equation}
Since $\cP$ and $B$ commute, we  get the following equation for $W$:
$$\d_tW+B(D)W=\cP(A+B)(D)F-\cP A(D)(A+B)(D)Z.$$
Because
$$\cP A(D) B(D)Z= \cP A(D) \cP B(D) Z=\cP A(D) W- (\cP A(D))^2Z,$$
this may be rewritten: 
\begin{equation}\label{eq:W} 
\d_tW+B(D)W=\cP(A+B)(D)F -\cP A(D) W +(\cP A(D))^2Z -\cP A^2(D)Z.
\end{equation} 
As $A(D)$ and $B(D)$ are of order $1$ and $0,$ respectively, 
 multipliers of orders  $1$ and $2,$  act  on $W$ and
$Z$ in the right-hand side. Hence the low frequencies of the corresponding terms are expected to be negligible
compared to the left-hand side of~\eqref{eq:W}. 
\medbreak
To make this heuristics rigorous,  let us look at the equation for  $W_j\triangleq \ddj W,$  namely
\begin{equation}\label{eq:Wj} 
\d_tW_j+B(D)W_j=\cP(A+B)(D)F_j-\cP A(D) W_j+(\cP A(D))^2Z_j -\cP A^2(D) Z_j.
\end{equation} 
Taking the Hermitian scalar product in $\C^n$ with $W_j,$ using \eqref{eq:positive}, the fact that
$B(D)$  is $0$-order and that $A(D)$ is $1$-st order yields
$$
\frac12\frac d{dt}|\wh W_j|^2+|\wh W_j|^2\leq C\Bigl((1+|\xi|) |\wh F_j|+|\xi||\wh W_j|+|\xi|^2|\wh Z_j|\Bigr)|\wh W_j|.
$$
Hence, integrating on $\R^d$ and taking advantage of the Fourier-Plancherel theorem yields:
$$
\frac12\frac d{dt}\|W_j\|_{L^2}^2+ \|W_j\|_{L^2}^2 \leq C\|W_j\|_{L^2}\Bigl((1+2^j)\|F_j\|_{L^2} + 2^j\|W_j\|_{L^2}
 + 2^{2j}\|Z_j\|_{L^2}\Bigr)$$
from which we eventually get for all $t\geq0$ and $j\in\Z,$ owing to Lemma \ref{SimpliCarre},
$$\displaylines{\|W_j(t)\|_{L^2}+\int_0^t\|W_j\|_{L^2}\,d\tau\leq \|W_j(0)\|_{L^2}
\hfill\cr\hfill+ C(1+2^j)\int_0^t\|F_j\|_{L^2}\,d\tau 
 + C2^{j}\int_0^t\|W_j\|_{L^2}\,d\tau + C2^{2j}\int_0^t\|Z_j\|_{L^2}\,d\tau.}$$
Therefore, if we multiply by $2^{js}$ and sum up on $j\leq j_0$ with $j_0$ chosen 
so that $C2^{j_0}\leq 1/2,$ then we end up with
$$\displaylines{\sum_{j\leq j_0}2^{js}\|W_j(t)\|_{L^2}+\frac 12\int_0^t\sum_{j\leq j_0}2^{js}\|W_j\|_{L^2}\,d\tau
\leq \sum_{j\leq j_0} 2^{js}\|W_j(0)\|_{L^2}
\hfill\cr\hfill+ C\int_0^t\sum_{j\leq j_0} 2^{js}\|F_j\|_{L^2}\,d\tau 
 + C\sum_{j\leq j_0} \int_0^t2^{j(s+2)}\|Z_j\|_{L^2}\,d\tau.}$$
 The last term may be controlled by the data according to \eqref{eq:Zlf}. 
 Furthermore, 
 $\|W_j\|_{L^2} \lesssim \|Z_j\|_{L^2}$ for all $j<0,$ and 
$2^{j(s+2})\simeq 2^{js}$ for $j_0\leq j<0.$ 
Hence the above inequality still holds if one sums up to $j=0$. 
In the end, this allows us to get the following additional bound: 
$$ \|W(t)\|_{\dot B^{s}_{2,1}}^\ell+ \int_0^t\|W\|_{\dot B^{s}_{2,1}}^\ell\,d\tau \lesssim 
\|Z_0\|_{\dot B^s_{2,1}}^\ell+\int_0^t\|F\|_{\dot B^s_{2,1}}^\ell\,d\tau.$$
Let us finally look at  the part of $Z$ that undergoes direct dissipation, namely
 $Z_2\triangleq \cP Z.$  We  claim that, as expected,   the low frequencies of $Z_2$ have
 better time integrability than the overall solution $Z.$ 
 Indeed, observing that $B(D)Z_2=W-\cP A(D)Z$
and   that  $\cP B(D)$  (restricted to functions defined on $M$) is invertible,
we may write 
$$
Z_2=((\cP B)(D))^{-1} W-((\cP B)(D))^{-1} \cP A(D) Z.$$
Hence, since $(\cP B)(D)$ (resp. $A(D)$) is a $0$-order (resp. $1$-st order) Fourier multiplier, we may write
$$\|Z_2\|^\ell_{\dot B^{s+1}_{2,1}} \lesssim \|W\|^\ell_{\dot B^{s+1}_{2,1}}  +
\|Z\|^\ell_{\dot B^{s+2}_{2,1}} \andf
\|Z_2\|^\ell_{\dot B^{s}_{2,1}} \lesssim \|W\|^\ell_{\dot B^{s}_{2,1}}  +
\|Z\|^\ell_{\dot B^{s+1}_{2,1}}.$$ 
Then,  remembering  \eqref{eq:Zlf} and using H\"older inequality and 
interpolation in Besov spaces when needed yields
$$\|Z_2\|_{L^2(\R_+;\dot B^{s}_{2,1})}^\ell + \|Z_2\|_{L^1(\R_+;\dot B^{s+1}_{2,1})}^\ell 
 \lesssim \|Z_0\|_{\dot B^s_{2,1}}^\ell+\int_0^t\|F\|_{\dot B^s_{2,1}}^\ell\,d\tau.$$
 This has to be compared by the following (optimal) inequality for $Z$:
 $$\|Z\|_{L^2(\R_+;\dot B^{s+1}_{2,1})}^\ell + \|Z_2\|_{L^1(\R_+;\dot B^{s+2}_{2,1})}^\ell 
 \lesssim \|Z_0\|_{\dot B^s_{2,1}}^\ell+\int_0^t\|F\|_{\dot B^s_{2,1}}^\ell\,d\tau.$$

\subsection{An  $L^p$ approach}

In this part, we are going to show that under slightly stronger structure 
assumptions\footnote{That are in particular satisfied by the linearized Euler equations
with relaxation.}  on the linear 
system \eqref{eq:Zlinear} than those that have been made so far,  it is possible to bound 
 the low frequencies of the solution on functional spaces built on $L^p$ for any $p\in[1,\infty].$
  This unusual setting is in sharp contrast with  the non dissipative case. 
  In fact, as pointed out  by P. Brenner in \cite{Brenner}, 
  apart from the notable exception of the transport equation, `most' 
   first order  `purely' hyperbolic  systems are ill-posed  in $L^p$ if $p\not=2.$ 
   It turns out that for nonlinear  partially dissipative systems satisfying the structure
   assumptions of this part, it is also possible to use, at least partially, 
   an $L^p$ type framework (see details  in \cite{CBD1,CBD3}). 
      This offers  one more degree of freedom in the choice of solutions 
      spaces allowing not only to prescribe weaker smallness conditions for global well-posedness, 
      but also to get  more  accurate informations on the qualitative properties of the constructed solutions. 
 
\medbreak
 In order to proceed, let us assume without  loss of generality that $M=\R^{n_1}\times\{0\}$ and 
 decompose $Z\in\R^n$ into \begin{small}$\begin{pmatrix} Z_1\\Z_2\end{pmatrix}$\end{small}.
 For expository purpose, further assume that there is no source term ($F=0$). 
 Then, System \eqref{eq:Z} may be rewritten by blocks as follows: 
  \begin{equation}\label{eq:Zblock}\frac d{dt}\begin{pmatrix} Z_1\\Z_2\end{pmatrix} +\begin{pmatrix} A_{11}(D)&A_{12}(D)\\A_{21}(D)&A_{22}(D)\end{pmatrix}
  \begin{pmatrix} Z_1\\Z_2\end{pmatrix}+\begin{pmatrix} 0\\B_{22}(D)Z_2\end{pmatrix}=
  \begin{pmatrix}0\\0\end{pmatrix},\end{equation}
  where the $0$-order Fourier multiplier $B_{22}(D)$   has symbol in $\cM_{n_2}(\R),$ and so on.  
  \smallbreak
In the spirit of the computations of the previous paragraph, let us introduce
\begin{equation}\label{eq:WW}
W\triangleq  Z_2+(B^{-1}_{22}A_{21})(D) Z_1 +(B_{22}^{-1}A_{22})(D)Z_2.\end{equation}
This definition of a damped mode is  consistent  with the one we had
before: we just applied to \eqref{eq:W0}  the $0$-order operator $(B_{22}(D))^{-1}$
that corresponds to the inverse of  $\cP B(D)$  restricted to  $M.$  
Now, we note that 
$$\d_tZ_2+B_{22}(D)W=0$$
and that from the definition of $Z$, we have
$$\d_tW+B_{22}(D)W = (B_{22}^{-1}A_{11})(D)\d_tZ_1+(B_{22}^{-1}A_{22})(D)\d_tZ_2.$$
Hence, using   System \eqref{eq:Zblock} for computing $\d_tZ,$
 we get the following  equation:
\begin{equation}\label{eq:equationW}\d_tW+B_{22}(D)W=-(B_{22}^{-1}A_{21})(D)\bigl(A_{11}(D)Z_1+A_{12}(D)Z_2\bigr)
-(B_{22}^{-1}A_{22}B_{22})(D)W.\end{equation}
Rewriting the equation of  $Z_1$ in terms of $W$ yields 
\begin{equation}\label{eq:equationZ1}\d_tZ_1 +\bigl(A_{11}(D)-(A_{12}B_{22}^{-1}A_{21})(D)\bigr)Z_1= (A_{12}B_{22}^{-1}A_{22})(D)Z_2
-A_{12}(D)W.\end{equation}
In order to pursue our analysis, we make the following assumption: 
\begin{equation}A_{11}(D)\equiv0\andf \cA(D)\triangleq -(A_{12}B_{22}^{-1}A_{21})(D)\quad\hbox{is a positive
operator}.\end{equation}
By positive, we mean that the symbol $A_{12}B_{22}^{-1}A_{21}$ has range in the set 
of positive Hermitian matrices of size $n_2.$ 
For this particular structure,  the above hypothesis turns out to be equivalent to Condition (SK) (see Lemma \ref{l:elliptic}). 
\smallbreak
Then, after applying $\ddj$ to \eqref{eq:equationW}  and for \eqref{eq:equationZ1},  we obtain:
\begin{equation}\label{eq:Wp}\left\{\begin{aligned}
&\d_tZ_{1,j}+\cA(D)Z_{1,j}=  (A_{12}B_{22}^{-1}A_{22})(D)Z_{2,j}
-A_{12}(D)W_j,\\
&\d_tW_j+B_{22}(D)W_j=-(B_{22}^{-1}A_{21}A_{12})(D)Z_{2,j}-(B_{22}^{-1}A_{22}B_{22})(D)W_j.\end{aligned}\right.\end{equation}
Using Duhamel formula for computing $Z_{1,j}$ from the first equation of \eqref{eq:Wp}, we get
$$Z_{1,j}(t)= e^{-t\cA(D)} Z_{1,j}(0)+\int_0^te^{-(t-\tau)\cA(D)}\bigl( (A_{12}B_{22}^{-1}A_{22})(D)Z_{2,j}(\tau)-A_{12}(D)W_j(\tau)\bigr)d\tau.$$
Since $\cA(D)$ is second order positive and satisfies the assumptions of  Lemma \ref{l:sg},  there exist two constants $c$ and  $C$ such that the following bound holds:
\begin{equation}\label{eq:boundALp}
\|e^{-t\cA(D)}\ddj z\|_{L^p(\R^{n_1};\R^{n_1})}\leq Ce^{-c2^{2j}t} \|\ddj z\|_{L^p(\R^{n_1};\R^{n_1})},\qquad j\in\Z.
\end{equation}
Then, we get from Bernstein inequality \eqref{eq:bernstein3},
 remembering that all the blocks of $A(D)$ are homogeneous
multipliers of degree $1$ and that $B_{22}^{-1}(D)$ is homogeneous of degree $0,$  
$$\|Z_{1,j}(t)\|_{L^p} \lesssim e^{-c2^{2j}t} \|Z_{1,j}(0)\|_{L^p} + \int_0^t e^{-c2^{2j}(t-\tau)}
\bigl(2^{2j}\|Z_{2,j}(\tau)\|_{L^p} + 2^j\|W_j(\tau)\|_{L^p}\bigr)\,d\tau,$$
whence taking the supremum  or the integral on $[0,t],$ 
$$\|Z_{1,j}(t)\|_{L^p} + 2^{2j}\int_0^t\|Z_{1,j}\|_{L^p}\,d\tau
\lesssim\|Z_{1,j}(0)\|_{L^p} +\int_0^t\bigl(2^{2j}\|Z_{2,j}\|_{L^p} + 2^j\|W_j\|_{L^p}\bigr)\,d\tau.$$
Similarly, Lemma \ref{l:elliptic} guarantees that we have 
\begin{equation}\label{eq:boundBLp}
\|e^{-tB_{22}(D)}\ddj z\|_{L^p(\R^{n_2};\R^{n_2})}\leq Ce^{-ct} \|\ddj z\|_{L^p(\R^{n_2};\R^{n_2})},\qquad j\in\Z,
\end{equation}
which allows to get eventually  
$$\|W_j(t)\|_{L^p} + \int_0^t\|W_j\|_{L^p}\,d\tau \lesssim \|W_j(0)\|_{L^p}+ 2^{2j}\int_0^t\|Z_j\|_{L^p}\,d\tau
+ 2^j\int_0^t\|W_j\|_{L^p}\,d\tau.$$
Owing to the factor $2^j,$  there exists an integer $j_0\in\Z$ so that the last term may be absorbed by 
the left-hand side for all $j\leq j_0.$ 
Hence, multiplying by $2^{js}$ then summing up on $j\leq j_0$ yields, with the notation
$\|z\|^{\ell,j_0}_{\dot B^s_{p,1}}\triangleq \sum_{j\leq j_0} 2^{j\sigma}\|\ddj z\|_{L^p},$ 
\begin{equation}\label{eq:Wt} \|W(t)\|_{\dot B^s_{p,1}}^{\ell,j_0} +\int_0^t\|W\|_{\dot B^s_{p,1}}^{\ell,j_0}\,d\tau
 \lesssim \|W_0\|_{\dot B^s_{p,1}}^{\ell,j_0}+ \int_0^t\|Z\|_{\dot B^{s+2}_{p,1}}^{\ell,j_0}\,d\tau\end{equation}
while the inequality for $Z_1$ gives us 
 \begin{equation}\label{eq:Z1}\|Z_1(t)\|_{\dot B^s_{p,1}}^{\ell,j_0} +\int_0^t\|Z_1\|_{\dot B^{s+2}_{p,1}}^{\ell,j_0}\,d\tau
 \lesssim \|Z_{1,0}\|^{\ell,j_0}_{\dot B^s_{p,1}}+ \int_0^t\bigl(\|Z_2\|_{\dot B^{s+2}_{p,1}}^{\ell,j_0}
 +\|W\|_{\dot B^{s+1}_{p,1}}^{\ell,j_0}\bigr)d\tau.\end{equation}
The definition of $W$ in \eqref{eq:WW} ensures that for all $j\leq j_0$ (with negative enough $j_0$), there holds that
\begin{equation}\label{eq:Z2}
\|W_j\|_{L^p}\lesssim \|Z_{2,j}\|_{L^p} + 2^j\|Z_{1,j}\|_{L^p}\andf
\|Z_{2,j}\|_{L^p}\lesssim \|W_{j}\|_{L^p} + 2^j\|Z_{1,j}\|_{L^p}.\end{equation}
Hence, adding up $\ep\cdot$\eqref{eq:Wt} to \eqref{eq:Z1} with $\ep$ small enough and 
negative enough $j_0$,  we conclude that 
$$\|Z(t)\|_{\dot B^s_{p,1}}^{\ell,j_0}+ \int_0^t\bigl(\|Z_1\|^{\ell,j_0}_{\dot B^{s+2}_{p,1}}+ \|W\|_{\dot B^s_{p,1}}^{\ell,j_0}\bigr)\,d\tau
\lesssim \|Z_0\|_{\dot B^s_{p,1}}^{\ell,j_0}.$$
Of course,  combining with \eqref{eq:Z2} yields also 
 $$\int_0^t \|Z_2\|^{\ell,j_0}_{\dot B^{s+1}_{p,1}}\,d\tau\lesssim \|Z_0\|_{\dot B^s_{p,1}}^{\ell,j_0}.$$
By the same token, if we consider a source term $F$ in \eqref{eq:Zblock}, one gets the following bound:
$$\|Z(t)\|_{\dot B^s_{p,1}}^{\ell,j_0}+ \int_0^t\bigl(\|Z_1\|^{\ell,j_0}_{\dot B^{s+2}_{p,1}}+
 \|Z_2\|^{\ell,j_0}_{\dot B^{s+1}_{p,1}}+ \|W\|_{\dot B^s_{p,1}}^{\ell,j_0}\bigr)\,d\tau
\lesssim \|Z_0\|_{\dot B^s_{p,1}}^{\ell,j_0}+ \int_0^t \|F\|^{\ell,j_0}_{\dot B^{s}_{p,1}}\,d\tau,$$
which is actually the same as the one we proved before for $p=2.$
\smallbreak
At the linear level, there is no restriction on the value of $p$: it can be any element of $[1,\infty].$
Reverting to the initial nonlinear system \eqref{eq:Z}, it is possible to work out a functional framework 
of $L^p$ type for the low frequencies of the solution. However, owing to the interactions between 
the low  and high frequencies  through the nonlinear terms, 
there are some restrictions on $p.$ The most obvious one  is that, if combining Bernstein and H\"older
inequalities for estimating  the medium frequencies 
in a $L^2$ type space of a product of low frequencies that belong to a $L^p$ type space, one needs to have
 $p\in[2,4].$  In high dimension, there are stronger restrictions on $p.$ 
The reader is referred to \cite{CBD1,CBD3} for more details and complete statements.


\section{Global existence in the critical regularity setting}\label{s:2}

The principal aim of this section is to prove the global existence of strong solutions for \eqref{GEQSYM}
supplemented with initial data that are  a perturbation of a constant state $\bar V$ 
satisfying   Condition (SK). For  notational simplicity, we  assume that $\bar V=0$ so that the system under consideration  
reads\footnote{The reader is  referred to \cite{CBD2} for the proof of similar results for 
more general symmetrizable quasilinear
partially dissipative hyperbolic systems satisfying (SK).}  
\begin{equation}\label{eq:ZZ}\d_tZ +\sum_{k=1}^d A^k(Z)\cdot\d_k Z + BZ=0.\end{equation}
It is assumed 
that the (smooth) given functions $A^1,\cdots,A^d$ range in the set of  $n\times n$ real symmetric matrices, and 
that $B=$\begin{small}$\begin{pmatrix} 0&0\\ 0&L_{2}\end{pmatrix}$\end{small} with $L_{2}\in GL_{n_2}(\R)$
satisfying for some $c>0,$
\begin{equation}\label{eq:B22}
L_{2}z\cdot z \geq c|z|^2,\qquad z\in\R^{n_2}.\end{equation}
Set $A(\xi)\triangleq i\sum_{k=1}^d\xi_k \bar A^k$  with $\bar A^k\triangleq A^k(0),$ 
and $B(\xi)\triangleq B.$ 
According to the linear analysis that was performed in the previous paragraph 
in the context of System \eqref{eq:ZZ}, Condition (SK) is equivalent to: 
  \begin{equation}\label{eq:SK}  {\rm Rank}\; \begin{pmatrix} B(\xi)\\BA(\xi)\\\dots\\ BA^{n-1}(\xi)\end{pmatrix}=n\quad\hbox{for all}\
  \xi\in\R^d\setminus\{0\}.\end{equation}

\subsection{The main results} 

In order to find out  a suitable  functional framework  for solving \eqref{eq:ZZ},  let us temporarily 
consider a smooth solution $Z.$   Taking advantage of the symmetry of the matrices $A^k$ and  integrating by parts, one gets  the following `energy identity':
$$\frac12\frac d{dt} \|Z\|_{L^2}^2
-\frac12\sum_{k,\ell,m} \int_{\R^d} Z^\ell Z^m\d_k(A^k_{\ell m}(Z)) \,dx+\int_{\R^d} BZ\cdot Z\,dx=0.$$
Therefore,  combining  with \eqref{eq:B22} and Gronwall inequality, we discover that
$$\|Z(t)\|_{L^2}^2+c\int_0^t\|Z_2\|_{L^2}^2\,d\tau\leq \|Z_0\|_{L^2}^2 \exp\Bigl(C\int_0^t\|\nabla Z\|_{L^\infty}\,d\tau\Bigr)\cdotp$$
Hence, even for controlling the $L^2$ norm of the solution, a bound of $\nabla Z$ in $L^1_{loc}(\R_+;L^\infty)$ is needed. 
Since no gain of regularity can be expected on the whole solution (see \eqref{eq:Zhf}), 
we must assume that $Z_0$ belongs to a functional space $X$ that is embedded in the set of globally 
Lipschitz functions.  If $X=H^s$ then this embedding holds if and only if  $s>d/2+1.$  In the framework of Besov spaces with last index $1,$ one can reach the critical index $s=d/2+1,$ owing
 to the (critical) embedding
 \begin{equation}\label{eq:embed} \dot B^{\frac d2}_{2,1}(\R^d)\hookrightarrow \cC_b(\R^d). \end{equation}
 Hence,   $X$ must be a subspace of $\dot B^{\frac d2+1}_{2,1}.$ Consequently, we shall take $s'=1+d/2$ in \eqref{eq:Zhf}. 

As regards the value of the regularity exponent $s$ in \eqref{eq:Zlf} for the low frequencies, a natural candidate is 
  $s=-1+d/2$ since \eqref{eq:Zlf} and \eqref{eq:Zhf} together give us 
a control of $Z$ in $L^1(\R_+;\dot B^{\frac d2+1}_{2,1})$ (provided we succeed in bounding 
 in $L^1(\R_+;\dot B^{\frac d2-1}_{2,1})$ the nonlinear  term $F$), 
and thus of $\nabla Z$ in $L^1(\R_+;L^\infty)$.  
Having at our disposal global $L^1$-in-time estimates  for the solution  will be 
particularly comfortable for further analysis in contrast with the classical `Sobolev' approaches
for partially dissipative systems where only $L^2$-in time estimates are available. 

To make a long story short, a good candidate for a solution space is    
the set of functions $Z$ in $\cC^1_b(\R_+\times\R^d;\R^n)$ such that
$$
Z^\ell\in \cC_b(\R_+;\dot B^{\frac d2-1}_{2,1})\cap L^1(\R_+;\dot B^{\frac d2+1}_{2,1})\andf
Z^h\in \cC_b(\R_+;\dot B^{\frac d2+1}_{2,1})\cap L^1(\R_+;\dot B^{\frac d2+1}_{2,1}).$$
According to the  linear analysis presented before, one can expect to get  additional informations for low frequencies, through the damped mode $W$ defined in \eqref{eq:WW}, that is essentially equivalent to $\d_tZ_2$ in our context. 
We will eventually obtain the following result that will be proved in the next subsection. 
\begin{thm}\label{ThmGlobal}  Let the Conditions \eqref{eq:B22} and \eqref{eq:SK} be  in force
and assume that $d\geq2.$ Then, there exists a positive constant $\alpha$ 
such that for all   $Z_0\in \wt{B}^{\frac{d}{2}-1,\frac{d}{2}+1}_{2,1}$  satisfying 
 \begin{equation}\label{eq:smalldata}
\cZ_0\triangleq \|Z_0\|_{\wt{B}^{\frac{d}{2}-1,\frac{d}{2}+1}_{2,1}} \leq \alpha,
\end{equation} System \eqref{eq:ZZ}   supplemented 
with initial data $Z_0$ admits a unique global-in-time solution $Z$ in the set 
$$E\triangleq\bigl\{Z\in \mathcal{C}_b(\mathbb{R}_+;\wt {B}^{\frac{d}{2}-1,\frac{d}{2}+1}_{2,1}),\:
Z\in L^1(\mathbb{R}_+;\dot{B}^{\frac{d}{2}+1}_{2,1})\!\andf\! \d_tZ_2\!\in\! L^1(\mathbb{R}_+;\dot{B}^{\frac{d}{2}-1}_{2,1})\bigr\}\cdotp$$
Moreover,  there exist an explicit  Lyapunov functional, equivalent to $\|Z\|_{\wt {B}^{\frac{d}{2}-1,\frac{d}{2}+1}_{2,1}}$  and  a  constant $C$ depending only 
on  the matrices $A^k$ and on $L_2$, and such that
\begin{equation} \label{eq:X0}
\cZ(t)\leq C\cZ_0\quad\hbox{for all }\ t\geq 0\end{equation} 
where\footnote{Whenever $X$ is a Banach space, $p\in[1,\infty]$ and $T\geq0,$ notation
$\|\cdot\|_{L_T^p(X)}$ designates the Lebesgue norm $L^p$ 
of  functions on $[0,T]$ with values in $X.$}
\begin{multline}\label{eq:X}\cZ(t)\triangleq\|Z\|^\ell_{L^\infty_t(\dot{B}^{\frac{d}{2}-1}_{2,1})}+\|Z\|^h_{L^\infty_t(\dot{B}^{\frac{d}{2}+1}_{2,1})} +\|Z\|_{L^1_t(\dot{B}^{\frac{d}{2}+1}_{2,1})}\\+\|\d_tZ_2\|_{L^1_t(\dot{B}^{\frac{d}{2}-1}_{2,1})}
 +\|Z_2\|^\ell_{L^1_t(\dot{B}^{\frac{d}{2}}_{2,1})}
 +\|Z_2\|^\ell_{L^2_t(\dot{B}^{\frac{d}{2}-1}_{2,1})}.\end{multline}
 \end{thm}
\medbreak
Choosing regularity $d/2-1$ for low frequencies has some disadvantages,  though: 
\begin{itemize}
\item it does not allow to treat the mono-dimensional case  since the low frequencies of the nonlinear terms  of type $DZ\times Z$ cannot be estimated 
in $L^1(\R_+;\dot B^{\frac d2-1}_{2,1})$ (this  is the needed regularity for the right-hand side
of \eqref{eq:Zlf}). 
Indeed, the numerical product does not map $\dot B^{\frac12}_{2,1}(\R)\times \dot B^{-\frac12}_{2,1}(\R)$ 
to $\dot B^{-\frac12}_{2,1}(\R).$  
\item it does not provide us with  uniform bounds in the high relaxation asymptotics
(see the beginning of Section \ref{s:relax} for more explanations).
\end{itemize}
\medbreak
Another possible choice  is  $s=d/2.$ Then,  the solution space becomes 
the set of $Z$ in $\cC^1_b(\R_+\times\R^d;\R^n)$ satisfying
$$Z^\ell\in \cC_b(\R_+;\dot B^{\frac d2}_{2,1})\cap L^1(\R_+;\dot B^{\frac d2+2}_{2,1})\andf
Z^h\in \cC_b(\R_+;\dot B^{\frac d2+1}_{2,1})\cap L^1(\R_+;\dot B^{\frac d2+1}_{2,1})$$
plus crucial informations from the damped mode that, in particular, will ensure that 
$$\nabla Z_2\in L^1(\R_+;\dot B^{\frac d2}_{2,1})\andf Z_2\in L^2(\R_+;\dot B^{\frac d2}_{2,1}).$$ 
This alternative framework allows to consider  initial data 
that are  less decaying at infinity (regularity  $\dot B^{\frac d2}_{2,1}$ for low frequencies
is less stringent than  $\dot B^{\frac d2-1}_{2,1}$),
to handle  the one-dimensional situation, and  to provide
crucial  uniform a priori bounds in the strong  relaxation limit. 
The only drawback is that  this alternative  framework  
requires seemingly  stronger structure assumptions on the system (that are nevertheless 
fulfilled by the compressible Euler equations).  
In order to specify them, let us rewrite System \eqref{eq:ZZ} by blocks as follows:
\begin{equation} \left\{ \begin{array}{l}\displaystyle \partial_tZ_1 + \sum_{k=1}^d\left(A_{11}^k(Z)\partial_{k}Z_1+A_{12}^k(Z)\partial_{k}Z_2\right)=0,\\ \displaystyle \partial_tZ_2 + \sum_{k=1}^d\left(A_{21}^k(Z)\partial_{k}Z_1+A_{22}^k(Z)\partial_{k}Z_2\right)+{L_2} Z_2=0. \end{array} \right. \label{GE}
\end{equation} 
Then, we need the following additional assumption:
\begin{enumerate}
\item[({\bf H3})]  For all $k\in\{1,\cdots,d\},$ $\bar A^k_{11}=0$ and $Z\mapsto A^k_{11}(Z)$ is linear
with respect to $Z_2.$
\end{enumerate}
Note that in the context of gas dynamics, the above assumption just means that there are no terms
like $\nabla \varrho$ or $\varrho\nabla\varrho$ in the density equation, which is indeed the case~!
 \begin{thm}\label{Thmd2}   In general dimension $d\geq1,$
 let the assumptions of Theorem \ref{ThmGlobal} concerning system \eqref{GEQSYM} be 
 in force and assume in addition that ({\bf H3}) holds true. 
  Then, there exists a positive constant $\alpha$ 
such that for all   $Z_0\in \wt {B}^{\frac{d}{2},\frac{d}{2}+1}_{2,1}$  satisfying 
 \begin{equation}\label{eq:smalldatabis}
\cZ'_0\triangleq  \|Z_0\|_{\wt{B}^{\frac{d}{2},\frac{d}{2}+1}_{2,1}} \leq \alpha,
\end{equation} System \eqref{GE}  supplemented with initial data $Z_0$  admits a unique global-in-time solution 
in the subspace $F$  of  functions $Z$ of $\mathcal{C}_b(\mathbb{R}_+;\wt{B}^{\frac{d}{2},\frac{d}{2}+1}_{2,1})$ such that
$$Z_2^\ell, Z^h\!\in\! L^1(\mathbb{R}_+;\dot{B}^{\frac{d}{2}+1}_{2,1}), \,\;\;\: Z_1^\ell\!\in\! L^1(\mathbb{R}_+;\dot{B}^{\frac{d}{2}+2}_{2,1})\!\andf\!\! \d_tZ_2\!\in\! L^1(\mathbb{R}_+;\dot{B}^{\frac{d}{2}}_{2,1}).$$
Moreover,   there exists an explicit  Lyapunov functional that is 
 equivalent to $\|Z\|_{\wt{B}^{\frac{d}{2},\frac{d}{2}+1}_{2,1}}$ 
and we have the following bound:
\begin{multline}\label{eq:Y}\cZ'(t)\leq C\cZ'_0
\quad\hbox{where}\quad
\cZ'(t)\triangleq\|Z\|_{L^\infty_t(\wt{B}^{\frac{d}{2},\frac{d}{2}+1}_{2,1})}+\|Z_1\|^\ell_{L^1_t(\dot{B}^{\frac{d}{2}+2}_{2,1})}\\
+\|Z_2\|^\ell_{L^1_t(\dot{B}^{\frac{d}{2}+1}_{2,1})}
+\|Z_2\|^\ell_{L^2_t(\dot{B}^{\frac{d}{2}}_{2,1})}+\|Z\|^h_{L^1_t(\dot{B}^{\frac{d}{2}+1}_{2,1})}+\|\d_tZ_2\|_{L^1_t(\dot{B}^{\frac{d}{2}}_{2,1})}.
\end{multline} \end{thm}
Theorem \ref{Thmd2} directly applies to the isentropic compressible Euler equations
with relaxation, written in terms of the sound speed $c$ and of $v$ 
(that is,  System \eqref{eq:eulerc}).
The result we get reads as follows: 
\begin{thm}\label{thm:euler} Let $\bar c>0,$ $d\geq1$ and $\gamma>1.$  
There exists a positive constant  $\alpha$ such that  for any  data $(c_0,v_0)$
such that $c_0-\bar c$ and $v_0$ belong to $\wt{B}^{\frac{d}{2},\frac{d}{2}+1}_{2,1},$ and satisfy
 \begin{equation}\label{eq:smalldataeuler}
\cA_0\triangleq  \|(c_0-\bar c,v_0)\|_{\wt{B}^{\frac{d}{2},\frac{d}{2}+1}_{2,1}} \leq \alpha,
\end{equation}
System  \eqref{eq:eulerc} with $\ep=1$ admits a unique global solution $(c,v)$ with 
$(c-\bar c,v)\in \mathcal{C}_b(\mathbb{R}_+;\wt{B}^{\frac{d}{2},\frac{d}{2}+1}_{2,1})$ that satisfies
\begin{multline}\label{eq:cv}\|(c-\bar c,v)\|_{L^\infty(\mathbb{R}_+;\wt{B}^{\frac{d}{2},\frac{d}{2}+1}_{2,1})}
+\|c-\bar c\|^\ell_{L^1(\mathbb{R}_+;\dot{B}^{\frac{d}{2}+2}_{2,1})}\\
+\|v\|^\ell_{L^1(\mathbb{R}_+;\dot{B}^{\frac{d}{2}+1}_{2,1})}
+\|v\|^\ell_{L^2(\mathbb{R}_+;\dot{B}^{\frac{d}{2}}_{2,1})}+\|(c-\bar c,v)\|^h_{L^1(\mathbb{R}_+;\dot{B}^{\frac{d}{2}+1}_{2,1})}+\|\d_tv\|_{L^1(\mathbb{R}_+;\dot{B}^{\frac{d}{2}}_{2,1})}
\leq C\cA_0.\end{multline}
\end{thm}
A similar statement  holds true for  the barotropic compressible Euler equations with general smooth pressure law 
$P$ satisfying $P'>0$ in the neighborhood of the reference density $\bar\varrho,$ 
 although one cannot `symmetrize'  the system any longer by using the sound speed. 
For more details, the reader may refer to  \cite{CBD2} where a  class of partially dissipative systems, more general than 
\eqref{GEQSYM}, is considered. 
 \medbreak  
In the rest of this section, we  focus on the proof of Theorem \ref{ThmGlobal}.  
The reader is  referred to \cite{CBD2} for more general systems   and for the proof of Theorem \ref{Thmd2}.  
A similar statement in the $L^p$ framework has been established in \cite{CBD3}.


\subsection{A priori estimates}

The overall strategy is to apply  the Littlewood-Paley truncation operator $\ddj$ to \eqref{eq:ZZ}, then to 
follow the method  that has been described in the previous section so as to get optimal estimates in $L^2$ for each 
dyadic block. Performing eventually a suitable weighted summation on $j$ will lead to the control 
of Besov norms of the solution, as stated in Theorem \ref{ThmGlobal}. 

Throughout, we assume that we are given a smooth and sufficiently decaying solution 
of \eqref{eq:ZZ} on $[0,T]\times\R^d$ such that 
\begin{equation}\label{eq:smallZ}
\sup_{t\in[0,T]}\|z(t)\|_{\dot B^{\frac d2}_{2,1}(\R^d)} \quad\hbox{is sufficiently small.}\end{equation}
We shall use repeatedly that, owing to the embedding \eqref{eq:embed}, the solution $Z$ is also small in $L^\infty([0,T]\times\R^d).$ 


\subsubsection*{$\bullet$ Low frequencies}

Let us denote    $Z_j\triangleq \ddj Z$  and $F_j\triangleq \ddj F,$ with 
\begin{equation}\label{eq:F} F\triangleq \sum_{k=1}^d (\bar A^k-A^k(Z))\d_kZ.\end{equation} 
We see that for all $j\in\Z,$ 
\begin{equation}\label{eq:ZjLF}\d_tZ_j+\sum_{k=1}^d \bar A^k\d_kZ_j +BZ_j=F_j.\end{equation}
Hence, taking the $L^2(\R^d;\R^n)$ scalar product with $Z_j$ and using  that the first order
terms are skew-symmetric yields
$$\frac 12\frac d{dt}\|Z_j\|_{L^2}^2 +\int_{\R^d} BZ_j\cdot Z_j\,dx=\int_{\R^d} Z_j\cdot F_j\,dx.$$
The term with $B$ may be bounded from below according to \eqref{eq:B22}. Hence, using Cauchy-Schwarz 
inequality for bounding the right-hand side delivers for some $c>0,$ 
$$\frac 12\frac d{dt}\|Z_j\|_{L^2}^2 + c\|Z_{2,j}\|_{L^2}^2 \leq\|F_j\|_{L^2}\|Z_j\|_{L^2}.$$
In order to recover the full  dissipation,  we proceed as in the previous section, introducing 
the functional $L_{r,\omega}$ defined in \eqref{eq:Lomega}. 
Adapting the computations therein to the case where the source term in \eqref{eq:Zlinear}
is nonzero, we get
for all $r>0$ and $\omega\in\S^{d-1}$ (with the notations of \eqref{eq:def}):
$$\displaylines{\frac d{dt} L_{r,\omega}(\wh Z_j)+\frac{\min(1,r^2)}{2}\sum_{k=0}^{n-1}\ep_k|B_\omega A_\omega^k\wh Z_j|^2
\leq \Re\bigl(\wh F_j\cdot\wh Z_j\bigr)\hfill\cr\hfill+\min(r,r^{-1})
\sum_{k=1}^{n-1}
 \Bigl(\Re\bigl(B_\omega A_\omega^{k-1}\wh Z_j\cdotp B_\omega A_\omega^k\wh F_j\bigr)
 + \Re\bigl(B_\omega A_\omega^{k-1}\wh F_j\cdotp B_\omega A_\omega^k\wh Z_j \bigr)\Bigr)\cdotp}$$
In light of Cauchy-Schwarz inequality, the sum in the right-hand side may
be bounded by $\sqrt{L_{r,\omega}(\wh Z_j) L_{r,\omega}(\wh F_j)}.$ 
Hence,   using that Condition (SK) and \eqref{eq:equivL} ensure the existence of
a positive constant $c_0>0$ such that for all $r>0$ and $\omega\in\S^{d-1},$ 
\begin{equation}\label{eq:SKL}
\sum_{k=0}^{n-1}\ep_k|B_\omega A_\omega^k\wh Z_j|^2\geq 2c_0  L_{r,\omega}(\wh Z_j),\end{equation}
we conclude that 
\begin{equation}\label{eq:dtLr}
\frac d{dt} L_{r,\omega}(\wh Z_j)+ c_0\min(1,r^2) L_{r,\omega}(\wh Z_j) \leq 
\sqrt{L_{r,\omega}(\wh Z_j) L_{r,\omega}(\wh F_j)}.\end{equation}
Let us denote  for all $j\in\Z,$ 
$$\cL_j\triangleq \|\wh Z_j\|_{L^2}^2+\sum_{k=1}^{n-1}\eps_k
\Re\int_{\R^d} (BA_\omega^{k-1}\wh Z_j(\xi))\cdot (BA_\omega^k\wh Z_j(\xi))\:\min(|\xi|,|\xi|^{-1})\,d\xi.$$
Integrating \eqref{eq:dtLr}  on $\R^d$ and observing that, by virtue of \eqref{eq:equivL}, we have
\begin{equation}\label{eq:equivLj}
\frac12\|Z_j\|_{L^2}^2\leq \cL_j\leq 2\|Z_j\|_{L^2}^2,\end{equation}
we get, up to a slight modification of $c_0,$ 
\begin{equation}\label{eq:Lj} 
\frac d{dt}\cL_j + c_0\min(1,2^{2j})\cL_j\leq C\|F_j\|_{L^2}\sqrt{\cL_j}. 
\end{equation}
Therefore, taking $X=\sqrt{\cL_j}$ and $A=  C\|F_j\|_{L^2}$ in Lemma \ref{SimpliCarre}, 
then multiplying by $2^{j(\frac d2-1)}$ delivers for all $j<0,$
$$2^{j(\frac d2-1)}\sqrt{\cL_j(t)} + \frac{c_0}2\:2^{j(\frac d2+1)}\int_0^t\sqrt{\cL_j}\,d\tau\leq 
2^{j(\frac d2-1)}\sqrt{\cL_j(0)} + C2^{j(\frac d2-1)}\int_0^t\|F_j\|_{L^2}\,d\tau.
$$
In order to bound the right-hand side, it suffices to combine 
the following  facts that are proved in e.g. \cite[Chap. 2]{BCD}: 
\begin{equation} \label{eq:num}
\hbox{For } d\geq2,\ \hbox{the numerical product maps}\  
\dot B^{\frac d2-1}_{2,1}(\R^d)\times\dot B^{\frac d2}_{2,1}(\R^d)\  \hbox{ to }\ \dot B^{\frac d2-1}_{2,1}(\R^d)\end{equation}
and, for all smooth function $\Phi:\R^d\to\R^p$ vanishing at $0,$
\begin{equation}\label{eq:compo}
\|\Phi(Z)\|_{\dot B^{\frac d2}_{2,1}}\leq C(\|Z\|_{L^\infty}) \|Z\|_{\dot B^{\frac d2}_{2,1}}.
\end{equation}
Hence, remembering also \eqref{eq:smallZ}, we conclude that
$$\|F\|_{\dot B^{\frac d2-1}_{2,1}} \lesssim \|Z\|_{\dot B^{\frac d2-1}_{2,1}}\|\nabla Z\|_{\dot B^{\frac d2}_{2,1}}.$$
So, finally,   there exist two positive constants $c_0$ and  $C$ such 
that for all $j<0,$ we have
\begin{multline}\label{eq:cLlf}
2^{j(\frac d2-1)}\sqrt{\cL_j(t)} + c_02^{j(\frac d2+1)}\int_0^t\sqrt{\cL_j}\,d\tau\leq 
2^{j(\frac d2-1)}\sqrt{\cL_j(0)} \\+ C\int_0^t c_j\|\nabla Z\|_{\dot B^{\frac d2}_{2,1}}\|Z\|_{\dot B^{\frac d2-1}_{2,1}}\,d\tau
\with \sum_{j\in\Z} c_j=1.
\end{multline}

\subsubsection*{$\bullet$ High frequencies}

In order to bound the high frequency part of the solution, we shall keep the functional 
$\cL_j$, but one cannot look  at $F$ defined in \eqref{eq:F}   as a source term 
since this would  entail a loss of one derivative. 
To overcome the difficulty, we mimic the proof of the $L^2$ estimate recalled 
at the beginning of this section, writing the system for $Z_j\triangleq \ddj Z$ as follows:  
\begin{equation}\label{eq:ZZj} 
\d_tZ_j+\sum_{k=1}^d A^k(Z)\d_kZ_j + BZ_j= \sum_{k=1}^d[A^k(Z),\ddj] \partial_kZ.
\end{equation}
Then, taking the $L^2(\R^d;\R^n)$ scalar product with $Z_j$ and integrating by parts yields:
$$\frac 12\frac d{dt}\|Z_j\|_{L^2}^2 +\int_{\R^d}\!BZ_j\cdot Z_j\,dx=\frac12\sum_{k,\ell,m}\int_{\R^d} \d_k(A^k_{\ell m}(Z)) 
Z_j^\ell\, Z_j^m\,dx
+\sum_{k=1}^d\int_{\R^d} [A^k(Z),\ddj] Z\cdot Z_j\,dx.$$
The last sum may be bounded according to the following classical 
commutator estimate (see e.g. \cite[Chap. 2]{BCD}) that is valid for all $s\in(-d/2,d/2+1]$: 
$$\|[A^k(Z),\ddj] \d_kZ\|_{L^2}\leq C c_j 2^{-js} \|\nabla(A^k(Z))\|_{\dot B^{\frac d2}_{2,1}} \|Z\|_{\dot B^s_{2,1}}\with 
\sum_{j\in\Z} c_j=1.$$
Thanks to the embedding \eqref{eq:embed} and to the definition of $\|\cdot\|_{\dot B^s_{2,1}},$
 we  have 
$$
\sum_{k,\ell,m}\int_{\R^d} \d_k(A^k_{\ell m}(Z)) 
Z_j^\ell\, Z_j^m\,dx \leq C c_j 2^{-js} \|\nabla(A^k(Z))\|_{\dot B^{\frac d2}_{2,1}} \|Z\|_{\dot B^s_{2,1}}\|Z_j\|_{L^2}.
$$
Hence,  owing to \eqref{eq:B22},  we have for all $s\in(-d/2,d/2+1],$ 
\begin{equation}\label{eq:Zjjj}
\frac 12\frac d{dt}\|Z_j\|_{L^2}^2 + c\|Z_{2,j}\|_{L^2}^2 \leq Cc_j 2^{-js}\|\nabla Z\|_{\dot B^{\frac d2}_{2,1}}
\|Z\|_{\dot B^s_{2,1}}\|Z_j\|_{L^2}.\end{equation}
To recover the full dissipation, one has to compute for all $r\geq1$ and $\omega\in\S^{d-1},$ 
the time derivative of 
$$r^{-1}\wt L_{r,\omega}(\wh Z_j)\with \wt L_{r,\omega}(\wh Z_j)\triangleq
\sum_{k=1}^{n-1}\eps_k\Re(B A_\omega^{k-1}\wh Z_j\cdot BA_\omega^k)$$
as it will generate the term $\sum_{k=1}^{n-1}\frac{\ep_k}2|B A_\omega^k\wh Z_j|^2,$  
that is, the missing dissipation. 
To proceed, one can keep $F$ defined in \eqref{eq:F} as a source
term and start from   \eqref{eq:ZjLF}. For $j\geq0,$
the term $r^{-1}$ yields the factor $2^{-j}$ that exactly compensates the loss of one derivative
when estimating $F$ in $\dot B^{\frac d2+1}_{2,1}.$ Hence, it suffices to estimate
$F$ in $\dot B^{\frac d2}_{2,1},$ which  may be done by combining \eqref{eq:compo} 
and the following fact:  
\begin{equation} \label{eq:num2}
\hbox{the numerical product maps}\  
\dot B^{\frac d2}_{2,1}(\R^d)\times\dot B^{\frac d2}_{2,1}(\R^d)\  \hbox{ to }\ \dot B^{\frac d2}_{2,1}(\R^d).\end{equation} 
Remembering \eqref{eq:smallZ}, we get
\begin{equation}\label{eq:FF}
\|F\|_{\dot B^{\frac d2}_{2,1}}\lesssim \|Z\|_{\dot B^{\frac d2}_{2,1}}\|\nabla Z\|_{\dot B^{\frac d2}_{2,1}}.\end{equation}
Now,   adding up the relation we get for $r^{-1}\wt L_{r,\omega}(\wh Z_j)$ (after space integration)
to \eqref{eq:Zjjj}  yields for all $j\geq0$:  
\begin{equation}\label{eq:Ljj}
\frac d{dt}\cL_j + \frac12\sum_{k=0}^{n-1}\int_{\R^d} \eps_k|B A^k_\omega Z_j(\xi)|^2\,d\xi
\lesssim \bigl(2^{-j}\|F_j\|_{L^2}
+ c_j 2^{-j(\frac d2+1)}
\|Z\|_{\dot B^{\frac d2+1}_{2,1}}^2\bigr)\|Z_j\|_{L^2},\end{equation}
and \eqref{eq:SKL}  guarantees  that
\begin{equation}\label{eq:Ljjj}
\sum_{k=0}^{n-1}\int_{\R^d} \eps_k|B A^k_\omega Z_j(\xi)|^2\,d\xi\gtrsim \cL_j.\end{equation}
Hence, using Lemma \ref{SimpliCarre}, multiplying by $2^{j(\frac d2+1)}$ and taking advantage of \eqref{eq:F}
and \eqref{eq:FF},  we end up with
\begin{multline}\label{eq:cLhf}
2^{j(\frac d2+1)}\sqrt{\cL_j(t)} + c2^{j(\frac d2+1)}\int_0^t\sqrt{\cL_j}\,d\tau\\\leq 
2^{j(\frac d2+1)}\sqrt{\cL_j(0)} + 
C\int_0^t c_j\|Z\|_{\dot B^{\frac d2+1}_{2,1}}\|Z\|_{\dot B^{\frac d2}_{2,1}\cap \dot B^{\frac d2+1}_{2,1}}\,d\tau.
\end{multline}

\subsubsection*{$\bullet$ Conclusion}
Let us put 
\begin{equation}\label{def:cL}
\cL\triangleq \sum_{j<0} 2^{j(\frac d2-1)}\sqrt{\cL_j} + \sum_{j\geq0} 2^{j(\frac d2+1)}\sqrt{\cL_j}\andf
\cH\triangleq\|Z\|_{\dot B^{\frac d2+1}_{2,1}}.
\end{equation}
Since  $\cL_j\simeq\|Z_j\|_{L^2}^2,$ we have the following equivalence:
\begin{equation}\label{eq:equi}
\cL\simeq \|Z\|_{\wt B^{\frac d2-1,\frac d2+1}_{2,1}}.
\end{equation}
Note that this implies that 
\begin{equation}\label{eq:compare}
\|Z\|_{\dot B^s_{2,1}}\lesssim\cL \quad\hbox{for all }\   \frac d2-1\leq s\leq \frac d2+1.\end{equation}
Hence, we deduce from \eqref{eq:cLlf} and \eqref{eq:cLhf} that
\begin{equation}\label{eq:cL0}\cL(t)+c\int_0^t\cH(\tau)\,d\tau \leq \cL(0) +C\int_0^t\cH(\tau)\cL(\tau)\,d\tau.
\end{equation}
We claim that there exists $\alpha>0$ such that if $\cL(0) < \alpha$ then, for all $t\in[0,T],$ we have
\begin{equation}\label{eq:cL}\cL(t)+\frac c2\int_0^t\cH(\tau)\,d\tau \leq \cL(0).
\end{equation}
Indeed, let us choose $\alpha\in(0,c/(2C))$ so that
$\cL\leq\alpha$ implies that \eqref{eq:smallZ} is satisfied, and set 
$$T_0\triangleq \sup \bigl\{T_1\in[0,T],\: \sup_{t\in[0,T_1]} \cL(t)\leq \alpha\bigr\}\cdotp$$
The above set is nonempty (as $0$ is in it) and contains its supremum since 
$\cL$ is continuous (remember that we assumed that $Z$ is smooth). Hence we have
$$\cL(T_0)+c\int_0^{T_0}\cH(\tau)\,d\tau\leq \cL(0)+ C\int_0^{T_0} \cH(\tau)\cL(\tau)\,d\tau
\leq \cL(0)+ \frac c2\int_0^{T_0} \cH(\tau)\,d\tau.$$
Using the smallness hypothesis on $\cL(0),$ one may conclude that $\cL< \alpha$ on $[0,T_0].$
As $\cL$ is continuous,  we must have $T_0=T$ and \eqref{eq:cL} thus holds on $[0,T].$ 
\smallbreak
Clearly, time $t=0$ does not play any particular role, and one can apply the same argument on any sub-interval
of $[0,T],$ which leads to:
\begin{equation}\label{eq:cLt0}\cL(t)+\frac c2\int_{t_0}^t\cH(\tau)\,d\tau \leq \cL(t_0),\qquad
0\leq t_0\leq t\leq T.
\end{equation}
Hence, provided that  $\|Z_0\|_{\wt B^{\frac d2-1,\frac d2+1}_{2,1}},$   is small enough, $\cL$ is a Lyapunov functional that is, in light of  \eqref{eq:equi}, equivalent  to  $\|Z\|_{\wt B^{\frac d2-1,\frac d2+1}_{2,1}}.$


\subsection{The damped mode}

Define $W$  by the  relation:
$$\d_tZ_2+L_2W=0.$$
Since $L_2$ is invertible,  the second line of \eqref{GE} yields 
\begin{equation}\label{eq:WWW}
W=Z_2+L_2^{-1}\sum_{k=1}^d\bigl(A_{21}^k(Z)\d_kZ_1+A_{22}^k(Z)\d_kZ_2\bigr),
\end{equation}
which allows to get the following equation for $W$: 
\begin{equation}\label{eq:dtW}\d_tW +L_2W = L_2^{-1}\biggl(\sum_{k=1}^d
\d_t(A_{21}^k(Z)\d_kZ_1) +\d_t(A_{22}^k(Z))\d_kZ_2 - A_{22}^k(Z) L_2\d_kW\biggr)\cdotp\end{equation}
Applying $\ddj$ to the above relation and denoting $W_j\triangleq \ddj W$ leads to 
$$\displaylines{\d_tW_j +L_2W_j = L_2^{-1}\biggl(\sum_{k=1}^d
\ddj\d_t((A_{21}^k(Z)-\bar A_{21}^k)\d_kZ_1) + \bar A_{21}^k\d_t\d_kZ_{1,j} \hfill\cr\hfill+\ddj(\d_t(A_{22}^k(Z))\:\d_kZ_2) +\ddj\bigl((\bar A_{22}^k-A_{22}^k(Z))W\bigr)
- \bar A_{22}^k L_2\d_kW_j\biggr)\cdotp}$$
Using \eqref{eq:B22}, an  energy method and Lemma \ref{SimpliCarre}, 
we get two positive constants $c$ and $C$ such that for all $j\in\Z$ and $t\in[0,T],$
\begin{multline}\label{eq:Wjj}
\|W_j(t)\|_{L^2}+c\int_0^t\|W_j\|_{L^2} \leq \|W_j(0)\|_{L^2}\\
+ C\int_0^t\sum_{k=1}^d\Bigl(\|\ddj\d_t((A_{21}^k(Z)-\bar A_{21}^k)\d_kZ_1)\|_{L^2}
+\| \bar A_{21}^k\d_t\d_kZ_{1,j}\|_{L^2}\\
+ \|\ddj(\d_t(A_{22}^k(Z)\:\d_kZ_2)\|_{L^2}+ \|\ddj\bigl((\bar A_{22}^k-A_{22}^k(Z))W\bigr)\|_{L^2}
+\|L_2^{-1}\bar A_{22}^k L_2\d_kW_j\|_{L^2}\Bigr)\,d\tau.
\end{multline}
 Bernstein inequality \eqref{eq:bernstein2}  guarantees that 
$$
\|L_2^{-1}\bar A_{22}^k L_2\d_kW_j\|_{L^2}\leq C2^j\|W_j\|_{L^2}.$$
Hence, there exists $j_0\in\Z$ such that for all  $j\leq j_0,$ the last term may be absorbed by the 
time integral of the left-hand side. 
\smallbreak
Next,  using \eqref{GE} to compute the time derivatives, we see that the terms with 
$$\d_t((A^k_{21}(Z)-\bar A^k_{21})\d_kZ_1)\quad\hbox{or}\quad \d_t(A_{22}^k(Z))\,\d_kZ_2$$  are linear combinations
of coefficients  of type $K(Z)\,Z\otimes\nabla^2Z,$  $K(Z)\,Z\otimes\nabla Z$ 
and $K(Z)\,\nabla Z\otimes\nabla Z$ for suitable smooth functions $K.$ 
Hence, using  \eqref{eq:num}, \eqref{eq:compo} and remembering \eqref{eq:smallZ} yields 
$$\|\d_t((A_{21}^k(Z)-\bar A_{21}^k)\d_kZ_1\|_{\dot B^{\frac d2-1}_{2,1}} 
+ \|\d_t(A_{22}^k(Z))\:\d_kZ_2\|_{\dot B^{\frac d2-1}_{2,1}}
\leq C  \bigl( \|Z\|_{\dot B^{\frac d2}_{2,1}}\|\nabla Z\|_{\dot B^{\frac d2}_{2,1}}
+ \|Z\|_{\dot B^{\frac d2}_{2,1}}^2\bigr)\cdotp$$
To handle  $\|\ddj\bigl((\bar A_{22}^k-A_{22}^k(Z))W\bigr)\|_{L^2},$ we split $W$ into low and high frequencies. 
For the low frequency part, we just write that by composition \eqref{eq:compo} and product law \eqref{eq:num}, 
$$\|\ddj\bigl((\bar A_{22}^k-A_{22}^k(Z))W^\ell\bigr)\|_{L^2}\leq  Cc_j 2^{-j(\frac d2-1)} \|Z\|_{\dot B^{\frac d2}_{2,1}}
\|W\|_{\dot B^{\frac d2-1}_{2,1}}^\ell.$$
For the high frequency part, we further decompose $W$ as follows (in light of \eqref{eq:WWW}):
\begin{multline}\label{eq:WWWW}
W=Z_2+L_2^{-1}\sum_{k=1}^d\bigl(\bar A_{21}^k\d_kZ_1+\bar A_{22}^k\d_kZ_2\bigr)
\\+L_2^{-1}\sum_{k=1}^d\bigl((A_{21}^k(Z)-\bar A_{21}^k)\d_kZ_1+(A_{22}^k(Z)-\bar A_{22}^k)\d_kZ_2\bigr),
\end{multline}
which allows to get
\begin{equation}\label{eq:Wh} 
\|W\|^h_{\dot B^{\frac d2}_{2,1}} \leq \|Z_2\|^h_{\dot B^{\frac d2}_{2,1}}+C\|\nabla Z\|^h_{\dot B^{\frac d2}_{2,1}}
+C\|Z\|_{\dot B^{\frac d2}_{2,1}}\|\nabla Z\|_{\dot B^{\frac d2}_{2,1}},\end{equation}
whence, using \eqref{eq:num2},  
$$\begin{aligned}
\|(\bar A_{22}^k-A_{22}^k(Z))W^h\|_{\dot B^{\frac d2}_{2,1}}&\lesssim \|Z\|_{\dot B^{\frac d2}_{2,1}}
\|W^h\|_{\dot B^{\frac d2}_{2,1}}\\&\lesssim
 \|Z\|_{\dot B^{\frac d2}_{2,1}}\bigl( \|Z_2\|^h_{\dot B^{\frac d2}_{2,1}}+\|\nabla Z\|^h_{\dot B^{\frac d2}_{2,1}}
+\|Z\|_{\dot B^{\frac d2}_{2,1}}\|\nabla Z\|_{\dot B^{\frac d2}_{2,1}}\bigr),\end{aligned}$$
and thus
$$\|\ddj\bigl((\bar A_{22}^k-A_{22}^k(Z))W^h\bigr)\|_{L^2}\lesssim c_j 2^{-j\frac d2}
\|Z\|_{\dot B^{\frac d2}_{2,1}} \bigl(1+\|Z\|_{\dot B^{\frac d2}_{2,1}}\bigr)\|\nabla Z\|_{\dot B^{\frac d2}_{2,1}}\with 
\sum_{j\in\Z} c_j=1.
$$
Plugging this information in \eqref{eq:Wjj}, multiplying by $2^{j(\frac d2-1)},$ 
summing up on $j\leq j_0$ and remembering that $\|Z\|_{\dot B^{\frac d2}_{2,1}}$ is small,   
we conclude that\footnote{Handling the intermediate frequencies
corresponding to $j_0\leq j<0$ may be done from \eqref{eq:cLlf} since, then, 
$2^{j(\frac d2+1)}\simeq 2^{j(\frac d2-1)}$ and $\|W_j\|_{L^2}\lesssim \|Z_j\|_{L^2}.$}
$$\displaylines{
\|W(t)\|_{\dot B^{\frac d2-1}_{2,1}}^\ell + \frac c2\int_0^t \|W\|_{\dot B^{\frac d2-1}_{2,1}}^\ell \,d\tau
\leq \|W_0\|_{\dot B^{\frac d2-1}_{2,1}}^\ell \hfill\cr\hfill+C \int_0^t \bigl(\|Z\|_{\dot B^{\frac d2}_{2,1}}\|\nabla Z\|_{\dot B^{\frac d2}_{2,1}}+\|Z\|_{\dot B^{\frac d2}_{2,1}}^2\bigr)\,d\tau
+C\int_0^t\|Z\|_{\dot B^{\frac d2+1}_{2,1}}^\ell\,d\tau.}$$
Since $$\|Z\|_{\dot B^{\frac d2}_{2,1}}\lesssim  \|Z\|_{\dot B^{\frac d2-1}_{2,1}}^\ell +  \|Z\|_{\dot B^{\frac d2+1}_{2,1}}^h$$ 
and 
$$\|Z\|_{\dot B^{\frac d2}_{2,1}}^2\lesssim  \|Z\|_{\dot B^{\frac d2-1}_{2,1}} \|Z\|_{\dot B^{\frac d2+1}_{2,1}},$$
taking advantage of \eqref{eq:cL} eventually yields :
\begin{equation}\label{eq:West}
\|W(t)\|_{\dot B^{\frac d2-1}_{2,1}}^\ell + c\int_0^t \|W\|_{\dot B^{\frac d2-1}_{2,1}}^\ell \,d\tau
\leq C\cL(0)\quad\hbox{for all } \ t\in[0,T].\end{equation}
Owing to \eqref{eq:Wh} and \eqref{eq:cL}, the high frequencies of $W$  also satisfy 
\begin{equation}\label{eq:WHF}
\|W(t)\|_{\dot B^{\frac d2-1}_{2,1}}^h + c\int_0^t \|W\|_{\dot B^{\frac d2-1}_{2,1}}^h \,d\tau
\leq C\cL(0)\quad\hbox{for all } \ t\in[0,T],\end{equation}
which  completes the proof of  \eqref{eq:X0}. 

\subsection{Proving Theorem \ref{ThmGlobal}}

Having the a priori estimates \eqref{eq:cL}, \eqref{eq:West} and \eqref{eq:WHF} at hand, 
constructing a global solution obeying Inequality  \eqref{eq:X0} for any data $Z_0$ satisfying 
\eqref{eq:smalldata} follows from rather standard arguments. 
First, in  order to benefit from the classical theory on first order hyperbolic systems, 
we remove the low frequency part of $Z_{0}$ so as to have an 
initial data in the  \emph{nonhomogeneous}  Besov space $B^{\frac d2+1}_{2,1}.$
More precisely, we set  for all $n\in\N,$ 
\begin{equation}\label{def:Zn}Z_0^n\triangleq (\Id-\dot S_n) Z_0\with \dot S_n\triangleq \chi(2^{-n}D).\end{equation}
In light of e.g. \cite[Chap. 4]{BCD}, we 
 get a unique maximal solution $Z^n$ in $$\cC([0,T^n); B^{\frac d2+1}_{2,1})
\cap \cC^1([0,T^n); B^{\frac d2}_{2,1}).$$
Since $B^{\frac d2}_{2,1}$ is embedded in $\dot B^{\frac d2}_{2,1},$ it is easy to prove from 
\eqref{GE} and the composition and product laws \eqref{eq:num} and \eqref{eq:compo} that
$\d_t Z_1$  and $(\d_t Z_2-L_2Z_2)$ are in $L^\infty_{loc}(0,T^n;\dot B^{\frac d2-1}_{2,1}),$
and as  $Z_{0}^n$ belongs to $\dot B^{\frac d2-1}_{2,1},$ we deduce that 
$Z^n$ is actually in $\cC([0,T^n); B^{\frac d2+1}_{2,1}\cap B^{\frac d2-1}_{2,1}),$ hence obeys 
\eqref{eq:cL} for all $t\in[0,T^n).$ 
In particular, the embedding $\dot B^{\frac d2}_{2,1}\hookrightarrow L^\infty$ guarantees that 
$$
\int_0^{T^n} \|\nabla Z^n\|_{L^\infty}\,dt <\infty,
$$
and thus the standard continuation criterion for first order hyperbolic symmetric systems
(again, refer to  e.g. \cite[Chap. 4]{BCD}) ensures that $T^n=\infty.$ In other words,  for
all $n\in\N,$ the function $Z^n$ is a global solution of \eqref{GE} that satisfies 
 \eqref{eq:cL}, \eqref{eq:West} and \eqref{eq:WHF}  for all $t\in\R_+.$ 
 Note that, owing to the definition \eqref{def:Zn}, we have
 $$
 \|Z_0^n\|_{\wt B^{\frac d2,\frac d2+1}_{2,1}}\leq   \|Z_0\|_{\wt B^{\frac d2,\frac d2+1}_{2,1}},\qquad n\in\N.$$
 Hence $(Z^n)_{n\in\N}$ is a sequence of global smooth solutions that is 
 bounded in the space $E$ of Theorem \ref{ThmGlobal}. 
\smallbreak
Proving the convergence of $(Z^n)_{n\in\N}$ relies on the following proposition 
that can be easily proved by writing out the system satisfied by the difference of two 
solutions $Z$ and $Z'$ of \eqref{GE},  namely, 
$$\d_t(Z-Z')+\sum_{k=1}^d A^k(Z)\cdotp \d_k(Z-Z') + B(Z-Z')= \sum_{k=1}^d (A^k(Z')-A^k(Z))\d_k Z', $$
applying the Littlewood-Paley truncation operator
$\ddj$ to this system then arguing as for getting \eqref{eq:Zjjj} and using product
laws (see the details in \cite[Prop. 2]{CBD2}):
\begin{prop}\label{p:stability}
 Consider two  solutions $Z$ and $Z'$ of \eqref{GE} in the space $E$ corresponding
to small enough initial data $Z_0$ and $Z'_0$ in $\wt B^{\frac d2-1,\frac d2+1}_{2,1}.$  Then we have 
for all $t\geq0,$
$$\|(Z-Z')(t)\|_{\dot B^{\frac d2}_{2,1}}\leq \|Z_0-Z'_0\|_{\dot B^{\frac d2}_{2,1}}
+C\int_0^t\bigl(\|Z\|_{\wt B^{\frac d2,\frac d2+1}_{2,1}}+\|Z'\|_{\wt B^{\frac d2,\frac d2+1}_{2,1}}\bigr)d\tau.$$
\end{prop}
From this proposition (applied to $Z^n$ and $Z^m$ for any $(n,m)\in\N^2$), Gronwall lemma
and  the definition of the initial data in \eqref{def:Zn}, we  gather that $(Z^n)_{n\in\N}$
is a Cauchy sequence in the space $\cC_b(\R_+;\dot B^{\frac d2}_{2,1}),$ 
hence converges to some function $Z$ in $\cC_b(\R_+;\dot B^{\frac d2}_{2,1}).$ 
As the regularity  is high, passing to the limit in the system is not an issue, and one can easily conclude
that $Z$  satisfies \eqref{GE} supplemented with data $Z_0.$

That $Z$  belongs to the smaller space $E$  stems from standard functional analysis. Typically, one uses that all the Besov spaces under consideration satisfy the Fatou 
property, that is, for instance
$$\|Z\|_{\wt B^{\frac d2-1,\frac d2+1}_{2,1}}\leq C\liminf \|Z^n\|_{\wt B^{\frac d2-1,\frac d2+1}_{2,1}}.$$
  The only property that is missing is the time continuity 
with range  in $\dot B^{\frac d2+1}_{2,1}.$ However, this   is known to be true
for general quasilinear symmetric systems (see e.g.  \cite[Chap. 4]{BCD}). 
\medbreak
Finally, the uniqueness follows from  Proposition \ref{p:stability}.


\section{Decay estimates and asymptotic behavior}\label{s:decay}

The global-in-time properties of integrability for the solution $Z$ that have been proved  
so far   ensure that $Z(t)$ tends to $0$ in the tempered distributional meaning 
when $t$ goes to $\infty.$
In the present section, we aim at specifying the decay rate for some Besov norms of $Z,$
whenever the initial data satisfy a (mild) additional  condition. 
In the pioneering works
by the Japanese school in the 70ies and early 80ies  (see e.g. \cite{MN,SK}),  it  was expressed in terms of
Lebesgue spaces $L^p$ for some $p\in[1,2).$  However, it is well understood now that it suffices
to prescribe  this condition in some  homogeneous Besov spaces with a negative regularity index.
\smallbreak
In order to understand how  those  spaces come into play, looking first at the  linearized 
at the linearized system \eqref{eq:Zlinear} with no source term is very informative.
Let $Z$ be  the corresponding solution. Using 
 \eqref{eq:whZ} and Fourier-Plancherel theorem yields for all $t\geq0$: 
\begin{equation}\label{eq:decayZj}
\|Z_j(t)\|_{L^2}\leq C\|Z_j(0)\|_{L^2} e^{-c\min(1,2^{2j})t},\quad j\in\Z.\end{equation}
This means that the high frequencies of $Z$ decay to $0$ exponentially fast, 
and that the low frequencies behave as those  of the heat flow. 
More precisely,   for all $\alpha\geq0,$ we have
$$({2^{2j}t})^{\alpha/2}\|Z_j(t)\|_{L^2}\leq C(2^{2j}t)^{\alpha/2} e^{-c2^{2j}t}\|Z_j(0)\|_{L^2},\qquad
t\geq0,\quad j<0.$$
Hence, since the function $x^{\alpha/2}e^{-x}$ is bounded on $\R_+,$
we eventually get for all $s\in\R,$
\begin{equation}\label{eq:decaylflinear}t^{\alpha/2}\|Z(t)\|_{\dot B^{s+\alpha}_{2,1}}^\ell\leq C_\alpha \|Z_0\|_{\dot B^s_{2,1}}^\ell,\qquad t\geq0.\end{equation}
We note that, as for  the free heat equation,  in order to obtain some decay for the low frequencies, 
 a shift a regularity is needed. This is  the reason why it is wise 
 to make an additional assumption (e.g. some negative regularity) on the initial data to eventually 
 get some decay rate   for the norms we considered before for the  global solutions to \eqref{eq:ZZ}. 
 In fact, to compare our results with the classical ones in the literature, one can  introduce
 another family of homogeneous Besov spaces, namely the sets $\dot B^s_{2,\infty}$ of
 tempered distributions $z$ on $\R^d$ satisfying 
 \begin{equation}\label{eq:Binfty}
\|z\|_{\dot B^{s}_{2,\infty}} \triangleq  \sup_{j\in\Z} 2^{js}\|\ddj z\|_{L^2} <\infty\andf
\lim_{j\to-\infty} \|\chi(2^{-j}D)z\|_{L^\infty} =0.\end{equation}
 Owing to the critical embedding 
$$\|z\|_{\dot B^{\frac d2-\frac dp}_{2,\infty}(\R^d)}\lesssim \|z\|_{L^p(\R^d)},\qquad 1\leq p\leq2,$$
making   assumptions in  spaces $\dot B^s_{2,\infty}$ with a negative $s$ is \emph{weaker}
than in  the pioneering works  on decay estimates \cite{MN} where the initial data were 
assumed to be  in $L^1$ (this corresponds to the endpoint value $\sigma_1=d/2$)
or (see  \cite{SK})   in  $L^p$ for some $p\in[1,2)$  (take $\sigma_1=d/p-d/2$). 
\medbreak
 This motivates the following statement that we shall prove in the rest of the section:
 \begin{thm}\label{Thm:decay}  Let the assumptions of Theorem \ref{ThmGlobal} be in force, and assume in addition 
that $Z_0\in\dot B^{-\sigma_1}_{2,\infty}$   for some $\sigma_1$ in $(1-d/2,d/2].$
Let $\alpha_1\triangleq(\sigma_1+d/2-1)/2$ and $c_0\triangleq (\|Z_0\|_{\dot B^{-\sigma_1}_{2,\infty}}+  \|Z_0\|_{\wt B^{\frac d2-1,\frac d2+1}_{2,1}})^{-1/\alpha_1}.$

Then, the global solution $Z$ constructed in Theorem \ref{ThmGlobal} also belongs
to $L^\infty(\R_+;\dot B^{-\sigma_1}_{2,\infty}),$ and
 there exists a constant $C_0$ that may be computed in terms of $c_0$  
such  that 
$$\displaylines{(1+c_0t)^{\alpha_1}  \|Z(t)\|_{\dot B^{\frac d2-1}_{2,1}}^\ell +  (1+c_0t)^{2\alpha_1}
\bigl(\|Z(t)\|_{\dot B^{\frac d2+1}_{2,1}}^h+  \|\d_tZ_2(t)\|_{\dot B^{\frac d2-1}_{2,1}}^\ell\bigr)
\leq C_0  \|Z_0\|_{\wt B^{\frac d2-1,\frac d2+1}_{2,1}}. }$$
\end{thm} 
\begin{remark}  Under the  (stronger) structure assumptions of Theorem \ref{Thmd2}, one can prove a similar result
assuming only that $\sigma_1$ is in wider range $(-d/2,d/2].$ The inequality we eventually get is 
$$(1+c_0t)^{\alpha_1}  \|Z(t)\|_{\dot B^{\frac d2}_{2,1}}^\ell +  (1+c_0t)^{2\alpha_1}
\|Z(t)\|_{\dot B^{\frac d2+1}_{2,1}}^h+  (1+c_0t)^{2\alpha_1} \|\d_tZ_2(t)\|_{\dot B^{\frac d2}_{2,1}}^\ell
\leq C_0  \|Z_0\|_{\wt B^{\frac d2,\frac d2+1}_{2,1}}$$
with $\alpha_1\triangleq (\sigma_1+d/2)/2$ and  $c_0\triangleq (\|Z_0\|_{\dot B^{-\sigma_1}_{2,\infty}}+  \|Z_0\|_{\wt B^{\frac d2,\frac d2+1}_{2,1}})^{-1/\alpha_1}.$
\end{remark}
\begin{remark}
Even though the  negative Besov space assumption 
is weaker than in e.g. \cite{SK}, the obtained decay rates  are the same ones. 
Note also that    $\|Z_0\|_{\dot B^{-\sigma_1}_{2,\infty}}$ 
can be arbitrarily large: only $\|Z_0\|_{\wt B^{\frac d2-1,\frac d2+1}_{2,1}}$ has to be small.  
\smallbreak
The linear decay rate for low frequencies turns out to be  the correct one for the  solution
of the nonlinear system \eqref{eq:ZZ}, and 
 better (algebraic) decay rates  hold true for the high frequencies and for the damped mode. 
At the same time, although  the high frequencies of the solution of the linearized system \eqref{eq:Zlinear} have exponential decay, 
it is not the case for the nonlinear system \eqref{eq:ZZ}
owing to the coupling between the low and high frequencies through the nonlinear terms.
We do not claim optimality of  the above  decay rates  for the high
frequencies but, for sure, it is very unlikely  that they are exponential even for very particular initial data. 
\end{remark}
Let us briefly explain the general strategy  of the proof. 
The starting point  is to  show that the additional  
negative regularity is propagated for all time (with a time-independent control).
Then, we shall  combine it with  Inequality \eqref{eq:cLt0} and   an interpolation 
argument so as to exhibit a decay inequality for $\|Z(t)\|_{\dot B^{\frac d2-1}_{2,1}}.$
The rate that we shall get  in this way turns out to  be precisely the one  that was  expected from our linear analysis in \eqref{eq:decaylflinear}. 
Then, interpolating with the estimate in the negative space will enable us to capture optimal decay 
rates for  intermediate norms $\|Z(t)\|_{\dot B^{s}_{2,1}}.$ 

To the best of our knowledge the idea of combining a Lyapunov inequality with dissipation
and interpolation to get  (optimal) decay rates originates from  the work by J. Nash on parabolic 
equations in \cite{Nash}\footnote{Special thanks to L.-M. Rodrigues for pointing
out this reference to us.}.  Implementing  it  on  other equations
in a functional framework close to ours is rather recent. 
The overall strategy  is well explained in a  work by Y. Guo and Y. Wang \cite{GY}  devoted
to the Boltzmann equation and the compressible Navier-Stokes equations
in the  Sobolev spaces  setting, 
and  Z. Xin and J. Xu  in \cite{XX}  used the same method  to prove
decay estimates for the compressible Navier-Stokes equations in the critical regularity framework.
In the context of partially dissipative systems, the idea of prescribing additional integrability in terms of  negative Besov norms instead of Lebesgue ones seems to originate from  a paper by J. Xu and S. Kawashima \cite{XK3}. 

Finally, let us emphasize that it is possible to do without a Lyapunov functional (like we did in e.g. \cite{DX1}) but, somehow,  the proof is
more technical and less `elegant'.

\subsection{Propagation of negative regularity}

In order to prove that the  regularity in $\dot B^{-\sigma_1}_{2,\infty}$ is propagated for all time, 
 let  us start from the equation of $Z_j$ written in the following way:
$$\d_tZ_j+\sum_{k=1}^dA_k(Z)\d_kZ_j+BZ_j=\sum_{k=1}^d[A_k(Z),\ddj]\d_kZ.$$
Taking the $L^2$ scalar product with $Z_j$ and using  \eqref{eq:B22}  yields
\begin{equation}\label{eq:Zjs}
\frac12\frac {d}{dt}\|Z_j\|_{L^2}^2+c\|Z_{2,j}\|_{L^2}^2\leq \sum_{k=1}^d\|[A_k,\ddj]Z\|_{L^2}\|Z_j\|_{L^2}.\end{equation}
One can show (combine the commutator inequalities of  \cite[Chap. 2]{BCD} with \eqref{eq:compo}) that
$$\sup_{j\in\Z} 2^{-j\sigma_1}\|[A_k(Z),\ddj]\d_kZ\|_{L^2}\leq C\|\nabla Z\|_{\dot B^{\frac d2}_{2,1}}\|Z\|_{\dot B^{-\sigma_1}_{2,\infty}}
\quad\hbox{if }\ -\frac d2\leq -\sigma_1<\frac d2+1.$$
Hence, dropping the nonnegative term in the left-hand side of \eqref{eq:Zjs},  using Lemma \ref{SimpliCarre} and taking the supremum on $j$ yields 
$$\|Z(t)\|_{\dot B^{-\sigma_1}_{2,\infty}} \leq \|Z_0\|_{\dot B^{-\sigma_1}_{2,\infty}}+C\int_0^t\|\nabla Z\|_{\dot B^{\frac d2}_{2,1}}
\|Z\|_{\dot B^{-\sigma_1}_{2,\infty}}\,d\tau,\qquad t\geq0,
$$
which, after applying Gronwall lemma, leads to 
$$\|Z(t)\|_{\dot B^{-\sigma_1}_{2,\infty}} \leq \|Z_0\|_{\dot B^{-\sigma_1}_{2,\infty}}\exp\Bigl(C\int_0^t\|\nabla Z\|_{\dot B^{\frac d2}_{2,1}}\,d\tau\Bigr)\cdotp$$
Whenever $Z_0$ satisfies \eqref{eq:smalldata},  the global solution of Theorem \ref{ThmGlobal} has (small) 
gradient in $L^1(\R_+;\dot B^{\frac d2}_{2,1}).$
Hence the above inequality guarantees that $Z$ is uniformly bounded in $\dot  B^{-\sigma_1}_{2,\infty}$: 
there exists a constant $C$ depending only on $\sigma_1$ and such that
\begin{equation}\label{eq:Zsigma1}
\sup_{t\geq0}\|Z(t)\|_{\dot B^{-\sigma_1}_{2,\infty}} \leq C\|Z_0\|_{\dot B^{-\sigma_1}_{2,\infty}}.
\end{equation}


\subsection{Decay estimates for the whole solution}

The starting point is Inequality  \eqref{eq:cLt0} that is valid for all $0\leq t_0\leq t,$
and the fact that 
$$
 \cL\simeq\sum_{j\in\Z}2^{j(\frac d2-1)+\max(1,2^{2j})}\|Z_j\|_{L^2}
 \andf \cH=\|Z\|_{\dot B^{\frac d2+1}_{2,1}}.
 $$
 Being monotonous,  the function  $\cL$ is almost everywhere differentiable on $\R_+$ and 
 Inequality \eqref{eq:cLt0}  thus implies that
\begin{equation}\label{eq:LH}\frac d{dt}\cL +c\cH\leq 0\quad\hbox{a.e. on}\ \R_+.\end{equation}
Now, if $-\sigma_1<d/2-1,$ then one may use the following interpolation inequality:
$$\|Z\|_{\dot B^{\frac d2-1}_{2,1}}^\ell\lesssim
\bigl(\|Z\|^\ell_{\dot B^{\frac d2+1}_{2,1}}\bigr)^{1-\theta_0} 
\bigl(\|Z\|^\ell_{\dot B^{-\sigma_1}_{2,\infty}}\bigr)^{\theta_0}\with 
(1-\theta_0)\Bigl(1+\frac d2\Bigr)-\sigma_1\theta_0=\frac d2-1$$
which implies,  taking advantage of \eqref{eq:Zsigma1}, that
\begin{equation}\label{eq:decay1}
\|Z(t)\|_{\dot B^{\frac d2+1}_{2,1}}^\ell \gtrsim  
\bigl(\|Z(t)\|_{\dot B^{\frac d2-1}_{2,1}}^\ell\bigr)^{\frac1{1-\theta_0}}
\|Z_0\|_{\dot B^{-\sigma_1}_{2,\infty}}^{-\frac{\theta_0}{1-\theta_0}}.
\end{equation}
To handle the high frequencies of $Z,$ we just write that, owing to \eqref{eq:X0}, we have
\begin{equation}\label{eq:decay2}
\|Z(t)\|_{\dot B^{\frac d2+1}_{2,1}}^h \gtrsim  
\bigl(\|Z(t)\|_{\dot B^{\frac d2+1}_{2,1}}^h\bigr)^{\frac1{1-\theta_0}}
\|Z_0\|_{\wt B^{\frac d2-1,\frac d2+1}_{2,1}}^{-\frac{\theta_0}{1-\theta_0}}.
\end{equation}
Putting \eqref{eq:decay1} and \eqref{eq:decay2} together 
and remembering that 
\begin{equation}\label{eq:equiv} \cL\simeq \|Z\|_{\dot B^{\frac d2-1}_{2,1}}^\ell
+\|Z\|_{\dot B^{\frac d2+1}_{2,1}}^h,\end{equation} 
one may thus write that for a small enough $c>0,$ we have
$$\cH\gtrsim c_0 \cL^{\frac1{1-\theta_0}} \with
c_0\triangleq 
c\bigl(\|Z_0\|_{\dot B^{-\sigma_1}_{2,\infty}}+\|Z_0\|_{\wt B^{\frac d2-1,\frac d2+1}_{2,1}}\bigr)^{-\frac{\theta_0}{1-\theta_0}}\cdotp$$
Reverting to \eqref{eq:LH}, one eventually obtains the following differential inequality:
$$\frac d{dt}\cL +c_0\cL^{\frac1{1-\theta_0}}\leq 0,$$
which readily  leads to 
\begin{equation}\label{eq:decaybf}\cL(t)\leq (1+c_0t)^{1-1/\theta_0}\cL(0).\end{equation}
Now, replacing $\theta_0$ with its value, and using \eqref{eq:equiv},
 one can conclude that
\begin{equation}\label{eq:decaylf}  \|Z(t)\|_{\wt B^{\frac d2-1,\frac d2+1}_{2,1}}
\lesssim (1+c_0t)^{-\alpha_1}\|Z_0\|_{\wt B^{\frac d2-1,\frac d2+1}_{2,1}}\with \alpha_1\triangleq\frac12\biggl(\sigma_1+\frac d2-1\biggr)\cdotp\end{equation}
  As regards the low frequencies of
 the solution,  this decay is  consistent with \eqref{eq:decaylflinear} in the case $s=-\sigma_1$ and 
 $\alpha=\sigma_1+d/2-1.$


\subsection{High  frequency decay}
{}From \eqref{eq:Ljj} and \eqref{eq:Ljjj}, we gather 
$$\frac12 \frac d{dt}\cL_j+c\cL_j\lesssim 2^{-j}\|F_j\|_{L^2}\|Z_j\|_{L^2}+
c_j2^{-j(\frac d2+1)} \|Z\|_{\dot B^{\frac d2+1}_{2,1}}^2\sqrt{\cL_j}
\with \sum_{j\geq0} c_j=1.$$
Hence, bounding $F_j$ according to \eqref{eq:FF} yields
$$\frac12 \frac d{dt}\cL_j+c\cL_j\lesssim c_j2^{-j(\frac d2+1)} 
\|Z\|_{\dot B^{\frac d2+1}_{2,1}}\bigl(\|Z\|_{\dot B^{\frac d2}_{2,1}}+\|Z\|_{\dot B^{\frac d2+1}_{2,1}}\bigr)\sqrt{\cL_j}
\with \sum_{j\geq0} c_j=1,$$
whence
$$\frac12\frac d{dt} (e^{ct}\sqrt{\cL_j})^2\lesssim c_j2^{-j(\frac d2+1)} 
\|Z\|_{\dot B^{\frac d2+1}_{2,1}}\|Z\|_{\wt B^{\frac d2,\frac d2+1}_{2,1}}
(e^{ct}\sqrt{\cL_j}).$$
By time integration (viz. we use Lemma \ref{SimpliCarre}),  we deduce that  
$$
\sqrt{\cL_j(t)} \leq e^{-ct}\sqrt{\cL_j(0)}+ C2^{-j(\frac d2+1)}\int_0^te^{-c(t-\tau)} c_j(\tau)
\|Z(\tau)\|_{\dot B^{\frac d2+1}_{2,1}}\|Z(\tau)\|_{\wt B^{\frac d2,\frac d2+1}_{2,1}}\,d\tau.$$
Hence, multiplying both sides by $2^{j(\frac d2+1)},$ then summing up on $j\geq0$ and using
the equivalence of the high frequency part of \eqref{def:cL} with the norm in $\dot B^{\frac d2+1}_{2,1},$ we end up with
$$\|Z(t)\|_{\dot B^{\frac d2+1}_{2,1}}^h\lesssim e^{-ct} \|Z_0\|_{\dot B^{\frac d2+1}_{2,1}}^h
+ \int_0^te^{-c(t-\tau)}\|Z(\tau)\|_{\dot B^{\frac d2+1}_{2,1}}\|Z(\tau)\|_{\wt B^{\frac d2,\frac d2+1}_{2,1}}\,d\tau.$$
Consequently, for all $t\geq0,$
$$\displaylines{
(1+c_0t)^{2\alpha_1}\|Z(t)\|_{\dot B^{\frac d2+1}_{2,1}}^h\lesssim (1+c_0t)^{2\alpha_1} e^{-ct} \|Z_0\|_{\dot B^{\frac d2+1}_{2,1}}^h
\hfill\cr\hfill+ \int_0^t \biggl(\frac{1+c_0t}{1+c_0\tau}\biggr)^{2\alpha_1}e^{-c(t-\tau)}\:
\bigl((1+c_0\tau)^{\alpha_1} \|Z(\tau)\|_{\wt B^{\frac d2,\frac d2+1}_{2,1}}\bigr)^2\,d\tau.}$$
Inequality \eqref{eq:decaybf}  ensures  that 
$$
\sup_{\tau\geq0} (1+c_0\tau)^{\alpha_1} \|Z(\tau)\|_{\wt B^{\frac d2,\frac d2+1}_{2,1}}
\leq C  \|Z_0\|_{\wt B^{\frac d2-1,\frac d2+1}_{2,1}}.$$
Furthermore,  one can find a constant $C_0$ depending only on $\alpha_1$ and $c_0$ such that
$$ \int_0^t\biggl(\frac{1+c_0t}{1+c_0\tau}\biggr)^{2\alpha_1}e^{-c(t-\tau)}\,d\tau\leq C_0,\qquad t \geq0.$$
Hence, in the end, we get
\begin{equation}\label{eq:decayhf}
\sup_{t\geq0} (1+c_0t)^{2\alpha_1} \|Z(t)\|_{\dot B^{\frac d2+1}_{2,1}}^h\leq C_0
 \|Z_0\|_{\wt B^{\frac d2-1,\frac d2+1}_{2,1}}.\end{equation}


\subsection{The decay of the damped mode} 

According to \eqref{eq:dtW} and to \eqref{GE},   the damped mode $W\triangleq -L_2^{-1}\d_tZ_2$
 satisfies a relation of the form
$$\d_tW+ L_2W\simeq Z\cdot\nabla^2Z +\nabla Z\cdot\nabla Z +(1+Z)\nabla W$$
and, according to \eqref{eq:WWW}, we have 
$$W\simeq Z_2+\nabla Z + Z\nabla Z.$$
Therefore, applying $\ddj$ to the above equation, taking the $L^2$ 
scalar product with $W_j\triangleq \ddj W$ and using Bernstein inequality in order to bound the last term, we get for all $j\in\Z,$
$$\displaylines{\frac12\frac d{dt}\|W_j\|_{L^2}^2 + c\|W_j\|_{L^2}^2\leq 
C\bigl(\|\ddj((1+Z)Z\cdot\nabla^2Z)\|_{L^2} +\|\ddj((1+Z)\nabla Z\cdot\nabla Z)\|_{L^2}\hfill\cr\hfill
+\|\ddj(\nabla Z_2\cdot Z)\|_{L^2}\bigr)\|W_j\|_{L^2}
+C2^j\|W_j\|_{L^2}^2.}
$$
Let us choose $j_0\in\Z$ such that $C2^{j_0}\leq c/2$ (so that the last term may be absorbed by the left-hand side).
Then,  using Lemma \ref{SimpliCarre},  multiplying both sides by $2^{j(\frac d2-1)},$ then summing up on $j\leq j_0,$
we end up with\footnote{Rigorously speaking the low frequencies 
that are here considered are lower than with our previous definition since it may happen 
that $j_0\leq 0.$ However, one may check that the high frequency decay estimate in \eqref{eq:decayhf}
still holds if we put the threshold at some $j_0\leq 0$: the argument 
we used  works if summing up on $j\geq j_0$ provided we change the `constants' accordingly.}
\begin{multline}\label{eq:Wdecay}\|W(t)\|_{\dot B^{\frac d2-1}_{2,1}}^\ell \leq e^{-ct}\|W_0\|_{\dot B^{\frac d2-1}_{2,1}}^\ell
\\+C\int_0^t e^{-c(t-\tau)}\bigl(\|(1+Z)Z\otimes\nabla^2Z\|_{\dot B^{\frac d2-1}_{2,1}}^\ell
+\|(1+Z)\nabla Z\otimes \nabla Z\|_{\dot B^{\frac d2-1}_{2,1}}^\ell+\|Z\otimes\nabla Z_2\|_{\dot B^{\frac d2-1}_{2,1}}^\ell\bigr)d\tau.\end{multline}
Since  $d\geq2,$ the product laws \eqref{eq:num} and \eqref{eq:num2} guarantee that 
$$\|(1+Z)Z\otimes\nabla^2Z\|_{\dot B^{\frac d2-1}_{2,1}}
+\|(1+Z)\nabla Z\otimes \nabla Z\|_{\dot B^{\frac d2-1}_{2,1}}\lesssim
\bigl(1+ \|Z\|_{\dot B^{\frac d2}_{2,1}}\bigr)  \|Z\|_{\dot B^{\frac d2}_{2,1}}
\|Z\|_{\dot B^{\frac d2+1}_{2,1}},$$
which, combined with \eqref{eq:decaylf} and the fact that $\|Z\|_{\dot B^{\frac d2}_{2,1}}$ is small  implies that 
$$\|(1+Z)Z\otimes\nabla^2Z\|_{\dot B^{\frac d2-1}_{2,1}}
+\|(1+Z)\nabla Z\otimes \nabla Z\|_{\dot B^{\frac d2-1}_{2,1}}\lesssim (1+c_0t)^{-2\alpha_1}\cL^2(0).$$
Similarly, we have 
$$\|(1+Z)Z\otimes\nabla Z_2\|_{\dot B^{\frac d2-1}_{2,1}}\lesssim  \|Z\|_{\dot B^{\frac d2}_{2,1}}^2
\lesssim  \|Z\|_{\wt B^{\frac d2-1,\frac d2+1}_{2,1}}^2\lesssim  (1+c_0t)^{-2\alpha_1}\cL^2(0).$$
Hence,  using  \eqref{eq:Wdecay} and arguing as in the previous paragraph, we end up with
\begin{equation}\label{eq:decayW}
\sup_{t\geq0} (1+c_0t)^{2\alpha_1} \|W(t)\|_{\dot B^{\frac d2-1}_{2,1}}^\ell\leq C_0
 \|Z_0\|_{\wt B^{\frac d2-1,\frac d2+1}_{2,1}}.\end{equation}
In other words, the decay rate for the low frequencies of the damped mode
in norm $\dot B^{\frac d2-1}_{2,1}$ is the same as that of the high frequencies of the whole solution. 
\medbreak 
Summing up the results of the previous paragraphs completes the
proof of Theorem \ref{Thm:decay}.


\section{On the strong relaxation limit}\label{s:relax}

This section is devoted to the study of a singular limit problem for the following 
class of  partially dissipative hyperbolic systems:   
\begin{equation}\label{eq:ZZep}\d_tZ^\eps +\sum_{k=1}^d A^k(Z^\eps)\d_k Z^\eps + \frac{BZ^\eps}\ep=0,\end{equation}
where, denoting  $\bar A^k_{\ell m}\triangleq A^k_{\ell m}(0)$
 and $\wt A^k_{\ell m}(Z)\triangleq A^k_{\ell m}(Z)-\bar  A^k_{\ell m},$ we assume that  for all $k\in\{1,\cdots,d\}$:
\begin{enumerate}
\item $\bar A_{11}^k=0,$ and $\wt A_{11}^k$ is linear with respect to $Z_2$ and independent 
of $Z_1$,
\item $\wt A_{12}^k$   and $\wt A_{21}^k$ are   linear with respect to $Z_1$ and independent 
of $Z_2$,
\item  $\wt A_{22}^k$  is  linear with respect to $Z,$
\item  Condition (SK) is satisfied by the pair $(A(\xi),B)$ with $A(\xi)$ defined in \eqref{eq:Axi}, at
every point $\xi\in\R^d.$
\end{enumerate}

The linearity assumption is here just for simplicity as well as the fact that there is no 
$0$-order nonlinear term. At the same time, assuming that 
$A_{12}^k$ and $A_{21}^k$ (resp. $A_{11}^k$) only depend on $Z_1$ (resp. $Z_2$)  is very helpful, if not essential. We shall see that it is satisfied by the compressible Euler equations written
in terms of the sound speed (see \eqref{eq:eulerc}). 
\smallbreak
We  want to study the so-called `strong  relaxation limit', that is
whether  the global solutions of \eqref{eq:ZZep} constructed  before  tend to satisfy 
some limit system when $\ep$ goes to $0.$ 

A  hasty analysis  suggests that the part of the solution that experiences direct dissipation, namely 
$Z_2^\ep$ with the notation of the previous sections,  tends to $0$ with a characteristic time of order $\ep$ and that, consequently, 
$Z_1^\ep$ tends to  be  time independent (since, for all $k\in\{1,\cdots,d\},$ 
we have $\bar A^k_{11}=0$ and $A^k_{11}$ is  independent of $Z_1$). 
To some extent this will prove to be true but, even for the simple case of the linearized one-dimensional compressible
Euler equations,   the situation is more complex than expected. 
Indeed, consider 
\begin{equation}\label{eq:lineareps}\left\{\begin{aligned}
&\d_ta+\d_xu=0,\\
&\d_tu+\d_xa+\ep^{-1}u=0.\end{aligned}\right.\end{equation}
In the Fourier space, this system translates into 
$$\frac d{dt}\begin{pmatrix} \wh a\\\wh u\end{pmatrix}
+\begin{pmatrix} 0&i\xi\\i\xi&\ep^{-1}\end{pmatrix}\begin{pmatrix} \wh a\\\wh u\end{pmatrix}=\begin{pmatrix}0\\0\end{pmatrix}\cdotp$$
\begin{itemize}
\item In low frequencies $|\xi|<(2\ep)^{-1},$ the matrix $A(\xi)$ of this system has 
the following two real eigenvalues:
$$\lambda^{\pm}(\xi)=\frac1{2\ep}\Bigl(1\pm \sqrt{1-(2\ep\xi)^2}\Bigr)\cdotp$$
For $\xi$ going to  $0,$ we observe that
$$\lambda^+(\xi)\sim \ep^{-1}\andf \lambda^-(\xi)\sim \ep\xi^2.$$
This means that one of the modes of the system is indeed damped with coefficient~$\ep^{-1}$  but that 
 the overall behavior of  solutions of the system  is  like for the inviscid limit (or for  the heat equation with vanishing diffusion).
\item In high frequencies $|\xi|>(2\ep)^{-1},$ the matrix $A(\xi)$ has 
the following two complex conjugated  eigenvalues:
$$\lambda^{\pm}(\xi)=\frac1{2\ep}\Bigl(1\pm i\sqrt{(2\ep\xi)^2-1}\Bigr)\cdotp$$
Clearly,   $\Re\lambda^\pm(\xi)=(2\ep)^{-1}$ and 
$\Im \lambda^\pm(\xi)\sim i\xi$ for $\xi\to\infty.$ Hence, there is indeed dissipation with characteristic time
$\ep$ for the high frequencies of the solution. 
\end{itemize}\smallbreak
The `low frequency regime' is expected  to dominate  when $\ep\to0,$ as it corresponds to $|\xi|\lesssim \ep^{-1}.$
Consequently, the overall behavior of System \eqref{eq:lineareps}   might be similar  to that of the heat
flow with diffusion $\ep,$ and one can wonder if  
 the high relaxation limit  is analogous to  the inviscid limit\footnote{This phenomenon that is well known in physics
is sometimes called \emph{overdamping}.}.
However, we have to keep in mind that the  low frequencies of the  `damped mode' (that here corresponds to 
the combination $u+\ep\d_xa$)  undergo  a much stronger dissipation.
This is of course an element that one has to take into consideration. 

Based on this simple example, it looks that in order to investigate the high relaxation limit,  it is suitable to use a functional 
framework that non only reflects the different   behavior of the low and high frequencies (with threshold
being located around $\ep^{-1}$) but also  emphasizes the better properties  of the damped mode.

\subsection{A `cheap' result of convergence}

Let us revert to the general class of Systems \eqref{eq:ZZep} supplemented with initial data $Z^\ep_0.$ 
 The structure  assumptions that we made at the beginning of the section enable us to apply Theorem \ref{Thmd2}. 
In this Subsection,  we shall take advantage of it and of elementary scaling considerations
so as to establish  that both $Z_1^\ep-Z_{1,0}^\ep$ and $Z_2^\ep$ converge strongly to $0$
for suitable norms. The reader may refer to the next subsection for 
a more accurate result. 
\medbreak
The starting observation is the following   change of time and space scale: 
\begin{equation}\label{eq:rescaling} \wt Z(t,x)\triangleq Z^\ep(\ep t,\ep x).\end{equation}
 Clearly,    $Z^\ep$ satisfies \eqref{eq:ZZep} if and only if $\wt Z$ satisfies \eqref{GE}. 
\medbreak
The following property of  homogeneous Besov norms is well known (see  \cite[Chap. 2]{BCD}):
\begin{equation}\label{eq:inv1}
\|z(\ep\cdot)\|_{\dot B^s_{2,1}} \simeq \ep^{s-d/2}\|z\|_{\dot B^s_{2,1}}.
\end{equation}
By adapting the proof therein, one can prove  that 
\begin{equation}\label{eq:inv2}
\|z(\ep\cdot)\|^\ell_{\dot B^s_{2,1}} \simeq \ep^{s-d/2}\|z\|^{\ell,\ep^{-1}}_{\dot B^s_{2,1}}\andf 
\|z(\ep\cdot)\|^h_{\dot B^s_{2,1}} \simeq \ep^{s-d/2}\|z\|^{h,\ep^{-1}}_{\dot B^s_{2,1}}
\end{equation}
where we have used the notation
\begin{equation}\label{eq:blh}\|z\|^{\ell,\alpha}_{\dot B^s_{2,1}}\triangleq \sum_{j\in\Z,\, 2^{j} < \alpha} 2^{js}\|\ddj z\|_{L^2}
\andf  \|z\|^{h,\alpha}_{\dot B^s_{2,1}}\triangleq \sum_{j\in\Z,\, 2^{j}\geq \alpha} 2^{js}\|\ddj z\|_{L^2}.
\end{equation}
Putting together  \eqref{eq:inv1}, \eqref{eq:inv2}, the change of unknowns \eqref{eq:rescaling}
and  Theorem \ref{Thmd2} readily gives the following  global existence result  that is valid for all $\ep>0.$
 \begin{thm}\label{Thm3}  There exists a positive constant $\alpha$  such that for all $\ep>0$ and data
 $Z_0^\ep$   satisfying 
 \begin{equation}\label{eq:smalldataep}
\cZ^\ep_0\triangleq  \|Z_0^\ep\|^{\ell,\ep^{-1}}_{\dot{B}^{\frac{d}{2}}_{2,1}} + \ep\|Z_0^\ep\|^{h,\ep^{-1}}_{\dot{B}^{\frac{d}{2}+1}_{2,1}} \leq \alpha,
\end{equation} System \eqref{eq:ZZep}  supplemented with initial data $Z_0^\ep$ 
 admits a unique global-in-time solution $Z^\ep$
satisfying the inequality 
\begin{equation}\label{eq:Yep}\cZ^\ep(t)\leq C\cZ^\ep_0\ \with\end{equation}
$$\displaylines{\cZ^\ep(t)\triangleq\|Z^\ep\|^{\ell,\ep^{-1}}_{L^\infty_t(\dot{B}^{\frac{d}{2}}_{2,1})}
+\ep\|Z^\ep\|^{h,\ep^{-1}}_{L_t^\infty(\dot{B}^{\frac{d}{2}+1}_{2,1})}+\ep\|Z_1^\ep\|^{\ell,\ep^{-1}}_{L^1_t(\dot{B}^{\frac{d}{2}+2}_{2,1})}
+\|Z_2^\ep\|^{\ell,\ep^{-1}}_{L^1_t(\dot{B}^{\frac{d}{2}+1}_{2,1})}
+\ep^{-1/2}\|Z_2^\ep\|^{\ell,\ep^{-1}}_{L^2_t(\dot{B}^{\frac{d}{2}}_{2,1})}\hfill\cr\hfill+\|Z^\ep\|^h_{L^1_t(\dot{B}^{\frac{d}{2}+1}_{2,1})}
+\| \d_t Z^\ep_2\|^\ell_{L^1_t(\dot{B}^{\frac{d}{2}}_{2,1})}.
}$$
\end{thm} 
The above theorem implies that $Z_1^\ep\to Z_1^\ep(0)$ and that $Z_2^\ep\to0$ when $\ep\to0.$
Indeed,  from the definition \eqref{eq:blh}, it is obvious that for all $\eta>0,$ $\beta\geq0$ and $s\in\R,$ we have
\begin{equation}\label{eq:comparaison}
\|z\|_{\dot B^{s+\beta}_{2,1}}^{\ell,\eta}\lesssim \eta^\beta
\|z\|_{\dot B^{s}_{2,1}}^{\ell,\eta}\andf
\|z\|_{\dot B^{s-\beta}_{2,1}}^{h,\eta}\lesssim \eta^{-\beta}
\|z\|_{\dot B^{s}_{2,1}}^{h,\eta}.\end{equation}
Hence,  using  \eqref{eq:Yep} and H\"older inequality yields
\begin{equation}\label{eq:Zinfty} \|Z^\ep\|_{L^2(\R_+;\dot B^{\frac d2}_{2,1})}^{h,\ep^{-1}} \leq C\ep^{1/2} \cZ^\ep_0.\end{equation}
Thanks to  \eqref{eq:smalldataep} and, again, to \eqref{eq:Yep}, this allows to get 
\begin{equation}\label{eq:ZL2}\|Z^\ep_2\|_{L^2(\R_+;\dot B^{\frac d2}_{2,1})} \leq C\alpha \ep^{1/2}.\end{equation}
In order to justify that $Z^\ep_1\to Z^\ep_1(0),$ one may bound $\d_tZ^\ep_1$ 
through  \eqref{eq:ZZep} remembering that 
the blocks $A_{11}^k$ are linear with respect to $Z_2^\ep$.
{}From the product law \eqref{eq:num},  and from \eqref{eq:Yep} and \eqref{eq:ZL2}, we get
$$\|\d_tZ_1^\ep\|_{L^2(\R_+;\dot B^{\frac d2-1}_{2,1})} \lesssim
\|Z_2^\ep\|_{L^2(\R_+;\dot B^{\frac d2}_{2,1})} \|Z^\ep\|_{L^\infty(\R_+;\dot B^{\frac d2}_{2,1})} \leq C\alpha^2\ep^{1/2},$$
 and thus 
 \begin{equation}\label{eq:Z1conv}\|Z^\ep_1(t)-Z^\ep_{1,0}\|_{\dot B^{\frac d2-1}_{2,1}}\leq C\alpha^2 (\ep t)^{1/2}\quad\hbox{for all }\ 
 t\geq0.\end{equation}
In conclusion, $Z_2^\ep$ tends to $0$ in $L^2(\R_+;\dot B^{\frac d2}_{2,1})$ with 
rate of convergence $\ep^{1/2},$ and $Z_1^\ep-Z_{1,0}^\ep$ 
converges to $0$ in $L^\infty([0,T];\dot B^{\frac d2-1}_{2,1})$ with rate $(\ep T)^{1/2},$ for all $T>0.$


\subsection{Connections with  porous media-like equations}

In order to exhibit  richer  dynamics in the asymptotics $\ep\to0,$  one may perform  the following `diffusive'   rescaling: 
\begin{equation}\label{eq:relaxZ}
(\wt Z_1^\ep,\wt Z_2^\ep) (\tau, x) = (Z_1^\ep,\ep^{-1} Z_2^\ep)(\ep^{-1}\tau,x).\end{equation}
Dropping  the exponents $\ep$ for better readability, we get the following system for $(\wt Z_1,\wt Z_2)$: 
\begin{equation} \left\{ \begin{aligned} &\partial_\tau\wt Z_1 
+ \sum_{k=1}^d \wt A^k_{11}(\wt Z_2)\d_k\wt Z_1
+ \sum_{k=1}^d\bigl(\bar A_{12}^k\!+\!\wt A_{12}^k(\wt Z_1)\bigr)\partial_{k}\wt Z_2=0,\\  
&\ep^2\partial_\tau\wt Z_2 +\ep \sum_{k=1}^d \bigl(\bar A_{22}^k\!+\!\wt A_{22}^k(\wt Z_1,\ep\wt Z_2)\bigr)\partial_{k}\wt Z_2
+\sum_{k=1}^d \bigl(\bar A_{21}^k\!+\!\wt A_{21}^k(\wt Z_1)\bigr)\partial_{k}\wt Z_1
+{L_2} \wt Z_2=0. \end{aligned} \right. \label{GEep}
\end{equation} 
{}From the second line, one can expect  
\begin{equation}\label{eq:Wtilde}\wt W\triangleq \wt Z_2 + L_2^{-1}\Bigl(\sum_{k=1}^d 
\bigl(\bar A^k_{21}+\wt A_{21}^k(\wt Z_1)\bigr) \d_k\wt Z_1\Bigr)\longrightarrow0.\end{equation} 
 In order to find out
what could be the limit system for $\wt Z_1,$ let us systematically express $\wt Z_2$ in terms of $\wt W$ and $\wt Z_1$ by means of \eqref{eq:Wtilde}  in  the first line of \eqref{GEep}. 
We get
$$\displaylines{
\d_\tau\wt Z_1 + \sum_k\bigl(\bar A_{12}^k\!+\!\wt A_{12}^k(\wt Z_1)\bigr)\d_k\wt W
+\wt A^k_{11}(\wt W)\d_k \wt Z_1\hfill\cr\hfill
+\sum_{k,\ell} \bigl(\bar A_{12}^k\!+\!\wt A_{12}^k(\wt Z_1)\bigr)L_2^{-1}\d_k\bigl(\bigl(\bar A_{21}^\ell\!+\!\wt A_{21}^\ell(\wt Z_1)\bigr)\d_\ell\wt Z_1\bigr)\hfill\cr\hfill
+\sum_{k,\ell} \wt A^k_{11}\bigl(L_2^{-1}\bigl(\bar A_{21}^\ell\!+\!\wt A_{21}^\ell(\wt Z_1)\bigr)\d_\ell\wt Z_1\bigr)\d_k\wt Z_1=0.}$$
Introducing the following second order operator:
\begin{equation}\label{def:A}
\cA\triangleq \sum_{k,\ell} \bar A^k_{12}L_2^{-1}\bar A^\ell_{2,1}\d_k\d_\ell,
\end{equation}
the above equation may be rewritten:
\begin{multline}\label{eq:wtZ1}
\d_\tau\wt Z_1 + \cA \wt Z_1+ Q_1(\wt Z_1,\nabla^2\wt Z_1) + Q_2(\nabla\wt Z_1,\nabla\wt Z_1)\\+T_1(\wt Z_1,\nabla\wt Z_1,\nabla\wt Z_1)+T_2(\wt Z_1,\wt Z_1,\nabla^2\wt Z_1)=S
\end{multline}
where, $Q_1,$ $Q_2$ (resp. $T_1,$ $T_2$) are bilinear (resp. trilinear) expressions that may be computed
in terms of the coefficients of the matrices $\wt A_{11}^k,$ $\wt A_{12}^k$ and of $L_2,$ and
 \begin{equation}\label{def:S}S\triangleq -\sum_{k=1}^d \bigl(\bar A^k_{12}+ \wt A_{12}^k(\wt Z_1)\bigr)\d_k\wt W-
\sum_{k=1}^d \wt A_{11}^k(\wt W)\d_k\wt Z_1.\end{equation}
Consequently, if \eqref{eq:Wtilde} is true, then we expect 
$\wt Z_1$ to tend to $\wt N$  with  $\wt N$ satisfying  
\begin{equation}\label{eq:limit}\d_\tau\wt N + \cA \wt N + Q_1(\wt N,\nabla^2\wt N) + Q_2(\nabla \wt N,\nabla \wt N)
+T_1(\wt N,\nabla \wt N,\nabla \wt N)+T_2(\wt N, \wt N,\nabla^2\wt N)=0.\end{equation}
Note that, as a consequence of Lemma \ref{l:elliptic} in Appendix, and since we assumed both Condition (SK)
and that $\bar A^k_{11}=0$ for all $k\in\{1,\cdots,d\},$  \eqref{eq:limit}
is a quasilinear (scalar) parabolic equation.
\smallbreak
Before justifying the above heuristics in the general case, let  us again consider the compressible Euler equations, that is
\begin{equation}\label{eq:eulerisen}
 \left\{\begin{aligned}&\d_t\varrho^\ep +\div(\varrho^\ep v^\ep)=0\quad &\hbox{in}\quad \R_+\times\R^d,\\
&\d_t(\varrho^\ep v^\ep)+\div(\varrho^\ep v^\ep\otimes v^\ep) +\nabla(P(\varrho^\ep)) +\ep^{-1}\varrho^\ep v^\ep=0 &\hbox{in}\quad \R_+\times\R^d.
\end{aligned}\right.\end{equation}
Under the isentropic assumption 
\begin{equation}\label{Pression1}P(z)=Az^\gamma\with\gamma>1\andf A>0,\end{equation}
the above system enters in the class \eqref{eq:ZZep}  if reformulated  in terms of $(c^\ep,\varrho^\ep)$, where
 \begin{equation}\label{eq:crho}
 c^\ep\triangleq \frac{(\gamma A)^\frac{1}{2}}{\wt\gamma}(\varrho^\ep)^{\wt\gamma}\andf\wt\gamma\triangleq \frac{\gamma-1}2\cdotp\end{equation}
 Indeed,  we get: 
\begin{equation} \left\{ \begin{aligned} &\partial_tc^\ep+v^\ep\cdot\nabla c^\ep+\tilde{\gamma} c^\ep\div v^\ep=0,\\ 
&\partial_tv^\ep+v^\ep\cdot\nabla v^\ep+\tilde{\gamma}c^\ep\nabla c^\ep+\ep^{-1} v^\ep=0. \end{aligned} \right.\label{CED4}
\end{equation} 
So, if we set $\bar c\triangleq \frac{(\gamma A)^\frac{1}{2}}{\wt\gamma}(\bar\varrho)^{\wt\gamma},$
then Conditions (1) to (4) below \eqref{eq:ZZep} are satisfied with $Z_1^\ep= c^\ep-\bar c$ and $Z_2^\ep=v^\ep.$
\medbreak
Now, performing the diffusive rescaling:
\begin{equation}\label{eq:diffusive}(\varrho^\ep,v^\ep)(t,x)=(\wt \varrho^\ep, \ep \wt v^\ep)(\ep t,x),\end{equation}
System \eqref{eq:eulerisen} becomes 
\begin{equation}\label{eq:eulerisenep}
 \left\{\begin{aligned}&\d_\tau\wt\varrho^\ep +\div(\wt\varrho^\ep \wt v^\ep)=0\quad &\hbox{in}\quad \R_+\times\R^d,\\
&\ep^2\d_\tau(\wt\varrho^\ep \wt v^\ep)+\ep\div(\wt\varrho^\ep \wt v^\ep\otimes \wt v^\ep) +\nabla(P(\wt\varrho^\ep)) +\wt\varrho^\ep \wt v^\ep=0 &\hbox{in}\quad \R_+\times\R^d.
\end{aligned}\right.\end{equation}
In light of the second equation, it is expected that 
$$\nabla(P(\wt\varrho^\ep)) +\wt\varrho^\ep \wt v^\ep\to 0\quad\hbox{when}\quad
\ep\to0,$$
and thus  that $\wt\varrho^\ep$  converges to some solution $\wt N$ of the porous media equation: 
 \begin{equation}\label{eq:PM}
 \d_\tau \wt N-\Delta(P(\wt N))=0.
 \end{equation}
 The general result we shall prove for  Systems \eqref{eq:ZZep}  reads as follows
 for  the particular case of the isentropic Euler equations\footnote{A statement in the same spirit, but allowing for Besov spaces constructed
on $L^p$ may be found  in \cite{CBD3}.}: 
 \begin{thm}\label{thm:relaxeuler}  Consider the Euler equations 
 with relaxation \eqref{eq:eulerisen}  in $\R^d$ (with $d\geq1$)  with pressure law \eqref{Pression1}
 and initial data $(\varrho_0^\ep,v_0^\ep)$ such that 
  $(\varrho^\ep-\bar\varrho) \in \dot B^{\frac d2}_{2,1}\cap\dot B^{\frac d2+1}_{2,1}$ and 
  $v_0^\ep\in   \dot B^{\frac d2}_{2,1}\cap\dot B^{\frac d2+1}_{2,1}.$
  There exists $\alpha>0$ independent of $\ep$ such that if  
  \begin{equation}
\|(\varrho_0^\ep-\bar\varrho, v_0^\ep)\|^{\ell,\varepsilon^{-1}}_{\dot{{B}}^{\frac{d}{2}}_{2,1}}
+\varepsilon\|(\varrho_0^\ep-\bar\varrho, v_0^\ep)\|^{h,\varepsilon^{-1}}_{\dot{{B}}^{\frac{d}{2}+1}_{2,1}}
 \leq \alpha
  \end{equation}
  then \eqref{eq:eulerisen} supplemented with  $(\varrho_0^\ep,v_0^\ep)$  has 
  a unique solution $(\varrho^\ep,v^\ep)$ with  $(\varrho^\ep-\bar\varrho,v^\ep)\in\cC_b(\R_+;\dot B^{\frac d2}_{2,1}\cap \dot B^{\frac d2+1}_{2,1})$
  satisfying in addition 
\begin{multline}\label{eq:conveps}\|(\varrho^\ep-\bar\varrho,v^\ep)\|^{\ell,\ep^{-1}}_{L^\infty(\R_+;\dot{B}^{\frac{d}{2}}_{2,1})}
+\ep\|(\varrho^\ep-\bar\varrho,v^\ep)\|^{h,\ep^{-1}}_{L^\infty(\R_+;\dot{B}^{\frac{d}{2}+1}_{2,1})}
+\ep^{1/2}\|\varrho^\ep-\bar\varrho\|_{L^2(\R_+;\dot{B}^{\frac{d}{2}+1}_{2,1})}\\
+\|v^\ep\|_{L^1(\R_+;\dot{B}^{\frac{d}{2}+1}_{2,1})}
+\ep^{-1/2}\|v^\ep\|^{\ell,\ep^{-1}}_{L^2(\R_+;\dot{B}^{\frac{d}{2}}_{2,1})}
+\| \ep^{-1}v^\ep+(\varrho^{\ep})^{-1}\nabla(P(\varrho^\ep))\|^{\ell,\ep^{-1}}_{L^1(\R_+;\dot{B}^{\frac{d}{2}}_{2,1})}\leq C\alpha.
\end{multline}
Furthermore, for any $\wt N_0$ in $\dot B^{\frac d2}_{2,1}$ such that
$\|\wt N_0\|_{\dot B^{\frac d2}_{2,1}}\leq\alpha,$
Equation \eqref{eq:PM} has a unique solution 
  $\wt N$ in the space
$\cC_b(\R_+;\dot B^{\frac d2}_{2,1})\cap L^1(\R_+;\dot B^{\frac d2+2}_{2,1})$
satisfying for all $t\geq0,$
$$\|\wt N(t)\|_{\dot B^{\frac d2}_{2,1}} +\int_0^t\|\wt N\|_{\dot B^{\frac d2+2}_{2,1}}\,d\tau\leq C\|\wt N_0\|_{\dot B^{\frac d2}_{2,1}}.$$
Finally, if one denotes by $(\wt\varrho^\ep,\wt v^\ep)$ the rescaled
solution of the Euler equations defined through \eqref{eq:diffusive} and 
assumes in addition that 
  $$\|\wt N_0-\wt \varrho_{0}\|_{\dot B^{\frac d2-1}_{2,1}}\leq C\ep,$$ then we have
 \begin{equation}\label{eq:wtconv}\biggl\|\wt{v}^\ep+\frac{\nabla(P(\wt\varrho^\ep))}{\widetilde{\varrho^\ep}}
\biggr\|_{L^1(\R_+;\dot B^{\frac d2}_{2,1})}
  +\|\wt N-\wt \varrho^\ep\|_{L^\infty(\R_+;\dot B^{\frac d2-1}_{2,1})}+
   \|\wt N-\wt \varrho^\ep\|_{L^1(\R_+;\dot B^{\frac d2+1}_{2,1})}\leq C\ep.\end{equation}
 \end{thm}
\begin{proof} Let us assume for a while that $\ep=1$
so that one can readily take advantage of Theorem \ref{thm:euler}.
As a first, we want to translate  Theorem \ref{thm:euler}  in terms of $\varrho,$  where $c$ and $\varrho$ 
(resp.  $\bar c$ and $\bar\varrho$) 
are interrelated through \eqref{eq:crho}.

On the one hand,  Inequality \eqref{eq:cv}, the property of interpolation in Besov spaces 
and H\"older inequality with respect to the time variable imply that
$$\|c-\bar c\|_{L^2(\R_+;\dot B^{\frac d2+1}_{2,1})} \leq C\|(c_0-\bar c,v_0)\|_{\wt B^{\frac d2,\frac d2+1}_{2,1}}.$$
On the other hand, using the fact that the composition inequality \eqref{eq:compo} is actually valid for all positive
Besov exponents (see e.g. \cite{BCD}[Chap. 2]), we may write that 
$$\|c-\bar c\|_{\dot B^{\frac d2+\alpha}_{2,1}}\approx \|\varrho-\bar\varrho\|_{\dot B^{\frac d2+\alpha}_{2,1}}
\quad\hbox{for }\ \alpha=0,1.$$
Finally, we note that $\d_t v= -v -\varrho^{-1}\nabla(P(\varrho))-v\cdot \nabla v$ and that
$$\|v\cdot\nabla v\|_{L^1(\R_+;\dot B^{\frac d2}_{2,1})} \leq C \|v\|_{L^\infty(\R_+;\dot B^{\frac d2}_{2,1})} 
\|\nabla v\|_{L^1(\R_+;\dot B^{\frac d2}_{2,1})}.$$
Therefore, the last term of $\d_tv$ may be `omitted' in  Inequality \eqref{eq:cv}, and we get 
\begin{multline}\label{eq:esteuler1}
\|(\varrho-\bar \varrho,v)\|_{L^\infty(\R_+;\wt{B}^{\frac{d}{2},\frac{d}{2}+1}_{2,1})}
+\|\varrho-\bar\varrho\|^\ell_{L^2(\R_+;\dot{B}^{\frac{d}{2}+1}_{2,1})}
+\|v\|_{L^1(\R_+;\dot{B}^{\frac{d}{2}+1}_{2,1})}
+\|v\|^\ell_{L^2(\R_+;\dot{B}^{\frac{d}{2}}_{2,1})}\\
+\|v+\varrho^{-1}\nabla(P(\varrho))\|_{L^1(\R_+;\dot{B}^{\frac{d}{2}}_{2,1})}
 \leq C\|(\varrho_0-\bar \varrho,v_0)\|_{\wt B^{\frac d2,\frac d2+1}_{2,1}}.\end{multline}
 Now, for general $\ep>0,$ performing the rescaling \eqref{eq:rescaling} and 
 remembering the equivalences \eqref{eq:inv1} and \eqref{eq:inv2} 
 gives the first part of Theorem \ref{thm:relaxeuler}. 

After performing the diffusive rescaling \eqref{eq:diffusive}, the rescaled pair $(\wt\varrho^\ep,\wt v^\ep)$ satisfies 
$$\partial_t\wt\varrho^\varepsilon-\Delta(P(\wt\varrho^\varepsilon))=-\div(\widetilde{\varrho}^\varepsilon\wt W^\varepsilon)
\with \wt W^\ep\triangleq \wt{v}^\varepsilon+\frac{\nabla(P(\wt\varrho^\varepsilon))}{\widetilde{\varrho}^\varepsilon}\cdotp$$
Thanks to \eqref{eq:inv1}, the bound for the last term in \eqref{eq:esteuler1} translates into 
  \begin{equation}\label{eq:weakW}
  \|\wt W^\ep\| _{L^1(\R_+;\dot B^{\frac d2}_{2,1})}\leq C\alpha\ep,
  \end{equation}
  which completes the proof of \eqref{eq:conveps}.
\smallbreak  
Proving that  $\wt\varrho^\varepsilon$  tends to  some  solution $\wt N$  of \eqref{eq:PM}
may be done exactly as in the  general case presented below in   Theorem \ref{th:relax}. 
\end{proof}

Let us finally turn to the study of the strong relaxation limit in the general case. 
The main result we shall get reads as follows:
\begin{thm}\label{th:relax}   Assume that\footnote{The one-dimensional case is tractable either under 
  specific assumptions on the nonlinearities that are satisfied by the Euler equations, 
 or  in a  slightly different functional framework.  More details may be found in \cite{CBD3}.}  $d\geq2$ and consider a system of type 
 \eqref{GEep}  for some $\ep>0.$ Let  the structure hypotheses listed below \eqref{eq:ZZep}
  be  in force.  There exists a positive constant $\alpha$ 
 (independent of $\ep$) such that  for any initial data $\wt N_0\in\dot B^{\frac d2}_{2,1}$  for \eqref{eq:limit} 
  and  $\wt Z^\ep_0 \in \dot B^{\frac d2}_{2,1}\cap\dot B^{\frac d2+1}_{2,1}$  for \eqref{GEep} satisfying
  \begin{eqnarray}\label{eq:smallN0}&\|\wt N_0\|_{\dot B^{\frac d2}_{2,1}} \leq \alpha, \\\label{eq:smallZ0}
&\cZ_0^\ep\triangleq\|\wt Z^\ep_{0,1}\|^{\ell,\varepsilon^{-1}}_{\dot{{B}}^{\frac{d}{2}}_{2,1}}+\varepsilon\|\wt  Z^\ep_{0,2}\|^{\ell,\varepsilon^{-1}}_{\dot{{B}}^{\frac{d}{2}}_{2,1}}+ \varepsilon\|\wt  Z^\ep_{0,1}\|^{h,\varepsilon^{-1}}_{\dot{{B}}^{\frac{d}{2}+1}_{2,1}}
+\varepsilon^2\|\wt Z^\ep_{0,2}\|^{h,\varepsilon^{-1}}_{\dot{{B}}^{\frac{d}{2}+1}_{2,1}} \leq \alpha,\end{eqnarray}
  System \eqref{eq:limit} admits a unique solution 
   $\wt N$ in the space
$$\cC_b(\R_+;\dot B^{\frac d2}_{2,1})\cap L^1(\R_+;\dot B^{\frac d2+2}_{2,1}),$$
satisfying for all $t\geq0,$
\begin{equation}\label{eq:wtN}\|\wt N(t)\|_{\dot B^{\frac d2}_{2,1}} +\int_0^t\|\wt N\|_{\dot B^{\frac d2+2}_{2,1}}\,d\tau\leq C\|\wt N_0\|_{\dot B^{\frac d2}_{2,1}},\end{equation}
and System \eqref{GEep}  has a unique  global-in-time solution  $\wt Z^\ep$ 
in $\cC(\R_+;\dot B^{\frac d2}_{2,1}\cap\dot B^{\frac d2+1}_{2,1})$ 
such that
\begin{multline}\label{eq:wtZep}
\|\wt Z^\ep_1\|^{\ell,\varepsilon^{-1}}_{L^\infty(\R_+;\dot{{B}}^{\frac{d}{2}}_{2,1})}
+\varepsilon\|\wt Z^\ep_2\|^{\ell,\varepsilon^{-1}}_{L^\infty(\R_+;\dot{{B}}^{\frac{d}{2}}_{2,1})}
+\varepsilon\|\wt Z^\ep_1\|^{h,\varepsilon^{-1}}_{L^\infty(\R_+;\dot{{B}}^{\frac{d}{2}+1}_{2,1})}\\
+\varepsilon^2\|\wt Z^\ep_2\|^{h,\varepsilon^{-1}}_{L^\infty(\R_+;\dot{{B}}^{\frac{d}{2}+1}_{2,1})}
+\|\wt Z^\ep_1\|^{\ell,\varepsilon^{-1}}_{L^1(\R_+;\dot{{B}}^{\frac{d}{2}+2}_{2,1})}
+\varepsilon^{-1}\|\wt Z^\ep_1\|^{h,\varepsilon^{-1}}_{L^1(\R_+;\dot{{B}}^{\frac{d}{2}+1}_{2,1})}
+\|\wt Z^\ep_2\|_{L^1(\R_+;\dot{{B}}^{\frac{d}{2}+1}_{2,1})}
\\+\|\wt Z^\ep_2\|^{\ell,\varepsilon^{-1}}_{L^2(\R_+;\dot{{B}}^{\frac{d}{2}}_{2,1})}
+\varepsilon^{-1}\|\wt W^\ep\|_{L^1(\R_+;\dot{{B}}^{\frac{d}{2}}_{2,1})}\leq C\cZ_0^\ep,\end{multline}
where  $\wt W^\ep$ has been defined in \eqref{eq:Wtilde}.
\medbreak
  If, in addition,   $$\|\wt N_0-\wt Z^\ep_{1,0}\|_{\dot B^{\frac d2-1}_{2,1}}\leq C\ep,$$ then we have
 \begin{equation}\label{eq:wtZ2}  \|\wt N-\wt Z^\ep_1\|_{L^\infty(\R_+;\dot B^{\frac d2-1}_{2,1})}+
   \|\wt N-\wt Z^\ep_1\|_{L^1(\R_+;\dot B^{\frac d2+1}_{2,1})}\leq C\ep.\end{equation}
 \end{thm}
 \begin{proof}
That  \eqref{GEep} supplemented with initial data $\wt Z_0$ admits a unique global solution 
satisfying \eqref{eq:wtZep} follows from Theorem \ref{Thmd2} after suitable rescaling. 
Indeed, if we make  the  change of unknowns:
\begin{equation}\label{eq:scaleZ}(\wt Z_1,\wt Z_2)(\tau,x)=(\check{Z_1},\dfrac{\check{Z_2}}{\varepsilon})\left(\dfrac{\tau}{\varepsilon^2},\dfrac{x}{\varepsilon}\right), \end{equation}
then we discover that   $\wt Z$ satisfies \eqref{GEep} if and only if  $\check Z$ is a solution to    \eqref{GE}. 
Then, taking advantage of the equivalence of norms pointed out in \eqref{eq:inv1} and \eqref{eq:inv2}
gives the desired global existence  result and \eqref{eq:wtZep} up to the last term since 
defining  the damped mode  as in \eqref{eq:WWW} would lead to the function
\begin{equation}\label{eq:Z2L2}
\wt Z_2+L_2^{-1}\sum_{k=1}^d (\bar A_{21}^k + \wt A^k_{21}(\wt Z_1))\d_k\wt Z_1
+\ep L_2^{-1}\sum_{k=1}^d\bigl(\bar A_{22}^k + \wt A^k_{22}(\wt Z_1,\ep\wt Z_2)\bigr)\d_k\wt Z_2.\end{equation}
However, combining Inequality \eqref{eq:wtZep} (without the last term of course)
with  \eqref{eq:comparaison} ensures that 
$$\|\wt Z_1\|_{L^\infty(\R_+;\dot B^{\frac d2}_{2,1})} + \ep\|\wt Z_2\|_{L^\infty(\R_+;\dot B^{\frac d2}_{2,1})} +
\|\nabla \wt Z_2\|_{L^1(\R_+;\dot B^{\frac d2}_{2,1})}\lesssim \cZ_0^\ep.$$
Hence  the last term of \eqref{eq:Z2L2} is of order $\ep$ in $L^1(\R_+;\dot B^{\frac d2}_{2,1}),$
and $\wt W$ does satisfy \eqref{eq:wtZep}.

In order to prove the convergence of $\wt Z_1$ to $\wt N,$ 
let us first verify that  $S$  defined in  \eqref{def:S} is of order $\ep$ in $L^1(\R_+;\dot B^{\frac d2-1}_{2,1}).$ 
As $d\geq2,$  it is just a matter of  taking  advantage of the product law 
 \eqref{eq:num} to get
$$\|S\|_{\dot B^{\frac d2-1}_{2,1}} \lesssim  
\bigl(1+\|\wt Z_1\|_{\dot B^{\frac d2}_{2,1}}\bigr)\|\nabla \wt W\|_{\dot B^{\frac d2-1}_{2,1}} 
+  \|\wt W\|_{\dot B^{\frac d2}_{2,1}}\|\nabla \wt Z_1\|_{\dot B^{\frac d2-1}_{2,1}}.$$
Hence, 
$$\|S\|_{L^1(\R_+;\dot B^{\frac d2-1}_{2,1})}\lesssim \bigl(1+\|\wt Z_1\|_{L^\infty(\R_+;\dot B^{\frac d2}_{2,1})}\bigr)
\|\wt W\|_{L^1(\R_+;\dot B^{\frac d2}_{2,1})}$$
and using \eqref{eq:wtZep}   and the smallness of the initial data  thus yields
\begin{equation}\label{eq:S}\|S\|_{L^1(\R_+;\dot B^{\frac d2-1}_{2,1})}\leq C\ep\alpha.\end{equation}
Let us next briefly justify that any data $\wt N_0$ satisfying
\eqref{eq:smallN0} gives rise to a unique global solution  $\wt N$ of \eqref{eq:limit} in 
$\cC_b(\R_+;\dot B^{\frac d2}_{2,1})\cap L^1(\R_+;\dot B^{\frac d2+2}_{2,1})$
satisfying \eqref{eq:wtN}.
In fact, since the operator $\cA$ is strongly elliptic, 
the  parabolic estimates in Besov spaces with last index $1$
recalled in  Proposition \ref{p:heat} ensure that any smooth enough  global solution
$\wt N$ satisfies   for all $t\geq0,$
$$\displaylines{\|\wt N(t)\|_{\dot B^{\frac d2}_{2,1}} +\int_0^t\|\wt N\|_{\dot B^{\frac d2+2}_{2,1}}\,d\tau
\lesssim \|\wt N_0\|_{\dot B^{\frac d2}_{2,1}} +\int_0^t
\Bigl(\| Q_1(\wt N,\nabla^2\wt N)\|_{\dot B^{\frac d2}_{2,1}} \hfill\cr\hfill
+ \|Q_2(\nabla \wt N,\nabla \wt N)\|_{\dot B^{\frac d2}_{2,1}}+\|T_1(\wt N,\nabla \wt N,\nabla \wt N)\|_{\dot B^{\frac d2}_{2,1}}+\|T_2(\wt N,\wt N,\nabla^2 \wt N)\|_{\dot B^{\frac d2}_{2,1}}\Bigr)\,d\tau.}$$
Using the stability of the space $\dot B^{\frac d2}_{2,1}$ by product and
an obvious interpolation inequality, the nonlinear
terms may be estimated as follows: 
$$
\begin{aligned}
\| Q_1(\wt N,\nabla^2\wt N)\|_{\dot B^{\frac d2}_{2,1}}&\lesssim 
\|\wt N\|_{\dot B^{\frac d2}_{2,1}}\|\nabla^2\wt N\|_{\dot B^{\frac d2}_{2,1}}
\lesssim \|\wt N\|_{\dot B^{\frac d2}_{2,1}}\|\wt N\|_{\dot B^{\frac d2+2}_{2,1}},\\
\| Q_2(\nabla \wt N,\nabla \wt N)\|_{\dot B^{\frac d2}_{2,1}}&\lesssim \|\nabla \wt N\|_{\dot B^{\frac d2}_{2,1}}^2
\lesssim \|\wt N\|_{\dot B^{\frac d2}_{2,1}}\|\wt N\|_{\dot B^{\frac d2+2}_{2,1}},\\
\|T_1(\wt N,\nabla \wt N,\nabla \wt N)\|_{\dot B^{\frac d2}_{2,1}}&\lesssim 
 \|\wt N\|_{\dot B^{\frac d2}_{2,1}} \|\nabla \wt N\|_{\dot B^{\frac d2}_{2,1}}^2
\lesssim \|\wt N\|_{\dot B^{\frac d2}_{2,1}}^2\|\wt N\|_{\dot B^{\frac d2+2}_{2,1}},\\
\|T_2(\wt N, \wt N,\nabla^2\wt N)\|_{\dot B^{\frac d2}_{2,1}}&\lesssim 
 \|\wt N\|_{\dot B^{\frac d2}_{2,1}}^2 \|\nabla^2\wt N\|_{\dot B^{\frac d2}_{2,1}}
\lesssim \|\wt N\|_{\dot B^{\frac d2}_{2,1}}^2\|\wt N\|_{\dot B^{\frac d2+2}_{2,1}}.
\end{aligned}$$
Hence, we have for all $t\geq0,$ 
$$\|\wt N(t)\|_{\dot B^{\frac d2}_{2,1}} +\int_0^t\|\wt N\|_{\dot B^{\frac d2+2}_{2,1}}\,d\tau
\lesssim \|\wt N_0\|_{\dot B^{\frac d2}_{2,1}} +\int_0^t\bigl(1+ \|\wt N\|_{\dot B^{\frac d2}_{2,1}}\bigr) \|\wt N\|_{\dot B^{\frac d2}_{2,1}}
\|\wt N\|_{\dot B^{\frac d2+2}_{2,1}}\,d\tau.
$$
Clearly, if the solution is small enough (which is  ensured
if the initial data is small) then, the last term of the right-hand side may be absorbed by the
left-hand side, leading to  Inequality  \eqref{eq:wtN}.
The above formal inequalities combined with   a suitable contracting mapping argument 
(in the spirit of the one that is used e.g. for solving the incompressible Navier-Stokes equations, see details in \cite[Chap. 5]{BCD}), 
allow to conclude to the global existence of a solution to \eqref{eq:limit}, fulfilling the desired properties. 
\smallbreak
To finish  the proof of Theorem \ref{th:relax},  we just have to compare  $\wt Z_1$    with $\wt N.$ 
To proceed, let us subtract  \eqref{eq:limit} from \eqref{eq:wtZ1}. We get 
the following equation for $\dN\triangleq \wt Z_1-\wt N$: 
$$\displaylines{
\d_\tau\dN+\cA\dN=S-Q_1(\wt Z_1,\nabla^2\dN)- Q_1(\dN,\nabla^2\wt N)
-Q_2(\nabla\wt Z_1,\nabla\dN)-Q_2(\nabla\dN,\nabla \wt N)\hfill\cr\hfill
-T_1(\dN,\nabla \wt Z_1,\nabla \wt Z_1)-T_1(\wt N,\nabla\dN,\nabla \wt Z_1)-T_1(\wt N,\nabla \wt N,\nabla\dN)\hfill\cr\hfill
-T_2(\dN,\wt Z_1,\nabla^2\wt Z_1)-T_2(\wt N,\dN,\nabla^2 \wt Z_1)-T_2(\wt N,\wt N,\nabla^2\dN).}$$
Hence, by virtue of  Proposition \ref{p:heat}, we have for all $t\geq0,$ 
$$\displaylines{\|\dN\|_{L_t^\infty(\dot B^{\frac d2-1}_{2,1})\cap L_t^1(\dot B^{\frac d2+1}_{2,1})} \lesssim 
\|\dN(0)\|_{\dot B^{\frac d2-1}_{2,1}}+\|S\|_{L_t^1(\dot B^{\frac d2-1}_{2,1})}+
\|Q_1(\wt Z_1,\nabla^2\dN)\|_{L_t^1(\dot B^{\frac d2-1}_{2,1})}\hfill\cr\hfill
+\|Q_1(\dN,\nabla^2\wt N)\|_{L_t^1(\dot B^{\frac d2-1}_{2,1})}
+\|Q_2(\nabla\wt Z_1,\nabla\dN)\|_{L_t^1(\dot B^{\frac d2-1}_{2,1})}
+\|Q_2(\nabla\dN,\nabla \wt N)\|_{L_t^1(\dot B^{\frac d2-1}_{2,1})}\hfill\cr\hfill
+ \|T_1(\dN,\nabla \wt Z_1,\nabla \wt Z_1)\|_{L_t^1(\dot B^{\frac d2-1}_{2,1})}
+\|T_1(\wt N,\nabla\dN,\nabla \wt Z_1)\|_{L_t^1(\dot B^{\frac d2-1}_{2,1})}
+\|T_1(\wt N,\nabla \wt N,\nabla\dN)\|_{L_t^1(\dot B^{\frac d2-1}_{2,1})}
\hfill\cr\hfill
+ \|T_2(\dN,\wt Z_1,\nabla^2\wt Z_1)\|_{L_t^1(\dot B^{\frac d2-1}_{2,1})}
+\|T_2(\wt N,\dN,\nabla^2\wt Z_1)\|_{L_t^1(\dot B^{\frac d2-1}_{2,1})}
+\|T_2(\wt N,\wt N,\nabla^2\dN)\|_{L_t^1(\dot B^{\frac d2-1}_{2,1})}.}$$
So, using  \eqref{eq:num}, the stability of  $\dot B^{\frac d2}_{2,1}$ by product,   \eqref{eq:S}, 
\eqref{eq:smallN0} and \eqref{eq:wtN}, we find that  
$$\begin{aligned}
\|\dN\|_{L_t^\infty(\dot B^{\frac d2-1}_{2,1})\cap L_t^1(\dot B^{\frac d2+1}_{2,1})} &\lesssim 
 \|\dN(0)\|_{\dot B^{\frac d2-1}_{2,1}}+\|\wt S\|_{L_t^1(\dot B^{\frac d2-1}_{2,1})}\\
  +\bigl((1&+\|(\wt N,\wt Z_1)\|_{L_t^\infty(\dot B^{\frac d2}_{2,1})})
  \|\wt N\|_{L_t^\infty(\dot B^{\frac d2}_{2,1})}
  +  \|\wt Z_1\|_{L_t^\infty(\dot B^{\frac d2}_{2,1})}\bigr)\|\nabla\dN\|_{L_t^1(\dot B^{\frac d2}_{2,1})}
\\&\qquad+\|(\wt N,\wt Z_1)\|_{L_t^\infty(\dot B^{\frac d2}_{2,1})}\|\nabla^2\wt Z_1\|_{L_t^2(\dot B^{\frac d2}_{2,1})}
\|\dN\|_{L_t^2(\dot B^{\frac d2}_{2,1})}\\
&\qquad+ \bigl(\|\nabla^2\wt N\|_{L_t^1(\dot B^{\frac d2}_{2,1})} +\|\nabla\wt Z_1\|_{L_t^2(\dot B^{\frac d2}_{2,1})}^2\bigr)
\|\dN\|_{L_t^\infty(\dot B^{\frac d2-1}_{2,1})}\\
&\lesssim \|\dN(0)\|_{\dot B^{\frac d2-1}_{2,1}}+\alpha\ep
+(\alpha+\alpha^2)\|\dN\|_{L_t^1(\dot B^{\frac d2+1}_{2,1})\cap L_t^\infty(\dot B^{\frac d2-1}_{2,1})}.
\end{aligned}
$$ 
Hence, as $\alpha$ is small enough, we get:
\begin{equation}\label{eq:convdN}
\|\dN\|_{L_t^\infty(\dot B^{\frac d2-1}_{2,1})\cap L_t^1(\dot B^{\frac d2+1}_{2,1})}\lesssim
 \|\dN(0)\|_{\dot B^{\frac d2-1}_{2,1}}+\alpha \ep  \quad\hbox{for all }\ t\geq0,\end{equation}
 which completes the proof of the theorem. 
 \end{proof}
 
 We end this section with a few remarks.  The first one is that, for small $\ep,$ 
   it is natural to modify the definition in \eqref{eq:Wtilde} so as to have
 a damped mode that is  expressed in terms of $\wt Z_2$ and $\wt N.$ If we set
   \begin{equation}\label{def:checkW}\check W\triangleq  \wt Z_2 + L_2^{-1}\sum_{k=1}^d\bigl(\bar A^k_{21}+\wt A^k_{21}(\wt N)\bigr)\d_k \wt N,
    \end{equation}
  then we have 
$$
 \wt W-\check W= L_2^{-1}\sum_{k=1}^d\bigl(A_{21}^k(\wt N)\d_k\dN+ \wt A_{21}^k(\dN)\d_k\wt Z_1\bigr)\cdotp
$$
 In order to bound the right-hand side, one can observe that 
 $$\begin{aligned}
 \|A_{21}^k(\wt N)\d_k\dN\|_{L^1(\R_+;\dot B^{\frac d2}_{2,1})}
 &\lesssim \|\wt N\|_{L^\infty(\R_+;\dot B^{\frac d2}_{2,1})} \|\dN\|_{L^1(\R_+;\dot B^{\frac d2+1}_{2,1})},\\
  \|\wt A_{21}^k(\dN)\d_k\wt Z_1     \|_{L^1(\R_+;\dot B^{\frac d2}_{2,1})}
 &\lesssim \|\dN\|_{L^2(\R_+;\dot B^{\frac d2}_{2,1})} \|\d_k\wt Z_1\|_{L^2(\R_+;\dot B^{\frac d2}_{2,1})}.\end{aligned}$$
 Hence, taking advantage of  Inequalities \eqref{eq:wtN}, \eqref{eq:wtZep} and \eqref{eq:wtZ2}, 
 and of interpolation inequalities yields
 \begin{equation}\label{eq:WWcc}
 \| \wt W-\check W\|_{L^1(\R_+;\dot B^{\frac d2}_{2,1})}\leq C\alpha\bigl( \|\dN(0)\|_{\dot B^{\frac d2-1}_{2,1}}+\alpha \ep\bigr),
 \end{equation}
 which guarantees that $\check W$ satisfies  \eqref{eq:Wtilde}.
 \medbreak
 Note also that, since  $\wt Z_1$ is bounded in $\cC_b(\R_+;\dot B^{\frac d2}_{2,1})$ independently of $\ep$,
  using \eqref{eq:wtN} and \eqref{eq:wtZ2}, and  interpolating, one obtains
$$\|\wt N-\wt Z_1\|_{L^\infty(\R_+;\dot B^{\frac d2-\beta}_{2,1})}\leq C\ep^\beta,\qquad \beta\in(0,1).$$
Finally, observe that if we introduce the following rescaled solution of the limit system:
 $$
 N^\ep(t,x)\triangleq \wt N^\ep(\ep t,x),$$ 
 then  combining \eqref{eq:wtZ2} with the definition of $\wt Z_1^\ep$ in \eqref{eq:relaxZ} yields
 $$ Z_1^\ep= N^\ep + \cO(\ep)\quad\hbox{in}\quad L^\infty(\R_+;\dot B^{\frac d2-1}_{2,1}) $$
 which is, indeed, a more accurate information than what we had in Theorem  \ref{Thm3} or in \eqref{eq:Z1conv}.
 Similarly, putting  \eqref{eq:wtZ2} and \eqref{eq:WWcc} together yields the following expansion:
 $$ Z_2^\ep(t,x)=-\ep L_2^{-1}\sum_{k=1}^d\bigl(\bar A^k_{21}+\wt A^k_{21}(\wt N(\ep t,x))\bigr)\d_k\wt N(\ep t,x)
 + \cO(\ep)\quad \hbox{in}\quad L^1(\R_+;\dot B^{\frac d2}_{2,1}).$$


\appendix
\section{}

The following classical result (see the proof in e.g. the Appendix of \cite{CBD1})
has been used a number of times in this text. 
\begin{lem}\label{SimpliCarre}
Let $X : [0,T]\to \mathbb{R}_+$ be a continuous function such that $X^2$ is differentiable. Assume that there exists 
 a constant $c\geq 0$ and  a measurable function $A : [0,T]\to \mathbb{R}_+$ 
such that 
 $$\frac{1}{2}\frac{d}{dt}X^2+cX^2\leq AX\quad\hbox{a.e.  on }\ [0,T].$$ 
 Then, for all $t\in[0,T],$ we have
$$X(t)+c\int_0^tX(\tau)\,d\tau\leq X_0+\int_0^tA(\tau)\,d\tau.$$
\end{lem}
We frequently took advantage of the fact that  applying derivatives or, more 
generally, Fourier multipliers on spectrally localized functions
is almost equivalent to multiplying by some constant depending only 
on the Fourier multiplier and on the spectral support. 

This is illustrated by  the classical Bernstein inequality  that states (see e.g. \cite[Chap. 2]{BCD}) that for any $R>0$ there exists
a constant $C$ such that for any $\lambda>0$ and any function $u:\R^d\to\R$  with Fourier transform $\wh u$
supported in the ball $B(0,R\lambda),$ we have
\begin{equation}\label{eq:bernstein1} 
\|\d^\alpha u\|_{L^q}\leq C^{1+|\alpha|} \lambda^{|\alpha|+d(\frac1p-\frac1q)}\|u\|_{L^p},\qquad
\alpha\in\N^d,\quad 1\leq p\leq q\leq\infty.
\end{equation}
The reverse Bernstein inequality asserts that, under the stronger assumption that $\wh u$ is 
supported in the annulus $\{x\in\R^d\,:\, r\lambda\leq |x|\leq R\lambda\}$ for some $0<r<R,$  then we have in addition, 
\begin{equation}\label{eq:bernstein2} 
\|u\|_{L^p}\leq C \lambda^{-1}\|\nabla u\|_{L^p},\qquad 1\leq p\leq\infty.
\end{equation}
A slight modification of the proof of \eqref{eq:bernstein1} allows to extend the result
to any smooth homogeneous multiplier : denoting by $M$ a smooth function on $\R^d\setminus\{0\}$
with homogeneity $\gamma,$ there exists a constant $C$ 
such that for any $\lambda>0$ and any function $u:\R^d\to\R$  with Fourier transform $\wh u$
supported   in the annulus $\{x\in\R^d\,:\, r\lambda\leq |x|\leq R\lambda\},$  we have
\begin{equation}\label{eq:bernstein3} 
\|M(D) u\|_{L^q}\leq C \lambda^{\gamma+d(\frac1p-\frac1q)}\|u\|_{L^p},\qquad
\alpha\in\N^d,\ 1\leq p\leq q\leq\infty.
\end{equation}

In the last section, in order to study the convergence to the limit system,  we used maximal regularity estimates in Besov spaces with last index $1$ for parabolic system. These 
estimates are well known for the heat equation (see e.g. \cite[Chap. 2]{BCD}). 
Below, we extend them   to 
semi-groups  generated by  strictly elliptic homogeneous multipliers in the following meaning:
we consider functions 
 $\cA\in \cC^\infty(\R^d\setminus\{0\};\cM_n(\C))$  homogeneous of degree $\gamma,$
 such that the matrix $\cA(\xi)$ is Hermitian and satisfies for some $c>0$:  
 \begin{equation}\label{eq:ellipticA}
  \bigl(\cA(\xi)z\cdot z\bigr)\geq c|\xi|^\gamma |z|^2,\qquad \xi\in\R^d\setminus\{0\},\quad z\in\C^n.\end{equation}


\begin{prop}\label{p:heat} Let $u\in \cC(\R_+;\cS')$ satisfy
\begin{equation} \label{eq:heateq1}
\left\{\begin{array}{ll} \d_tz + \cA(D)z= f\quad &\hbox{on }\ \R_+\times\R^d,\qquad\\[1ex]
z|_{t=0}=z_0\quad &\hbox{on }\ \R^d.\end{array}\right.\end{equation}
 Then, for any $p\in[1,\infty]$ and $s\in\R$ the following inequality holds true for all $t>0:$
\begin{equation}\label{eq:maxreg1}
\|z(t)\|_{\dot B^s_{p,1}}+\int_0^t\|z\|_{\dot B^{s+\gamma}_{p,1}}\,d\tau
\leq C\biggl(\|z_0\|_{\dot B^s_{p,1}}+\int_0^t\|f\|_{\dot B^s_{p,1}}\,d\tau\biggr)\cdotp
\end{equation}
\end{prop}
\begin{proof}
If $z$ satisfies \eqref{eq:heateq1}  then for any $j\in\Z,$ we have
$$
\d_tz_j + \cA(D)z_j= f_j \with z_j\triangleq \ddj z\andf f_j\triangleq \ddj f.
$$
Hence, according to Duhamel's formula, 
\begin{equation}\label{eq:zj}z_j(t)=e^{-t\cA(D)} z_{0,j}+\int_0^te^{-(t-\tau)\cA(D)}f_j(\tau)\,d\tau.\end{equation}
Let us provisionally  admit the following lemma:
\begin{lem}\label{l:sg}  
 There exist two constants $c_0$ and $C$ such that the following inequality 
holds for all $j\in\Z,$ $t\geq0$ and $p\in[1,\infty]$:
\begin{equation}\label{eq:ddjA}\|\ddj e^{t\cA(D)}z\|_{L^p}\leq C e^{-c_02^{\gamma j}t} \|\ddj z\|_{L^p}.
\end{equation}
\end{lem}
Then, plugging \eqref{eq:ddjA} in \eqref{eq:zj}  yields for all $j\in\Z,$
\begin{equation}\label{eq:heat2}
\|z_j(t)\|_{L^p}\lesssim e^{-c_02^{\gamma j}t}\|z_{0,j}\|_{L^p}
+\int_0^t e^{-c_02^{\gamma j}(t-\tau)}\|f_j(\tau)\|_{L^p}\,d\tau.
\end{equation}
Hence,  taking  the supremum norm on $[0,t]$ (resp. integrating on $[0,t]$), we get for all $j\in\Z,$
$$\|z_j\|_{L^\infty(0,t;L^p)} + 2^{\gamma j} \|z_j\|_{L^1(0,t;L^p)} 
\lesssim  \Bigl(1-e^{-c_02^{\gamma j}t}\Bigr)\Bigl(\|z_{0,j}\|_{L^p}+ \|f_j\|_{L^1(0,t;L^p)}\Bigr)\cdotp
$$
Just bounding the prefactor in the right-hand side by $1,$ multiplying the two sides by $2^{js}$
and summing up on $\Z$ yields the claimed inequality. 
 \end{proof}

\begin{proof}[Proof of Lemma \ref{l:sg}] Thanks to  the homogeneity of $\cA,$ using 
a suitable change of  variables reduces the proof to the case $j=0.$  Indeed, if we set $\zeta(x)\triangleq z(2^{-j}x),$ then we have $\dot\Delta_0 \zeta(2^jx)=\ddj z(x)$
and $$e^{-2^{\gamma j}\lambda\cA(D)}\dot\Delta_0\zeta(2^jx)=e^{-\lambda\cA(D)}\ddj z(x),\qquad \lambda\geq0.$$
Then, consider a  function~$\phi$
in~$\cD(\R^d\setminus\{0\})$ with value~$1$ on a neighborhood of
the support of $\varphi$ and write that
$$\begin{aligned}
e^{-t\cA(D)}\dot\Delta_0\zeta&=\cF^{-1}\left(\phi e^{-\lambda\cA(\cdot)}
\wh{\dot\Delta_0\zeta}\right)\\[1ex]
 & =  g_\lambda \star \dot\Delta_0u\with
 g_\lambda(x) \triangleq
(2\pi)^{-d} \Int_{\R^d}  e^{i\,x\cdot\xi}\phi(\xi)
e^{-\lambda\cA(\xi)}d\xi.\end{aligned}$$
If it is true that 
\begin{equation}\label{eq:g} 
\|g_\lambda\|_{L^1}\leq Ce^{-c_0\lambda}
\end{equation}
then  using that the convolution product maps  $L^1\star L^p$ to $L^p$ implies that    
$$
\|e^{-\lambda\cA(D)}\dot\Delta_0\zeta\|_{L^p}
\leq \|g_\lambda\|_{L^1}\|\dot\Delta_0\zeta\|_{L^p}
\leq C e^{-c_0\lambda}\|\dot\Delta_0\zeta\|_{L^p},
$$
and we get \eqref{eq:ddjA} after reverting to the original variables.
\medbreak
In order to prove \eqref{eq:g},   it suffices to establish that
$$|g_\lambda(x)|\leq C(1+|x|^2)^{-d} e^{-c_0\lambda},\qquad x\in\R^d,\quad\lambda>0.$$
Now,  integrating by parts, we get
$$(2\pi)^dg_\lambda(x)= (1+|x|^2)^{-d} h_\lambda(x)\with 
h_\lambda(x)\triangleq \int_{\R^d}   e^{i\,x\cdot\xi}\phi(\xi) (\Id-\Delta_\xi)^d\Bigl(e^{-\lambda\cA(\xi)}\Bigr)d\xi.$$
Of course, the integral may be restricted to $\Supp\phi$ which is a compact subset of $\R^d\setminus\{0\}.$
Owing to \eqref{eq:ellipticA}, on this subset, there exists a positive constant $c_0$ such that 
all the real parts of the eigenvalues of $\cA(\xi)$ are bounded from below by $2c_0.$
Now,  since the differential of the exponential map may be computed by the formula
$$D\, e^X\cdotp  H = \int_0^1 e^{(1-\tau)X} H e^{\tau X}\,d\tau,\qquad H\in\cM_n(\R),$$
 the chain rule entails that
 $$ D_\xi\Bigl(e^{-\lambda\cA(\xi)}\Bigr)\cdot H = -\lambda\int_0^1 e^{-\lambda(1-\tau)\cA(\xi)}
 \bigl(D_\xi\cA(\xi)\cdot H\bigr)e^{-\lambda\tau\cA(\xi)}\,d\tau.$$
 Hence,  there exist two  constants $C$ and $C'$ such that 
 $$ \biggl|D_\xi\Bigl(e^{-\lambda\cA(\xi)}\Bigr)\biggr|\leq C'\lambda e^{-2c_0\lambda}
 \leq C e^{-c_0\lambda},\qquad \lambda>0,\quad \xi\in\Supp\phi.$$
 By induction, one can get similar estimates for higher order derivatives
 of $\xi\mapsto e^{-\lambda\cA(\xi)},$ which eventually yields 
 $$|h_\lambda(x)|\leq C e^{-c_0\lambda},\qquad x\in\R^d,\quad\lambda>0,$$
 and completes the proof. 
\end{proof}

\begin{remark} In the case $p=2$ 
one can work out a shorter proof, based  on the Fourier-Plancherel theorem. However, it is interesting to point out 
that the very same result holds for any value of $p$ in $[1,\infty]$ \emph{including $1$ and $\infty$}, 
and with a constant independent of $p.$
\end{remark}

The following lemma ensures that in the setting of System \eqref{GE}, 
if both  Condition (SK) and $A^k_{11}(\bar V)=0$ for all $k\in\{1,\cdots,d\}$ are satisfied,  
then the second order differential operator $\cA$ defined in \eqref{def:A} 
is indeed strictly elliptic in the sense of Proposition \ref{p:heat}, with  $\gamma=2.$
\begin{lem}\label{l:elliptic}   Consider two  $n\times n$ Hermitian matrices $A$ and $B$ such that
\begin{equation}\label{eq:StructA}
A=\begin{pmatrix} 0 &A_{12}\\ A_{21} &A_{22}\end{pmatrix}\andf
B=\begin{pmatrix}0& 0\\ 0&B_{22}\end{pmatrix}\end{equation}
with $A_{12}\in\cM_{n_1,n_2}(\C),$ $A_{21}\in\cM_{n_2,n_1}(\C),$ $A_{22}\in\cM_{n_2,n_2}(\C)$
and $B_{22}\in\cM_{n_2,n_2}(\C).$ Suppose also that $B_{22}$ is positive.
Then, $B_{22}$ is invertible and  the following two properties are equivalent:
\begin{enumerate}
\item The matrix $A_{12}B_{22}^{-1}A_{21}$ is a $n_1\times n_1$ positive matrix.\smallbreak
\item Condition (SK) holds true (that is, the four equivalent conditions of Lemma \ref{l:SK} are satisfied).
\end{enumerate}
\end{lem}
\begin{proof}
The invertibility of $B_{22}$ being  obvious, let us first
assume that $A_{12}B_{22}^{-1}A_{21}$ is  positive. Then, the rank of $A_{21}$ must be equal to $n_1$
and so does the rank of $B_{22} A_{21}.$ 
Now, we observe that
$$BA =\begin{pmatrix} 0&0\\ B_{22} A_{21}& B_{22}A_{22}\end{pmatrix}\cdotp$$
Hence, the rank of $\begin{pmatrix} B\\ BA\end{pmatrix}$ is equal to $n_1+n_2=n,$ and Condition (SK)
is thus satisfied.
\medbreak
Conversely, if $A$ has the special structure \eqref{eq:StructA} then an easy induction reveals that 
the bottom left  block    of  any positive power $k$ 
of $A$  ends with $A_{21}.$  The same property clearly holds for $BA^k$ that thus looks like 
$$
BA^k=\begin{pmatrix}0&0\\ B_{22} C_k A_{21}& D_k\end{pmatrix}\quad\hbox{for some }\ 
C_k, D_k\in\cM_{n_2}(\C).$$
Now, since $B_{22}$ is invertible, we have for all $k\in\N,$
$${\rm rank}\,(B_{22} C_k A_{21})\leq {\rm rank}\,(A_{21})= {\rm rank}\,(B_{22} A_{21}).$$
As the block at the bottom left of $BA$ is equal to  $B_{22} A_{21},$
one can conclude that, under assumption \eqref{eq:StructA} we automatically have 
$$
{\rm rank}\begin{pmatrix} B\\ BA\end{pmatrix}= {\rm rank}\left(\begin{smallmatrix} B\\BA\\\dots\\ BA^{n-1}\end{smallmatrix}\right)\cdotp$$
Hence, if we assume in addition that Condition (SK) is satisfied, then we must have 
$ {\rm rank}(B_{22}A_{21})= n_1,$ and thus  $ {\rm rank}(A_{21})= n_1,$ too. 
Now, since ${}^t\!\bar A_{12}=A_{21},$ we have for all $z\in\C^{n_1},$
$$A_{12}B_{22}^{-1}A_{21}z\cdot z = B_{22}^{-1}A_{21}z\cdot A_{21}z.$$
As $B_{22}^{-1}$ is  positive, the right-hand side
is nonnegative and  vanishes if and only if
$A_{21}z=0$ and thus if and only if $z=0$ since  $ {\rm rank}(A_{21})= n_1.$
Hence  $A_{12}B_{22}^{-1}A_{21}$ is positive, which completes the proof.  
\end{proof}



\begin{thebibliography}{99}

\bibitem{Alinhac} S. Alinhac: Temps de vie des solutions  r\'eguli\`eres
des \'equations d'Euler compressibles axisym\'etriques en dimension deux, 
{\em Inventiones Mathematicae}, {\bf 111}, 627--670, 1993. 
 
 \bibitem{BCD} 
H. Bahouri, J.-Y. Chemin and  R. Danchin: {\it Fourier Analysis and Nonlinear Partial Differential Equations,} Grundlehren der mathematischen Wissenschaften, {\bf 343}, 
Springer, 2011.

\bibitem{BZ} K. Beauchard and E. Zuazua: Large time asymptotics for partially dissipative hyperbolic systems, 
{\em Arch. Rational Mech. Anal,} {\bf 199}, 177--227, 2011.

\bibitem{BS} S. Benzoni-Gavage and D. Serre: {\em Multi-dimensional Hyperbolic Partial Differential Equations : First-order Systems and Applications.} Oxford Science Publications, New-York, 2007.

\bibitem{BHN} S. Bianchini, B. Hanouzet  and R. Natalini:  Asymptotic behavior of smooth solutions for partially dissipative hyperbolic systems with a convex entropy, {\em Comm. Pure and Appl. Math.}, {\bf 60}, 1559--1622, 2007.

\bibitem{BN} R. Bianchini and R. Natalini:  Nonresonant bilinear forms for partially dissipative hyperbolic systems violating the Shizuta-Kawashima condition, {\em J. Evol. Equ.} {\bf 22}(3), Paper No. 63, 2022. 

\bibitem{Brenner} P. Brenner: The Cauchy problem for symmetric hyperbolic systems in 
$L^p,$ {\em  Math. Scand.}, {\bf  19},  27--37, 1966.

\bibitem{BCBT}  C. Burtea, T. Crin-Barat and J. Tan:  Relaxation limit for a damped one-velocity Baer-Nunziato model to a  Kappila model, arXiv:2109.07746.

\bibitem{ChL} J.-Y. Chemin and N. Lerner: Flot de champs de vecteurs
non lipschitziens et \'equations de Navier-Stokes, {\em Journal of Differential Equations}, {\bf 121},  314--328, (1995).

\bibitem{CBD1} T. Crin-Barat and R. Danchin: Partially dissipative one-dimensional hyperbolic systems in the critical regularity setting, and applications,   \emph{Pure and Applied Analysis},  {\bf 4}(1), 85--125, 2022. 

\bibitem{CBD2} T. Crin-Barat and R. Danchin: Partially dissipative hyperbolic systems in the critical regularity setting: the multi-dimensional case, \emph{Journal de Math. Pures et App.}, {\bf 165},  1--41, 2022. 

\bibitem{CBD3} T. Crin-Barat and R. Danchin: Global existence for partially dissipative hyperbolic systems in the $L^p$ 
framework, and relaxation limit,   \emph{Mathematische Annalen}, to appear.

\bibitem{D1} R. Danchin: Global existence in critical spaces for compressible Navier-Stokes Equations, 
{\em Inventiones Mathematicae,} {\bf 141,}  579--614, 2000.

\bibitem{D2} R. Danchin: {\em  Fourier analysis methods for compressible flows} 
in Topics on compressible Navier-Stokes equations, Edited by D. Bresch, 
Panoramas et Synth\`eses, {\bf 50}, 43--106, 2016.

\bibitem{D3} R. Danchin: {\em Fourier analysis methods for the compressible Navier-Stokes equations}
in  Y. Giga and A. Novotn\'y (eds). Handbook of Mathematical Analysis in Mechanics of Viscous Fluids. Springer, Cham, 2018.

\bibitem{DD} R. Danchin and B. Ducomet: Existence of strong solutions with critical regularity to a polytropic model for radiating flows, {\em Annali di Matematica}, {\bf 196}, 107--153, 2017.



\bibitem{DX1} R. Danchin and J. Xu: Optimal time-decay estimates for the compressible Navier-Stokes equations 
in the critical $L^p$ framework, {\em Arch. Rational Mech. Anal.}, {\bf 224},  53--90, 2017. 

\bibitem{GM} V. Giovangigli and L. Matuszewski: Structure of Entropies in Dissipative Multicomponent Fluids, {\em Kin. Rel. Models},  {\bf 6},  373--406, 2013. 

\bibitem{GY} V. Giovangigli and W.-A. Yong: Volume Viscosity and Internal Energy Relaxation : Symmetrization and Chapman-Enskog Expansion, {\em Kin. Rel. Models}, {\bf 8},  79--116, 2015. 

\bibitem{god} S. K. Godunov: An interesting class of quasi-linear systems, {\it Dokl. Akad. Nauk SSSR}, {\bf 139}, 521--523, 1961 (Russian).


\bibitem{GW} Y. Guo and Y. Wang: Decay of dissipative equations and negative Sobolev spaces, 
{\em Comm. Part. Differ. Equ.}, {\bf 37}, 2165--2208, 2012.

 
\bibitem{Ho} L. H\"ormander: Hypoelliptic second order differential equations, {\em  Acta Math.}, {\bf 119},
147--171, 1967.

\bibitem{JR} S. Junca and M.  Rascle: Strong relaxation of the isothermal {E}uler system to the heat
equation, {\em Z. angew. Math. Phys.},
 {\bf 53},  239--264, 2002.

\bibitem{Ka84}  S. Kawashima: {\em Systems of a hyperbolic-parabolic composite type, with applications to the equations of magneto-hydrodynamics.} Thesis, Kyoto University, 1984.

\bibitem{KY1} S. Kawashima and W.-A. Yong:  Dissipative structure and entropy for hyperbolic systems of balance laws,
{\em Arch. Rational Mech. Anal.,} {\bf 174}, 345--364, 2004.

\bibitem{KY2} S. Kawashima and W.-A. Yong: Decay estimates for hyperbolic balance laws, 
{\em Journal for Analysis and its Applications,} {\bf 28}, 1--33, 2009.

\bibitem{LC}  C. Lin and J.-F. Coulombel. The strong relaxation limit of the multi-dimensional Euler equations, 
{\em Nonlinear Diff. Eq. and Applications}, {\bf 20}, 447--461, 2013.
 
 \bibitem{MN} A. Matsumura and T. Nishida: The initial value problem for the equations of motion of viscous and heat-conductive gases, 
 {\em J. Math. Kyoto Univ.}, {\bf 20}, 67--104, 1980.
 
 \bibitem{Nash} J. Nash: Continuity of solutions of parabolic and elliptic equations.
{\em Amer. J. Math.}, {\bf 80},  931--954, 1958. 
 
 \bibitem{PW} Y.-J. Peng and V. Wasiolek: Uniform global existence and parabolic limit 
 for partially dissipative hyperbolic systems, 
 {\em Journal of Differential Equations}, {\bf 260}, 7059--7092, 2016.  
 
 \bibitem{QW} P. Qu and Y. Wang: Global classical solutions to partially dissipative hyperbolic systems violating the Kawashima condition,
 {\em Journal de Math\'ematiques Pures et Appliqu\'ees,} {\bf 109}, 93--146, 2018.
 
 \bibitem{Serre} D. Serre: \emph{Systems of conservation laws with dissipation}, 
 Unpublished lecture notes for a course given at the SISSA (Trieste), 2007.  
 
\bibitem{SK}   S. Shizuta and S. Kawashima: Systems of equations of hyperbolic-parabolic type with applications to the
discrete Boltzmann equation, {\em Hokkaido Math. J.,} {\bf 14}, 249--275, 1985.

\bibitem{S86} T. Sideris: Formation of singularities in three-dimensional compressible fluids,
{\em Comm. in Math. Phys.}, {\bf 101}, 475--485, 1985. 

\bibitem{STW} T. Sideris, B. Thomases  and D. Wang: Long time behavior of solutions to the 3D compressible Euler equations with damping, {\em Comm. Partial Differential Equations,} {\bf 28}, 795--816, 2003.

\bibitem{Villani} C. Villani: {\em Hypocoercivity},  Mem. Am. Math. Soc., {\bf 202}, 2010.


\bibitem{XX}  Z. Xin and J. Xu: Optimal decay for the compressible Navier-Stokes equations without additional smallness assumptions,  {\em Journal of Differential Equations,}  {\bf 274}, 543--575, 2021.

\bibitem{XK1} J. Xu and S. Kawashima: Diffusive relaxation limit of classical solutions to the damped compressible Euler equations, {\em Journal of Differential Equations}, {\bf 256}, 771--796, 2014.

\bibitem{XK2}  J. Xu and S. Kawashima: Global classical solutions for partially dissipative hyperbolic system of balance laws, 
{\em Arch. Rational Mech. Anal,} {\bf 211}, 513--553, 2014.

\bibitem{XK3} J. Xu and S. Kawashima: The optimal decay estimates on the framework of Besov spaces for generally dissipative systems, {\em Arch. Rational Mech. Anal,} {\bf 218}, 275--315, 2015.

\bibitem{XW} J. Xu and Z. Wang:  Relaxation limit in Besov spaces for compressible {E}uler equations,
{\em Journal de Mathématiques Pures et Appliquées}, {\bf 99}, 43--61, 2013.

\bibitem{Yong} W.-A Yong: Entropy and global existence for hyperbolic balance laws, 
{\em Arch. Rational Mech. Anal,} {\bf 172}, 247--266, 2004. 

\bibitem{Zeng}  Y. Zeng: Gas dynamics in thermal nonequilibrium and general hyperbolic 
systems with relaxation, {\em  Arch. Ration. Mech. Anal.}, {\bf 150}(3), 225--279, 1999.

\end{thebibliography}
\end{document}